\documentclass[11pt,leqno]{amsart}
\usepackage{amsthm,amsfonts,amssymb,amsmath,oldgerm}
\numberwithin{equation}{section}
\usepackage{graphicx}
\usepackage{fullpage}

\renewcommand\d{\partial}

\def\eps{\varepsilon }

\renewcommand\d{\partial}

\newcommand\R{\mathbb R}
\newcommand\C{\mathbb C}

\def\eps{\varepsilon}

\newcommand\br{\begin{remark}}
\newcommand\er{\end{remark}}
\newcommand\bp{\begin{pmatrix}}
\newcommand\ep{\end{pmatrix}}
\newcommand{\be}{\begin{equation}}
\newcommand{\ee}{\end{equation}}
\newcommand\ba{\begin{equation}\begin{aligned}}
\newcommand\ea{\end{aligned}\end{equation}}

\newcommand{\qut}[1]{\textquotedblleft #1\textquotedblright}

\newcommand{\kUnstableLowerLeft}{0.199910210210210}

\newcommand{\kUnstableUpperRight}{0.999999999997}

\newcommand{\bap}{\begin{app}}
\newcommand{\eap}{\end{app}}
\newcommand{\begs}{\begin{exams}}
\newcommand{\eegs}{\end{exams}}
\newcommand{\beg}{\begin{example}}
\newcommand{\eeg}{\end{exaplem}}
\newcommand{\bpr}{\begin{proposition}}
\newcommand{\epr}{\end{proposition}}
\newcommand{\bt}{\begin{theorem}}
\newcommand{\et}{\end{theorem}}
\newcommand{\bc}{\begin{corollary}}
\newcommand{\ec}{\end{corollary}}
\newcommand{\bl}{\begin{lemma}}
\newcommand{\el}{\end{lemma}}
\newcommand{\bd}{\begin{definition}}
\newcommand{\ed}{\end{definition}}
\newcommand{\brs}{\begin{remarks}}
\newcommand{\ers}{\end{remarks}}

\newcommand{\Id}{{\rm Id }}
\newcommand{\I}{{\rm I}}

\newcommand{\diag}{{\rm diag }}

\newtheorem{theorem}{Theorem}[section]
\newtheorem{proposition}[theorem]{Proposition}
\newtheorem{corollary}[theorem]{Corollary}
\newtheorem{lemma}[theorem]{Lemma}

\theoremstyle{remark}
\newtheorem{remark}[theorem]{Remark}
\theoremstyle{definition}
\newtheorem{definition}[theorem]{Definition}

\newtheorem{example}[theorem]{Example}

\newcommand{\RM}{\mathbb{R}}

\newcommand{\CM}{\mathbb{C}}

\newcommand{\cn}{\operatorname{cn}}

\newcommand{\ac}{\text{\rm{ac}}}

\newcommand{\beq}{\begin{equation}}
\newcommand{\eeq}{\end{equation}}

\newcommand{\atau}{\alpha_{\tau}}
\newcommand{\au}{\alpha_{u}}
\newcommand{\ak}{\alpha_{k}}
\renewcommand{\ac}{\alpha_{c}}

\title{Stability of viscous St. Venant roll-waves: from
 onset to infinite-Froude number limit}

\author{Blake Barker}
\address {Brown University, Providence, RI 02912}
\email{Blake\underline{ }Barker@brown.edu}
\thanks{Research of B.B. was partially supported
under NSF grants no. DMS-0300487, DMS-0801745, and CNS-0723054.}

\author{Mathew A. Johnson}
\address{University of Kansas, Lawrence, KS 66045}
\email{matjohn@ku.edu}
\thanks{Research of M.J. was partially supported under NSF grant no. DMS-1211183.
 }
 
\author{Pascal Noble}
\address{Institut de Math\'ematiques de Toulouse, Toulouse, France}
\email{pascal.noble@math.univ-toulouse.fr}
\thanks{Research of P.N. was partially supported by the French ANR Project no.
ANR-09-JCJC-0103-01.}

\author{L.Miguel Rodrigues}
\address{Universit\'e Lyon 1, Institut Camille Jordan, INRIA \'EP Kaliffe, Villeurbanne, France}
\email{rodrigues@math.univ-lyon1.fr}
\thanks{Research of M.R. was partially supported by the ANR project
BoND ANR-13-BS01-0009-01.}

\author{Kevin Zumbrun}
\address{Indiana University, Bloomington, IN 47405}
\email{kzumbrun@indiana.edu} 
\thanks{Research of K.Z. was partially supported
under NSF grant no. DMS-0300487.}
\begin{document}

\begin{abstract}
We study the spectral stability of roll-wave solutions of the viscous St. Venant equations modeling inclined shallow-water flow,
both at onset in the small-Froude number or ``weakly unstable'' limit
$F\to 2^+$ and for general values of the Froude number $F$, including the limit $F\to +\infty$. In the former, $F\to 2^+$, limit,
the shallow water equations are formally approximated by a Korteweg de Vries/Kuramoto-Sivashinsky (KdV-KS) equation 
that is a singular perturbation of the standard Korteweg de Vries (KdV) equation
modeling horizontal shallow water flow. Our main analytical result is to rigorously validate this formal limit, showing that 
stability as $F\to 2^+$ is equivalent to stability of the corresponding
KdV-KS waves in the KdV limit. Together with recent results obtained 
for KdV-KS by Johnson--Noble--Rodrigues--Zumbrun and Barker, 
this gives 
not only the first rigorous verification of stability for any single 
viscous St. Venant roll wave, but
a complete classification of stability in the weakly unstable limit. 
In the remainder of the paper, we investigate numerically and analytically the  
evolution of the stability diagram as Froude number increases to infinity. Notably, we find transition at around $F=2.3$ from weakly unstable to different, large-$F$ behavior, with stability determined by simple power law relations. 
The  latter stability criteria are potentially useful in hydraulic engineering 
applications, for which typically $2.5\leq F\leq 6.0$.
\end{abstract}

\date{\today}
\maketitle


\section{Introduction}\label{s:introduction}
In this paper, we investigate the stability of periodic wavetrain,
or {\it roll-wave}, solutions of the inclined viscous shallow-water equations
of St. Venant, appearing in nondimensional Eulerian form as
\begin{equation}\label{swe}
\displaystyle
\partial_t h+\partial_x (hu)=0,\quad \partial_t (hu)+\partial_x \left(hu^2+\frac{h^2}{2F^2}\right)=h-|u| u
+\nu\partial_x(h\partial_x u),
\end{equation} 
where $F$ is a Froude number, given by the ratio between (a chosen reference) speed of the fluid 
and speed of gravity waves, and $\nu=R_e^{-1}$, with $R_e$ the Reynolds number of the fluid.
System \eqref{swe} describes the motion of a thin layer of fluid flowing down an inclined plane,
with $h$ denoting fluid height, $u$ fluid velocity averaged with respect
to height, $x$ longitudinal distance along the plane, and $t$ time.
The terms $h$ and $|u|\,u$ on the righthand side of the second
equation model, respectively, gravitational force and turbulent friction along the bottom.\footnote{For simplicity, henceforth we restrict to cases where $u\geq0$ and write the latter term simply as $u^2$.}

Roll-waves are well-known hydrodynamic instabilities 
of \eqref{swe}, arising in the region $F>2$ for which constant solutions, 
corresponding to parallel flow, are unstable.
They appear in the modeling of such diverse
phenomena as landslides, river and spillway flow, and the topography of
sand dunes and sea beds;
see Fig. \ref{fig_phy} (a)-(b) for physical examples of roll waves and
Fig. \ref{fig_phy} (c)
for a typical wavetrain solution of \eqref{swe}.
As motivated by these applications, their stability properties have been studied
formally, numerically, and experimentally in various physically interesting regimes;
see, for example, \cite{BM} for a useful survey of this literature.
However, up until now, there has been no complete rigorous stability 
analysis of viscous St. Venant roll-waves either at the linear (spectral) or nonlinear level.

Recently, the authors, in various combinations, have developed a theoretical framework for the study of nonlinear stability of these and related periodic waves. Specifically, for the model at hand, it was shown in \cite{JZN} that, under standard {\it diffusive spectral stability assumptions} (conditions (D1)--(D3) in \S \ref{s:conditions}) together with a 
technical ``slope condition'' (\eqref{e:Eslope} below)
satisfied for ``moderate'' values $2<F\lessapprox 3.5$ of $F$
and a genericity assumption ((H1) below) satisfied almost everywhere
in parameter space,\footnote{Indeed, this appears numerically
 to be satisfied for all profiles.}
{\it roll-waves are nonlinearly stable in the sense that localized perturbations converge to 
localized spatial modulations of the background periodic wave.}
See also \cite{BJNRZ1,BJNRZ2} for discussions in the related context of 
the Kuramoto--Sivashinsky equation \cite{K,KT,Si1,Si2}. 
More recently, for general (partially) parabolic systems, detailed nonlinear asymptotic behavior under localized and
nonlocalized perturbations has been established in \cite{JNRZ2} in terms of certain formal modulation, 
or ``Whitham,'' equations.\footnote{See \cite{OZ3} for the easier multidimensional case, in which behavior is asymptotically linear due to faster decay of the linearized propagator.}

This reduces the study of stability and asymptotic behavior, at least for 
moderate values of $F$, 
to verification of the spectral stability conditions (D1)--(D3), concerning Floquet spectrum of the associated eigenvalue ODE.
However, 
it is in general a hard problem to verify such spectral assumptions analytically. 
Indeed, up to now, spectral stability has not been rigorously verified for 
{\it any} roll wave solution of the viscous St. Venant equations \eqref{swe}.

\begin{figure}[htbp]
\begin{center}
$
\begin{array}{lll}
(a) \includegraphics[scale=.3]{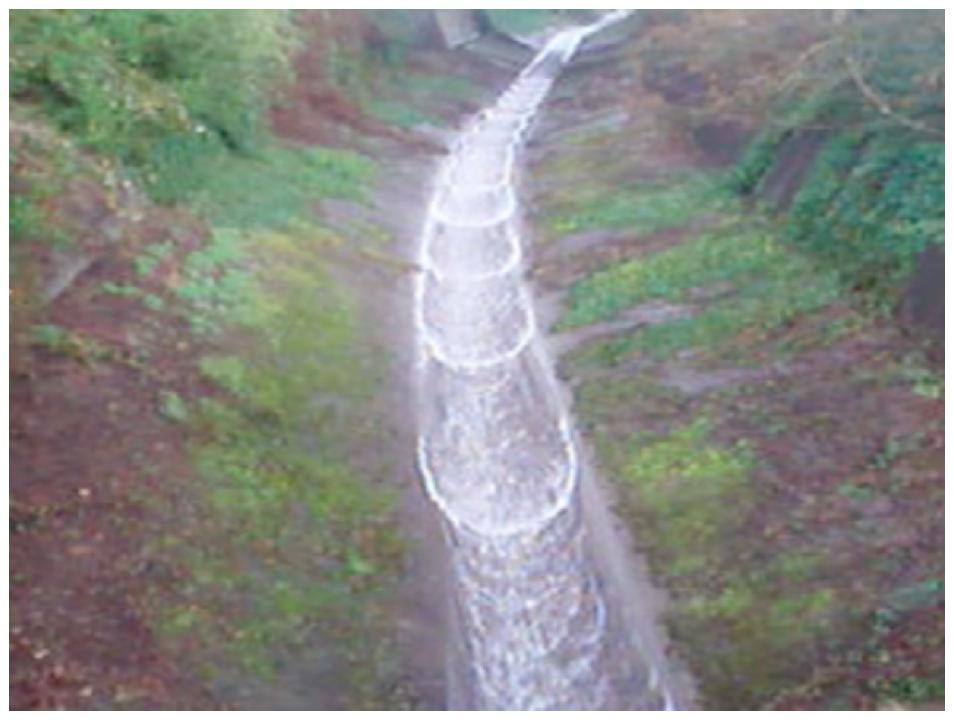}
   & (b)\ \includegraphics[scale=.35]{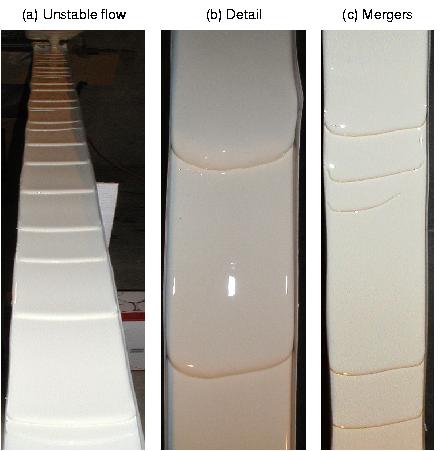}&
 (c) \includegraphics[scale=.3]{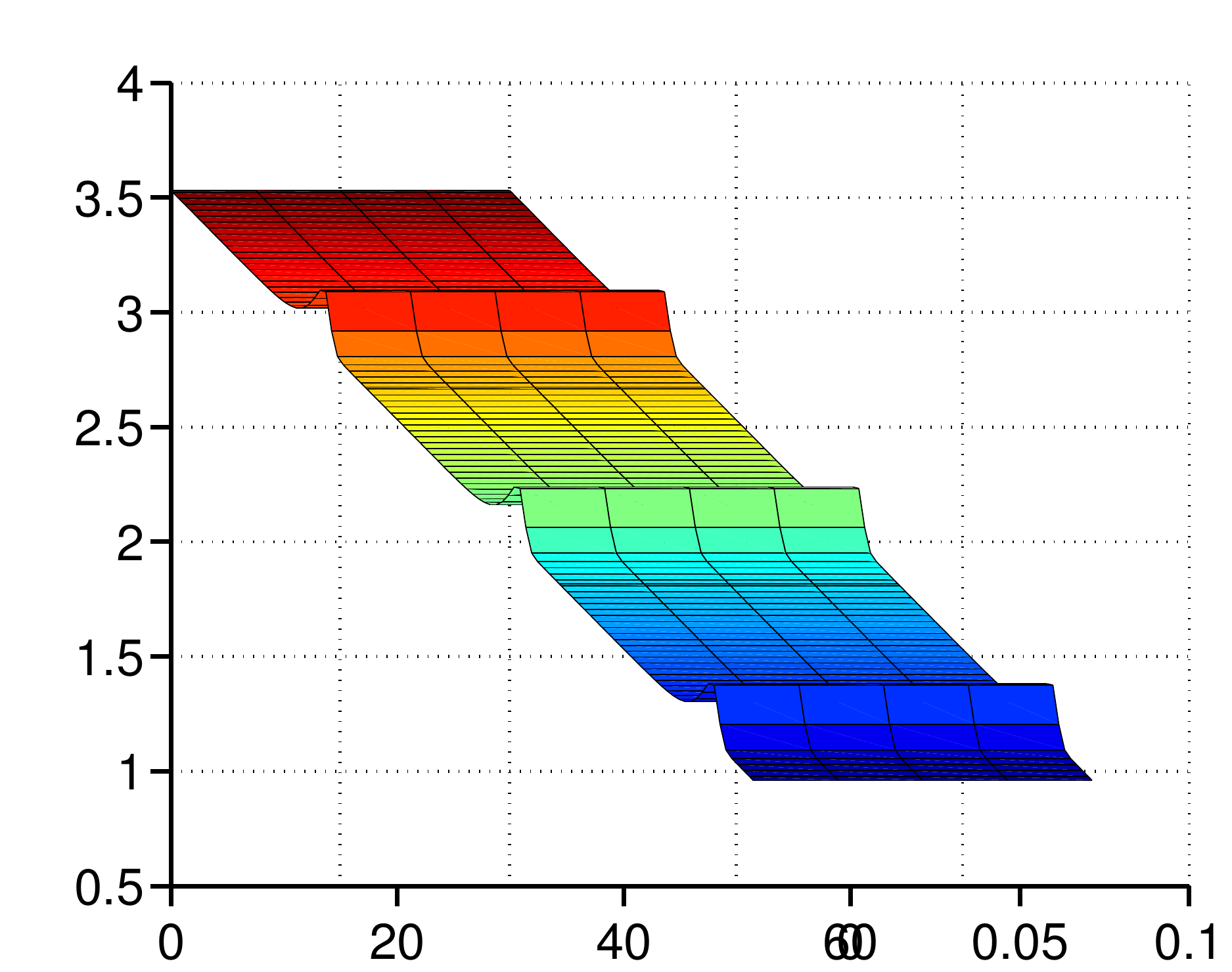}
\end{array}
$
\end{center}
\caption{Roll waves (a) on a spillway and (b) in the lab: 
pictures courtesy of Neil Balmforth, UBC. (c) Periodic profile of \eqref{swe}, 
$F = \sqrt{6}$, $\nu = 0.1$, $q = 1.5745$, $X = 17.15$.  For better comparison to experiment, we extended the profile here as constant in transverse direction.}
\label{fig_phy}\end{figure}

In some particular 
situations,
for example, at the onset of hydrodynamical instability, 
analytical proof of spectral stability
may be possible using perturbation techniques. However, most of the known examples concern reaction diffusion equations and related models like the Swift Hohenberg equations, Rayleigh B\'enard convection or  Taylor Couette flows that are all described, near the instability threshold of a background constant solution, by a Ginzburg-Landau equation derived as an amplitude equation \cite{M3}. Associated with classical Hopf bifurcation, this normal form may be rigorously validated in terms of existence and stability
by Lyapunov-Schmidt reduction about a limiting constant-coefficient operator \cite{CE,M1,M2}.

By contrast, the corresponding model for onset of hydrodynamic (roll-wave) instability in \eqref{swe}
is, at least formally, the {\it Korteweg--de Vries/Kuramoto--Sivashinsky equation} (KdV-KS)
\begin{equation}\label{kdv-ks}
\displaystyle
\partial_S v+v\partial_Y v+\varepsilon\partial_Y^3 v+\delta\big(\partial_Y^2 v+\partial_Y^4 v\big)=0,\quad \forall S>0,\forall Y\in\R,
\end{equation}
\noindent
with $0<\delta\ll 1$, $\eps>0$, a singular perturbation of the Korteweg-de Vries (KdV) equation\footnote{Without loss of generality, one can assume that $\eps^2+\delta^2=1$. See \cite{BJNRZ2}.}. Equation \eqref{kdv-ks} is derived as an amplitude equation for the shallow water system \eqref{swe} near the critical value $F=2$ above which steady constant-height flows are unstable, in the small-amplitude limit $h=\bar h+\delta^2 v$ and in the KdV time and space scaling $(Y,S)=(\delta (x-c_0 t),\delta^3 t)$ with $\delta=\sqrt{F-2}$, where $c_0$ is an appropriate
reference wave speed: see Section \ref{s:existence} below for details in the Lagrangian formulation. Alternatively, it may be derived from the full Navier-Stokes equations with free boundary from which \eqref{swe} is derived in the shallow-water limit, for
Reynolds number $R$ near the critical value $R_c$ above which steady Nusselt flows are unstable;
see \cite{Wi,YY}.

In 
this
case, neither existence nor stability reduce to computations involving constant-coefficient operators; rather, the reference states are arbitrary-amplitude periodic solutions of KdV, and the relevant operators (variable-coefficient) linearizations thereof.
{\it This makes behavior considerably richer, and both analysis and validation of the amplitude equations considerably more complicated
than in the Ginzburg--Landau case mentioned above.} Likewise, onset occurs for \eqref{swe} not through Hopf bifurcation from a single
equilibrium, but through Bogdanev--Takens, or {\it saddle-node} bifurcation involving collision of two equilibria, as discussed, e.g., in \cite{HC,BJRZ}, with limiting period thus $+\infty$, consistent with the $1/\delta$ spatial scaling of the formal model.
(The standard unfolding of a Bogdanev--Takens bifurcation as a perturbed Hamiltonian system is also consistent with KdV-KS; see Remark \ref{btrmk}.)

Nevertheless, similarly as in previous works by Mielke \cite{M1,M2} in the reaction diffusion setting, where the stability of periodic waves for the amplitude (Ginzburg-Landau) equation provides a stability result for periodic waves of the full system (Swift Hohenberg equation or Rayleight B\'enard convection), we may expect that stability for the amplitude equation, here the KdV-KS equation \eqref{kdv-ks} will provide some information on the stability of periodic waves for the viscous St. Venant system \eqref{swe}, at least in the weakly-unstable limit $F\to 2^+$. 
Our first main goal is to 
rigorously validate
this conjecture, showing that stability of roll-waves in the weakly unstable limit $F\to 2^+$ is determined by stability of corresponding solutions of \eqref{kdv-ks} under the rescaling described above. Together with previous results \cite{BN,JNRZ1,B} on stability for \eqref{kdv-ks}, this gives the first complete nonlinear stability results for roll waves of \eqref{swe}: more, {\it it gives a complete classification of stability in the weakly unstable limit.}

This gives at the same time a rigorous justification in a particular instance of the much more generally applicable and better-studied
\eqref{kdv-ks} as a canonical model for weak hydrodynamic instability in inclined thin-film flow; see, e.g., \cite{BN,CD,CDK,PSU}.
Looked at from this opposite point of view, \eqref{swe} gives an interesting extension in a specific case
of \eqref{kdv-ks} into the large-amplitude, strongly-unstable regime. Our second main goal is, by a combination of rigorous analysis and (nonrigorous but numerically well-conditioned) numerical experiment, to continue our analysis into this large-amplitude regime,
performing a systematic stability analysis for $F$ on the entire range of existence $F>2$ of periodic roll-wave
solutions of \eqref{swe}.
Our main finding here is a remarkably simple power-law description of curves bounding the region of stability in parameter space from above and below, across which particular high-frequency and low-frequency stability transitions
occur.
These curves eventually meet, yielding instability for $F$ sufficiently large. The large-$F$ description is quite different from the small-$F$ description of weakly unstable theory; indeed, there is a dramatic transition from small- to large-$F$ behavior at $F\approx2.3$, with behavior governed thereafter by the large-$F$ version. This distinction appears important for hydraulic engineering applications, where $F$ is typically $2.5-6.0$ and sometimes $10-20$ or higher \cite{Je,Br1,Br2,A,RG1,RG2,FSMA}.
Far from the onset, nonlinear stability for spectrally stable waves remains an interesting open question due to failure of the
technical (slope) condition \eqref{e:Eslope}; as discussed below, there
is reason to think this may be dropped.


\subsection{Summary of previous work}\label{s:previous}
We begin by recalling some known results that will be relied upon throughout our analysis. In particular,  we begin by recalling how spectral stability (in a suitable diffusive sense) may provide a detailed nonlinear stability result, a fact that strongly underpins and motivates our spectral studies.  We then recall the relevant numerical and analytical results for the amplitude equation \eqref{kdv-ks}, upon which our entire weakly-unstable analysis for $0<F-2\ll 1$ hinges.

\subsubsection{Diffusive spectral stability conditions}\label{s:conditions}
We first recall the standard diffusive spectral stability conditions as defined in various contexts in, for example, \cite{S1,S2,JZ1,JZ2,JZN,BJNRZ2,JNRZ2}.

Let $u(x,t)=\bar u(x-ct)$ define a spatially periodic traveling-wave solution of a general 
partial differential equation $\partial_t u=\mathcal{F}(u)$ with period (without loss of generality) one, or, equivalently, $\bar u$ be a stationary solution of $\partial_t u=\mathcal{F} + c\partial_x u$ with period one,
and let $L:=(d\mathcal{F}/du)(\bar u)+c\partial_x$ denote the associated linearized operator
about $\bar u$.  
As $L$ is a linear differential operator with $1$-periodic coefficients, standard results from Floquet theory dictate that non-trivial solutions
of $Lv=\lambda v$ can not be integrable on $\RM$, 
more generally
they can not have finite norm in $L^p(\mathbb{R})$ for any $1\leq p<\infty$.  
Indeed, it follows by standard arguments that the $L^2(\RM)$-spectrum of $L$ is purely continuous and that $\lambda\in \sigma_{L^2(\RM)}(L)$ if and only if the spectral problem $Lv=\lambda v$ has an $L^\infty(\RM)$-eigenfunction of the form
\[
v(x;\lambda,\xi)=e^{i\xi x}w(x;\lambda, \xi)
\]
for some $\xi\in[\pi,\pi)$ and $w\in L^2_{\rm per}([0,1])$; see \cite{G} or \cite[p.30-31]{R} for details.
In particular, $\lambda\in\sigma_{L^2(\RM)}(L)$ if and only if there exists a $\xi\in[-\pi,\pi)$ such
that there is a non-trivial $1$-periodic solution of the equation
\[
L_\xi w=\lambda w,\quad\textrm{where}\quad \left(L_\xi w\right)(x):=e^{-i\xi x}L\left[e^{i\xi\cdot}w(\cdot)\right](x).
\]
and
\[
\sigma_{L^2(\RM)}(L)=\sigma_{L^\infty(\RM)}(L)=\bigcup_{\xi\in[-\pi,\pi)}\sigma_{L^2_{\rm per}([0,1])}(L_\xi).
\]
The parameter $\xi$ is referred to as the Bloch or Floquet frequency and the operators $L_\xi$ are the Bloch operators associated to $L$.  Since the Bloch operators have compactly embedded domains in $L^2_{\rm per}([0,T])$ their spectrum consists entirely of discrete eigenvalues that depend continuously on the Bloch parameter $\xi$.  
Thus, the spectrum of $L$ consists entirely of $L^\infty(\RM)$-eigenvalues and may be decomposed into countably many curves
$\lambda(\xi)$ such that $\lambda(\xi)\in\sigma_{L^2_{\rm per}([0,1])}(L_\xi)$ for $\xi\in[-\pi,\pi)$.

Suppose, further, that $\bar u$ is a transversal\footnote{In a sense compatible with the algebraic structure of the system.} orbit of the traveling-wave ODE $\mathcal{F}(u)+c\partial_x u=0$.  Then near $\bar u$, the Implicit Function Theorem 
guarantees
a smooth manifold of nearby $1$-periodic traveling-wave solutions of 
(possibly) different speeds, with some dimension
$N\in\mathbb{N}$,\footnote{For both \eqref{swe} and \eqref{kdv-ks}, an easy dimensional count gives
$N=2$ \cite{JZN,BJNRZ3} because of the presence of one local conservation law in the respective sets of equations.} 
not accounting for invariance under translations. Then, the \emph{diffusive spectral stability conditions are}:

\begin{enumerate}
  \item[{(D1)}]
$\sigma_{L^2(\RM)}(L)\subset\{\lambda\ |\ \Re \lambda<0\}\cup\{0\}$.
  \item[{(D2)}]
There exists a $\theta>0$ such that for all $\xi\in[-\pi,\pi)$ we have
$\sigma_{L^2_{\rm per}([0,1])}(L_{\xi})\subset\{\lambda\ |\ \Re \lambda\leq-\theta|\xi|^2\}$.
  \item[{(D3)}]
$\lambda=0$ is an eigenvalue of $L_0$ with generalized eigenspace $\Sigma_0\subset L^2_{\rm per}([0,1])$ of dimension $N$.
\end{enumerate}

Under mild additional technical hypotheses to do with regularity of the coefficients of $\mathcal{F}$, 
hyperbolic-parabolic structure, etc., conditions (D1)--(D3) have been shown in all of the above-mentioned
settings --- in particular for periodic waves of either \eqref{swe} or
\eqref{kdv-ks} --- to imply {\it nonlinear modulation stability, at Gaussian rate}:
more precisely, provided $\|(\tilde u-\bar u)|_{t=0}\|_{L^1(\RM)\cap H^s(\RM)}$ is sufficiently small for some $s$ sufficiently large, 
there exists a function $\psi(x,t)$
with $\psi(x,0)\equiv 0$ such that
the solution satisfies
\be\label{gauss}
\|\tilde u(\cdot,t)-\bar u(\cdot-\psi(\cdot,t)-ct)\|_{L^p(\RM)}
\ +\ 
\| \nabla_{x,t}\psi(\cdot,t) \|_{L^{p}(\RM)}\leq C(1+t)^{-\frac 12 (1-1/p)},
\qquad 2\leq p\leq \infty,
\ee
valid for all $t>0$;
see \cite{JZ1,JZ2,JZN,BJNRZ2,JNRZ2}. In the case of \eqref{swe}, \eqref{kdv-ks}, for which coefficients depend analytically on the solution, 
essentially there suffices the single technical hypothesis:
\begin{enumerate}
  \item[{(H1)}]
The $N$ zero eigenvalues of $L_0$ split linearly as $\xi$ is varied with $|\xi|$ sufficiently small, in the sense that they may be expanded as
$\lambda_j(\xi)= \alpha_j \xi + o(\xi)$ for some constants $\alpha_j\in\CM$ distinct.
\end{enumerate}

We note that, since the existence of the expansion $\lambda_j(\xi)= \alpha_j \xi + o(\xi)$ in (H1) may be proved independently, the hypothesis (H1) really concerns distinctness of the $\alpha_j$, which is equivalent to the condition that the characteristics of a (formally) related first-order Whitham modulation system
be distinct, a condition that, in the case of analytic dependence of the underlying equations, as here, either holds generically with respect to nondegenerate parametrizations of the manifold of periodic traveling waves, or else uniformly fails. 
For \eqref{swe}, there is an additional {\it slope condition}
\be\label{e:Eslope}
h_x/h<(c\nu F)^{-2},
\ee
used to obtain hyperbolic-parabolic damping and high-frequency resolvent estimates by Kawashima-type energy estimates, necessary to obtain
the desired nonlinear modulational stability result; see \cite[Section 4.3]{JZN}.
Condition \eqref{e:Eslope} is known to be sufficient but a priori not necessary. As discussed in Section \ref{s:discussion}, we believe that this can be replaced by its average, $0<(c\nu F)^{-2}$, hence dropped\footnote{See the Authors Note at the end of Section \ref{s:discussion} for recent progress in this direction.}. However condition \eqref{e:Eslope} holds evidently in the small-amplitude limit $F\to 2^+$, and is observed numerically for moderate values $2<F\lessapprox 3.5$. 
The above
nonlinear stability results motivate a detailed analytical inspection of the conditions (D1)-(D3) and 
(H1), 
which is precisely the intent of the weakly nonlinear
analysis presented in Section \ref{s:F-->2} below.  We note that these conditions may be  readily checked numerically, in  a well-conditioned way, using either {\it Hill's method} (Galerkin approximation), or {\it numerical Evans function analysis} (shooting/continuous orthogonalization); 
see \cite{BJNRZ1,BJNRZ2,BJNRZ3}.

\subsubsection{Numerical evaluation for viscous St. Venant and KdV-KS}\label{s:KdVKS}
The diffusive spectral stability conditions (D1)--(D3) have been studied numerically for the viscous St. Venant equations \eqref{swe} in \cite{BJNRZ3} 
for certain ``typical'' waves and Froude numbers $F$, with results indicating existence of both stable
%
and unstable waves: more precisely, the existence of a single ``band'' of stable waves as period is varied for fixed $F$.
This echoes the much earlier numerical study of roll-waves of the classical Kuramoto--Sivashinsky equation (KS) ($\eps=0$ for \eqref{kdv-ks}) in \cite{FST} and elsewhere, that obtained similar results.

Equations \eqref{kdv-ks} have received substantially more attention, as canonical models for hydrodynamical instability in a variety of thin-film settings; as derived formally in \cite{CDK}, see also \cite[p.16, footnote~10]{R}, the model \eqref{kdv-ks} with the addition of a further term $D(v_Y^2)_Y$, $D$ constant, gives a general form for such instabilities in the weakly unstable regime. A systematic numerical study of this more general model was carried out in \cite{CDK}, across all values of $\eps$, $\delta$, $D$, and the period $X$ of the wave, and, by different methods in \cite{BJNRZ2}, for the value $D=0$ only; see Figure \ref{f:islands}, reprinted from \cite{BJNRZ2} (in close agreement also with the results of \cite{CDK}). As noted in \cite{CDK}, it may be observed from Figure 2 that the small stable band for $\eps/\delta \ll 1$ enlarges with addition of dispersion/decrease in $\delta$, reaching its largest size at $\delta/\eps=0$ (corresponding to the singular KdV) limit. For intermediate ratios of $\delta/\eps$, behavior can be considerably more complicated, with bifurcation to multiple stable bands as this ratio is varied.

\begin{figure}[htbp]
\begin{center}
\includegraphics[scale=.35]{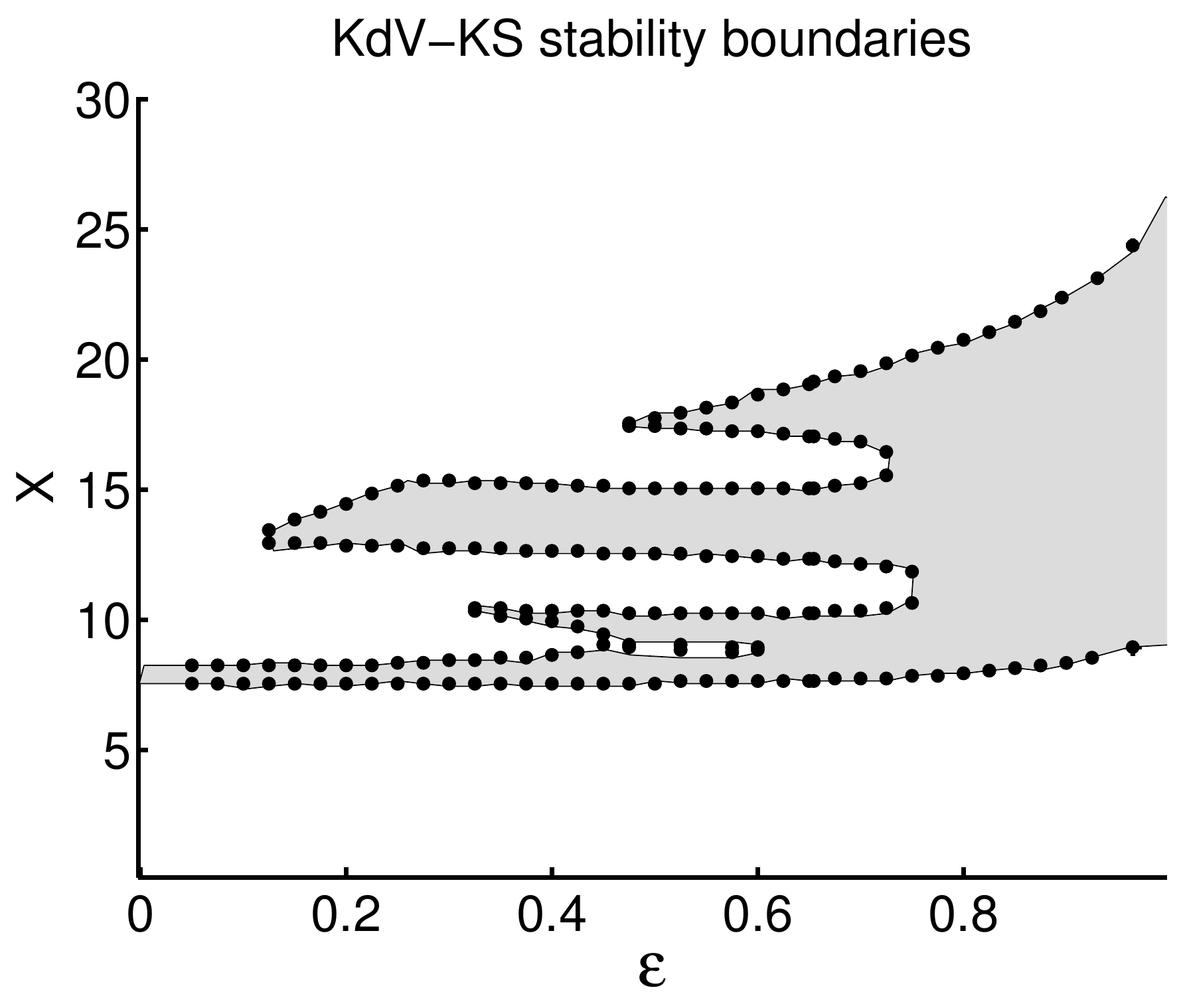}\quad
\caption{Stability boundaries (in period $X$) vs. parameter $\eps$ 
for the KdV-KS equation \eqref{kdv-ks} with $\eps^2+\delta^2=1$.
}\label{f:islands}  
\end{center}
\end{figure}

\subsubsection{The KdV limit $\delta\to0^+$}\label{s:limit}
Of special interest for us is the KdV limit $\delta \to 0^+$ for \eqref{kdv-ks}, 
treated with varying degrees of rigor in \cite{EMR,BN,NR2,JNRZ2,B}:
a singularly perturbed Hamiltonian --- indeed, {\it completely integrable} --- system.
We cite briefly the relevant results; for details, see \cite{JNRZ2,B}.

\begin{proposition}[Existence \cite{EMR}]\label{p:kdvsolnexpand}
Given any positive integer $r\geq 1$, there exists $\delta_0>0$ such that there exist periodic traveling wave solutions 
$v_\delta(\theta)$, $\theta=Y-\sigma_\delta S$, of \eqref{kdv-ks} (with $\epsilon=1$) that are analytic functions of $\theta\in\RM$ and $C^r$ functions of $\delta\in[0,\delta_0)$.  
When $r\geq 3$, profiles $v_\delta$ expand as $\delta\to 0^+$ as a $2$-parameter family
\begin{equation}\label{kdvsolnexpand}
\left\{\begin{aligned}
v_\delta(\theta;a_0,k)&=T_0(\theta; a_0,k,\kappa)+\delta T_1(\theta)+\delta^2 T_2(\theta)+O(\delta^3),\\
\sigma_\delta &=\sigma_0(a_0,k,\kappa)+\delta^2 \sigma_2+O(\delta^3),
\end{aligned}\right.
\end{equation}
where 
$$
\displaystyle
T_0(\theta;a_0,k, \kappa)\ =\ 
a_0+12k^2\kappa^2\cn^2\left(\kappa\,\theta,k\right),\quad 
\sigma_0\ =\ a_0+4\kappa^2\,(2k^2-1),
$$
comprise the $3$-parameter family (up to translation) of periodic KdV profiles and their speeds;
$\cn(\cdot,k)$ is the Jacobi elliptic cosine function with elliptic modulus $k\in(0, 1)$;
$a_0$ is a parameter related to Galilean invariance;
and $\kappa=\mathcal{G}(k)$ is determined via the selection principle 
$$
\displaystyle
\left(\frac{K(k)\mathcal{G}(k)}{\pi}\right)^2=
   \frac{7}{20}\frac{2(k^4-k^2+1)E(k)-(1-k^2)(2-k^2)K(k)}{(-2+3k^2+3k^4-2k^6)E(k)+(k^6+k^4-4k^2+2)K(k)},
$$
where $K(k)$ and $E(k)$ are the complete elliptic integrals of the first and second kind. The period $X(k)=2K(k)/\mathcal{G}(k)$ is in one-to-one correspondence with $k$. Moreover the functions $(T_i)_{i=1,2}$ are (respectively odd and even) solutions of the linear equations
\be\label{cork}
\mathcal{L}_0[T_0]T_1=T_0''+T_0'''',\quad
\mathcal{L}_0[T_0]T_2=\left(\frac{T_1^2}{2}-\sigma_2 T_0\right)'+ T_1''+ T_1'''',
\ee
on $(0,2K(k)/\mathcal{G}(k))$ with periodic boundary conditions,
where 
$\mathcal{L}_0[T_0]:=-\partial_\theta^3-\partial_\theta\left(T_0-\sigma_0\right)$
denotes the linearized KdV operator about $T_0$.
\end{proposition}

Throughout this manuscript, we let $v_\delta(\cdot;a_0,k)$ and $\sigma_\delta(a_0,k)$ denote the periodic traveling wave profiles and wave speeds, respectively, as described in Proposition \ref{kdvsolnexpand}.  

\br\label{r:selection}
\textup{
The additional parameter $\kappa$ for periodic KdV waves as compared to KdV-KS waves reflects the existence of
the additional conserved quantity of the Hamiltonian at $\delta=0$. The selection principle $\kappa=\mathcal{G}(k)$ is precisely the condition that
the periodic Hamiltonian orbits at $\delta=0$ persist, to first order, for $0<\delta\ll 1$.  Alternatively, this condition
can be written explicitly as $\int_0^{2K(k)/\kappa}T_0\,(T_0''+T_0'''')dx=0$.
}
\er

By Galilean invariance of the underling equation \eqref{kdv-ks}, the stability properties of the above-described $X=X(k)$-periodic solutions 
are independent of the parameter $a_0$.  Hence, for stability purposes one may identify waves with a common period, 
fixing $a_0$ and studying stability of a one-parameter family in $k$. It is known \cite{KSF,Sp,BD} that the spectra of the linearized operator $\mathcal{L}[T_0]$, considered
on $L^2(\RM)$, about a periodic KdV wave $T_0$ is spectrally stable in the neutral, 
Hamiltonian, sense, i.e., all eigenvalues of the Bloch operators
\[
\mathcal{L}_\xi[T_0]:=(\partial_Y+i\xi)\left(-(\partial_Y+i\xi)^2- T_0+\sigma_0\right):L^2_{\rm per}(0,X)\to L^2_{\rm per}(0,X),
\]
considered with compactly embedded domain $H^3_{\rm per}(0,X)$, are purely imaginary for each $\xi\in[-\pi/X,\pi/X)$. Moreover, the explicit description of the spectrum obtained in \cite{BD} also yields\footnote{In \cite{BJNRZ2,B}, some of these facts remained unnoticed to the authors. In particular, in \cite{BJNRZ2}, the condition that only $\xi=0$ yields $\lambda=0$ was gathered to condition (A) below to form condition (A1) and distinctness of the $\alpha_j$ was identified as condition (A2).} that $\lambda=0$ is
an eigenvalue of $\mathcal{L}_0[T_0]$ of algebraic multiplicity three, that $\lambda=0$ is an eigenvalue of $\mathcal{L}_\xi[T_0]$ only if $\xi=0$, and that the three zero eigenvalues of $\mathcal{L}_0[T_0(\cdot;a_0,k,\mathcal{G}(k))]$ expand about for $|\xi|\ll 1$ as
\begin{equation}\label{kdvspecexpand}
\lambda_j(\xi)=i\alpha_j(\xi)\xi=i\xi \alpha_j + O(\xi^2),~~~j=1,2,3
\quad
\hbox{\rm with $\alpha_j\in\RM$ distinct.}
\end{equation}
We introduce one final technical condition, first observed then proved numerically to hold, at least for 
KdV waves that are limits as $\delta\to 0$ of
stable waves of \eqref{kdv-ks} \cite{BD,BJNRZ2,B}:\footnote{
	In fact (A) has been verified (see Proposition \ref{p:blake} below)
	on essentially the entire range $k\in (0,1)$; 
	we know of no instance where it fails.
}
\begin{itemize}
\item[(A)]
A given parameter $k\in(0,1)$ is said to satisfy condition (A) if 
the non-zero eigenvalues of the linearized (Bloch) KdV operator $\mathcal{L}_\xi[T_0]$ about $T_0(\cdot; a_0,k, \mathcal{G}(k))$
are simple for each $\xi\in[-\pi/X,\pi/X)$.
\end{itemize}
Note that the set of 
$k\in(0,1)$ for which property (A) holds is open. 

Given a periodic traveling wave solution $T_0(\cdot;a_0,k,\mathcal{G}(k))$ of the KdV equation with 
elliptic modulus $k\in(0,1)$ satisfying condition (A) above,
we now consider the spectral stability of the associated family of periodic traveling wave solutions $v_\delta(\cdot;a_0,k)$,
defined for $\delta\in[0,\delta_0)$, with $\delta_0$ as in Proposition \ref{kdvsolnexpand}, as solutions of the KdV-KS equation \eqref{kdv-ks}. To this end, notice that,
	assuming $k\in(0,1)$ satisfies condition (A),
	the non-zero Bloch-eigenvalues $\lambda(\xi)$ of the linearized KdV-KS operator
\[
L_\xi[v_\delta]=e^{-i\xi \cdot}\left[-\delta\left(\partial_Y^4+\partial_Y^2\right)-\partial_Y^3-\partial_Y\left(v_\delta-\sigma_\delta\right)\right]e^{i\xi \cdot}:L^2_{\rm per}(0,X)\to L^2_{\rm per}(0,X)
\]
admit a smooth\footnote{In the sense of Proposition \ref{p:kdvsolnexpand} that one can reach arbitrary prescribed regularity.} expansion in $\delta$ for $0<\delta\ll 1$. 
In particular, for each pair $(\xi,\lambda_0)$ with $\lambda_0\in\sigma(\mathcal{L}_\xi[T_0])\setminus\{0\}$
and $\xi\in[-\pi/X,\pi/X)$ there is a unique spectral curve $\lambda(\xi,\lambda_0,\delta)$ bifurcating from $\lambda_0$ 
smoothly in $\delta$, and it takes the form
\begin{equation}\label{kdvks-evexpand}
\lambda(\delta;\xi,\lambda_0)\ =\ \lambda_0+\delta\lambda_1(\xi,\lambda_0)+O(\delta^2)
\end{equation}
for some $\lambda_1(\xi,\lambda_0)$.
It is then natural to expect that the signs of the real parts of the first order correctors $\lambda_1(\xi,\lambda_0)$ in the above expansion be indicative of stability or instability of the near-KdV profiles $u_\delta$ for $0<\delta\ll 1$.
With this motivation in mind, for any $k\in(0,1)$ that satisfies condition (A) above\footnote{Condition (A) is independent of $a_0$, 
	holding
	for every $a_0$ or for none. Likewise, 
$ {\rm Ind}(k)$
is independent of $a_0$.}, we define
\begin{equation}\label{index}
{\rm Ind}(k):=\sup_{\substack{\lambda_0\in\sigma(\mathcal{L}_\xi[T_0(\cdot;a_0,k,\mathcal{G}(k))])\setminus\{0\}\\ \xi\in[-\pi/X(k),\pi/X(k))}}\Re\left(\lambda_1(\xi,\lambda_0)\right).
\end{equation}
Evidently, ${\rm Ind}(k)>0$ is a sufficient condition for the spectral instability for $0<\delta\ll 1$ of the near-KdV waves $v_\delta$ bifurcating from $T_0$.  The next proposition states that the condition ${\rm Ind}(k)<0$ is also sufficient for the diffusive spectral stability of the $v_\delta$, for $0<\delta\ll 1$. Define the open set
\[
\mathcal{P}\ :=\ \left\{\ k\in(0,1)\ \middle|\ \textrm{condition}~(A)~\textrm{holds for}~k~\textrm{and Ind}(k)<0\ \right\}.
\] 

\begin{proposition}[Limiting stability conditions \cite{JNRZ1}]\label{p:kdvstab}
For each $k\in\mathcal{P}$ there exists a neighborhood $\Omega_k\subset(0,1)$ of $k$ and  $\delta_0(k)>0$ such that $\Omega_k\subset\mathcal{P}$ and for any $(a_0,\tilde k,\delta)\in\R\times\Omega_k\times(0,\delta_0(k))$ the non-degeneracy and spectral stability conditions (H1) and (D1)-(D3) hold for $v_\delta(\cdot; \tilde a_0, \tilde k)$. In particular, $\mathcal{P}$ is open and 
$\delta_0(\cdot)$ can be chosen uniformly on compact subsets of $\mathcal{P}$.
\end{proposition}

Proposition \ref{p:kdvstab} reduces the question of diffusive spectral stability and nonlinear stability --- in the sense defined in Section \ref{s:conditions} --- of the near-KdV profiles constructed in Proposition \ref{kdvsolnexpand} to the verification of 
the structural condition (A) and the evaluation of the function ${\rm Ind}(k)$.  
Note condition (A) is concerned only with the spectrum
of the linearized KdV operator about the limiting KdV profile $v_0$; its validity is discussed in detail in \cite[Section~1]{JNRZ2}.
Further, note that due to the triple eigenvalue of the KdV linearized operator
at the origin, the fact that $\textrm{Ind}(k)$ is sufficient for stability is far from a foregone conclusion, and this represents
the main contribution 
of \cite{JNRZ1}.

To evaluate ${\rm Ind}(k)$, using the complete integrability of the KdV equation to find an explicit parametrization of the KdV spectrum and eigenfunctions about $v_0$ \cite{BD}, one can construct a continuous multi-valued mapping\footnote{While $\lambda_1(\xi,\lambda_0)$ is defined above only when $\lambda_0$ is a simple eigenvalue of the limiting linearized KdV operator $\mathcal{L}_\xi[v_0]$ , it possesses an explicit expression in terms of $k$ and an auxiliary Lax spectral parameter that extend this function to points where simplicity fails. Note in particular that this extension is triple-valued at $(0,0)$.}
\[
[-\pi/X,\pi/X)\times\RM i\ni (\xi,\lambda_0)\mapsto \Re(\lambda_1(\xi,\lambda_0))\in\RM
\]
that is explicitly computable in terms of Jacobi elliptic functions; see \cite{BN} or \cite[Appendix A.1]{BJNRZ2}. This mapping may then be analyzed numerically. Numerical investigations of \cite{BN,BJNRZ2} suggest stability for limiting periods $X=X(k)$ in an open interval $(X_m,X_M)$ and instability for $X$ outside $[X_m,X_M]$, with $X_m\approx 8.45$ and $X_M\approx 26.1$, thus completely classifying stability for $0<\delta\ll1$. The following result of Barker \cite{B}, established by numerical proof using
interval arithmetic, 
gives rigorous validation of these observations for all limiting periods $X(k)$
except for a set near the boundaries
of the domain of existence $X(k)\in (2\pi, +\infty)\approx (6.2832, +\infty)$
corresponding to the limits $k\to 0,1$.

%
\begin{proposition}[Numerical proof \cite{B}]\label{p:blake}
With $k_{\min} =$ \kUnstableLowerLeft \   and $k_{\max}=$ \kUnstableUpperRight, corresponding to $X_{\min}\approx 6.284$ and $X_{\max}\approx 48.3$, condition (A) holds on $[k_{\min},k_{\max}]$. Moreover, there are $k_l \in [0.9421,0.9426]$ and $k_r\in[0.99999838520,0.99999838527]$, corresponding to $X_l\in [8.43,8.45]$ and $X_r\in [26.0573, 26.0575]$, 
such that $\mathcal{P}\cap[k_{\min},k_{\max}]=(k_l,k_r)$ and ${\rm Ind}$ takes positive values on $[k_{\min},k_{\max}]\setminus[k_l,k_r]$.
\end{proposition}

As noted in \cite{B}, the limits $k\to 0$ and $k\to 1$ not treated in Proposition \ref{p:blake}, corresponding to Hopf and homoclinic limits, though inaccessible by the numerical methods of \cite{B}, should be treatable by asymptotics relating spectra
to those of (unstable \cite{BJRZ,BJNRZ3}) limiting constant and homoclinic profiles.

\subsection{Description of main results}\label{s:results}

As mentioned previously, the fact that the KdV-KS equation \eqref{kdv-ks} serves as an amplitude equation for the shallow water 
system \eqref{swe} in the weakly nonlinear regime $0<\delta=\sqrt{F-2}\ll 1$ suggests that Proposition \ref{p:kdvstab} and Proposition \ref{p:blake} should
have natural counterparts for the stability of roll-waves in \eqref{swe}.  To explore this connection,
following \cite{JZN,JNRZ2,BJRZ}, we find it convenient to rewrite the viscous St. Venant 
equations \eqref{swe} in their equivalent  Lagrangian form:
\begin{equation}\label{swl}
\displaystyle   
\partial_t \tau-\partial_x u=0,\quad \partial_t u+\partial_x\left(\frac{\tau^{-2}}{2F^2}\right)=1-\tau\,u^2+\nu\partial_x(\tau^{-2}\partial_x u),
\end{equation}
where $\tau:=1/h$ and $x$ denotes now a Lagrangian marker rather than 
a physical location $\tilde x$, satisfying the relations $dt/d\tilde x=u(\tilde x,t)$ and 
$dx/d\tilde x=\tau(\tilde x,t)$.
In these coordinates, slope condition \eqref{e:Eslope} takes the form
\be\label{e:Lslope}
2 \nu u_x < F^{-2}.
\ee
Hereafter, we will work exclusively with the formulation \eqref{swl}.

\br
\textup{
Though nontrivial, the one-to-one correspondence between periodic waves of the Eulerian and Lagrangian forms is a well-known fact. A fact that seems to have remained unnoticed until very recently is that this correspondence extends to the spectral level even in its Floquet-by-Floquet description. In particular without loss of generality one may safely study spectral stability in either formulation. See for instance \cite{BNR} for an explicit description of the former correspondence and \cite{BMR} for the spectral connection, both in the closely related context of the Euler--Korteweg system. 
}
\er

\subsubsection{The weakly unstable limit $F\to 2^+$}\label{s:2limit}
Our first three results, and the main analytical results of this paper, comprise a rigorous validation of KdV-KS as a description
of roll-wave behavior in the weakly unstable limit $F\to 2^+$. Let $(\tau_0, u_0)$, $u_0={\tau_0}^{-1/2}$, 
be a constant solution of \eqref{swl}, and $c_0:=\tau_0^{-3/2}/2$. Setting $\delta=\sqrt{F-2}$, we introduce the rescaled dependent and independent variables
\be\label{rescaled}
\tilde \tau= 3\delta^{-2}\left(\frac{\tau}{\tau_0}-1\right),
\quad
\tilde u= 6\delta^{-2}\left(\frac{u}{u_0}-1\right),
\quad
Y=\frac{\tau_0^{5/4}\delta(x-c_0t)}{\nu^{1/2}},\quad S:=\frac{\delta^3 t}{4\tau_0^{1/4}\nu^{1/2}}.
\ee
Our first result concerns existence of small amplitude periodic traveling wave solutions in the limit $\delta\to 0^+$.

\begin{theorem}[Existence]\label{maine}
There exists $\delta_0>0$ such that there exist periodic traveling wave solutions  of 
\eqref{swl}, in the rescaled coordinates \eqref{rescaled} $(\tilde \tau,\tilde u)_\delta(\theta)$, 
$\theta=Y-\tilde c_\delta(a_0,k)S$, that are analytic functions of $\theta\in\RM$, $(a_0,k)\in\R\times(0,1)$ and  $\delta\in[0,\delta_0)$ and that, in the limit $\delta\to 0^+$, expand as a $2$-parameter family
\begin{equation}\label{kdvcompare}
\left\{\begin{aligned}
\tilde \tau_\delta(\theta;a_0,k) &=-
v_{\tilde\delta}(\theta;a_0,k) +O(\delta^2),\\
\tilde u_\delta(\theta;a_0,k)&= 
-\tilde \tau_\delta(\theta;a_0,k)
+\tfrac{\delta^2}{2}(\tilde q(a_0,k)-3\tau_0^{-1}\tilde c_\delta(a_0,k)\tilde \tau_\delta(\theta;a_0,k)),\\
\tilde c_\delta(a_0,k)&=  \sigma_{\tilde\delta}(a_0,k)+O(\delta^2),
\end{aligned}\right.
\end{equation}
where $\tilde\delta=\delta/(2\tau_0^{1/4}\nu^{1/2})$, $v_\delta(\theta;a_0,k)$ and $\sigma_\delta(a_0,k)$ are the small-$\delta$ traveling-wave profiles and speeds of KdV-KS described in \eqref{kdvsolnexpand} and 
$$
\tilde q(a_0,k)\ =\ 24k^2(1-k^2)(\mathcal{G}(k))^4-a_0\,(\tfrac12a_0+4(\mathcal{G}(k))^2(2k^2-1))\ \equiv\ v_0^2/2-\sigma_0v_0+v_0''
$$ 
is a constant of integration in the limiting KdV traveling-wave ODE (see \eqref{kdvprof}).  
\end{theorem}

For brevity, throughout the paper we shall often leave implicit the dependence on $(a_0,k)$ or $a_0$.

\br \label{regrmk}
\textup{
As we will see in the analysis below, the weakly unstable limit for the  St. Venant equations \eqref{swl} is a {\it regular perturbation}
of KdV, rather than a singular perturbation as in the KdV-KS case, a fact reflected in the stronger regularity conclusions of 
Theorem~\ref{maine} as compared to 
Proposition \ref{p:kdvsolnexpand}.
}
\er

Our next result concerns the spectral stability of the small amplitude periodic traveling wave solutions constructed
in Theorem~\ref{maine} when subject to arbitrary small localized (i.e. integrable) perturbations on the line.

\begin{theorem}[Limiting stability conditions]\label{mains}
For each $k\in\mathcal{P}$ there exists a neighborhood $\Omega_k\subset(0,1)$ of $k$ and  $\delta_0(k)>0$ such that for any $(a_0,\tilde k,\delta)\in\R\times\Omega_k\times(0,\delta_0(k))$ the non-degeneracy and spectral stability conditions (H1) and (D1)-(D3) hold for $(\tau,u)_\delta(\cdot; a_0, \tilde k)$,
along with slope condition \eqref{e:Lslope}.
In particular, $\delta_0(\cdot)$ can be chosen uniformly on compact subsets of $\mathcal{P}$.
Conversely for each $k\in(0,1)$ such that condition (A) holds but $\textrm{Ind}(k)>0$ (where \textrm{Ind} is defined as in \eqref{index}), there exists a neighborhood $\Omega_k\subset(0,1)$ of $k$ and  $\delta_0(k)>0$ such that if $(a_0,\tilde k,\delta)\in\R\times\Omega_k\times(0,\delta_0(k))$ then $(\tau,u)_\delta(\cdot; a_0, \tilde k)$ is spectrally unstable.
\end{theorem}

From Theorem~\ref{maine}, it follows in 
particular that 
%
our 
roll waves have asymptotic period $\sim \delta^{-1}$ and 
amplitude $\sim \delta^2$ in the weakly unstable limit $F\to 2^+$; that is,
this is a {\it long-wave, small-amplitude} limit.
Likewise, $u_x\sim \delta^3$ so that 
\eqref{e:Lslope} is automatically satisfied for $\delta\ll 1$.
In Theorems~\ref{maine} and~\ref{mains}, rescaling period and amplitude to order one, we find that in this regime \emph{existence
and stability} are indeed well-described by KdV-KS $\to$ KdV:
to zeroth order by KdV, and to first correction by KdV-KS.  
%

Combining Theorem \ref{mains} with Proposition \ref{p:blake},
and untangling coordinate changes, we obtain the
following 
{\it  essentially complete description of stability of viscous St. Venant roll-waves
in the 
limit $F\to 2^+$.}

\begin{corollary}[Limiting stability region]\label{c:stabclass}
For $\delta=\sqrt{F-2}$ sufficiently small, uniformly for $\delta\,X$ on compact sets, periodic traveling waves of \eqref{swl}
are stable for (Lagrangian) periods $X \in \tfrac{\nu^{1/2}}{\tau_0^{5/4}\delta}(X_l,X_r)$
and unstable for $X \in \tfrac{\nu^{1/2}}{\tau_0^{5/4}\delta}[X_{\min},X_l)$ and $X\in \tfrac{\nu^{1/2}}{\tau_0^{5/4}\delta}(X_r,X_{\max}]$ 
where $X_{\min}$, $X_l$, $X_r$, $X_{\max}$ are as in Proposition~\ref{p:blake}, in particular, 
$X_{\min}\approx 6.284$, $X_l\approx 8.44$, $X_r\approx 26.1$, and $X_{\max}\approx 48.3$.
\end{corollary}

\subsubsection{Large-Froude number limit $F\to+\infty$}\label{s:intro_infty}
We complement the above weakly nonlinear analysis by continuing into the large amplitude regime, beginning with a study
of the distinguished large-Froude number limit $F\to +\infty$.
The description of this limit requires a choice of scaling in the parameters indexing the family of waves. To this end, we first emphasize that a suitable 
parametrization,
available for the full range of Froude numbers, is given by $(q,X)$, where $q:= -c\bar \tau-\bar u$ is a constant of integration in the associated traveling-wave ODE, corresponding to {\it total outflow}, and $X$ is the {\it period}. 
As discussed in Section \ref{s:per} below, the associated two-parameter family of possible scalings may be reduced by the requirements that (i) the limiting system be nontrivial, and (ii) the limit be a regular perturbation, to a one-parameter family indexed by $\alpha\geq -2$, given explicitly via
\be\label{scales}
\tau= a F^\alpha, \quad
u= b F^{-\alpha/2}, \quad
c= c_0 F^{-1 - 3\alpha/2}, \quad
X= X_0 F^{-1/2 - 5\alpha/4}, \quad
q=q_0 F^{-\alpha/2}.
\ee
where $a,b:\RM\to\RM$ and $c_0,X_0,q_0$ are real constants. Note, from the relation $X=1/k$ between period and wave number $k$,
that we have also $k= k_0 F^{1/2 + 5\alpha/4}$.

Under this rescaling, 
moving to the co-moving frame $(x,t)\mapsto (k(x-ct),t)$,
we find
that $X$-periodic traveling wave solutions
of \eqref{swl} correspond to $X_0$-periodic solutions of
the {\it rescaled profile equation}
\be\label{rprof}
 a''= (- a^2/c_0k_0^2 \nu) 
\big(k_0  a' F^{-3/2-3\alpha/4}(c_0^2-1/ a^3) - 1
+  a(q_0 - c_0F^{-1} a)^2 - 2c_0 k_0^2 \nu ( a')^2/ a^3\big),
\ee
where $b$ is recovered from $a$ via the identity
$
 b= -q_0 - c_0 F^{-1} a;
$
see Section \ref{s:per} below for details.
Noting that the behavior of  $F^{-3/2-3\alpha/4}$ as $F\to\infty$ depends on whether $\alpha=-2$ or $\alpha>-2$, one finds two different limiting
profile equations in the limit $F\to\infty$: a (disguised\footnote{Indeed, the profile equation in this case is Hamiltonian in the unknown $h=\frac{1}{a}$}) 
Hamiltonian equation supporting a selection principle, when $\alpha>-2$, and a non-Hamiltonian equation in the boundary case $\alpha=-2$; see Section \ref{s:per} below.  
Further, by elementary phase plane analysis when $\alpha>-2$
or direct numerical investigation when $\alpha=-2$, periodic solutions of the limiting profile equations are seen to exist as $2$-parameter families parametrized
by the period $X_0$ and the discharge rage $q_0$.
Noting that, by design, the rescaled profile equation \eqref{rprof} is a \emph{regular perturbation} of the appropriate limiting profile equation for all $\alpha\geq -2$,
we readily obtain the following asymptotic description  

\begin{theorem}\label{t:limprof}
For sufficiently large $F$, 
generically, 
$X_0$-periodic profiles of \eqref{rprof}, obtained under the scaling \eqref{scales}, emerge for each $\alpha\geq -2$ from $X_0$-periodic solutions of the appropriate limiting profile equation obtained by taking $F\to\infty$ in \eqref{rprof} and, when $\alpha=-2$, satisfying a suitable selection principle.
\end{theorem}

Next, we study the spectral stability of a pair of fixed periodic profiles $(\bar{a},\bar{b})$ constructed above.  
One may readily check that, under the further rescaling  
$Fb=\check b$ and $F^{1/2+\alpha/4} \lambda=\Lambda$, the linearized spectral problems around such a periodic profile $(\bar{a},\bar{b})$
is given by
\begin{equation}\label{e:specF->infty}
\begin{aligned}
\Lambda  a -c_0k_0 a'-k_0\check b'&=0\\
F^{-3/2-3\alpha/4}\Big(\Lambda \check b-c_0k_0 \check b'- k_0 (a/\bar a^3)'\Big)&=
-2F^{-1}\bar a \bar b \check b - \bar b^2 a + 
\nu k_0^2 (\check b'\bar a^2 + 2 c_0 \bar a'a/\bar a^3)',
\end{aligned}
\end{equation}
where 
$(a,b)$ denotes the perturbation of the underlying state $(\bar{a},\bar{b})$.
The limiting spectral problems obtained by taking $F\to\infty$ again depend on whether $\alpha=-2$ or $\alpha>-2$.  In particular,
we note the spectral problem is Hamiltonian, and hence possesses a natural four-fold symmetry about the real- and imaginary-axes, when $\alpha>-2$; see
\ref{s:per} below for details.
Since \eqref{e:specF->infty} is, again by design, a \emph{regular perturbation} of the appropriate limiting spectral problems for all $\alpha\geq -2$, 
we obtain by standard perturbation methods (e.g., the spectral/Evans function convergence results of \cite{PZ}) the following sufficient \emph{instability condition}.

\begin{corollary}\label{t:limstab}
For all $\alpha\geq -2$, under the rescaling \eqref{scales}, the profiles of \eqref{rprof} converging as $F\to\infty$ to solutions of the appropriate limiting profile equation, as described in Theorem \ref{t:limprof} are \emph{spectrally unstable} if the appropriate limiting spectral problem about the limiting profiles admit $L^2(\RM)$-spectrum in $\Lambda$ with positive real part.
\end{corollary}

As clearly discussed in Section \ref{s:per}, for $\alpha>-2$ the limiting profile equation, associated selection principle, and limiting spectral
problem are \emph{independent} of the value of $\alpha$.  
%
%
Thus, the above instability criterion for $F\to +\infty$ can be determined by the study of just two
model equations: one for $\alpha=-2$ and one for any other fixed $\alpha>-2$.
Both regimes include particularly physically interesting choices since $\alpha=0$ corresponds to holding the outflow $q$ constant as $F\to\infty$, while
$\alpha=-2$ corresponds to holding the Eulerian period $\Xi(X_0)$ constant as $F\to\infty$\footnote{
Here, the Eulerian scaling relation $\Xi=\Xi_0 F^{-\frac1 2-\frac\alpha 4}$ 
follows by $\Xi=\int_0^X \bar \tau(x)dx$, \eqref{scales}, and convergence of the profile $a$.}.
In Section \ref{s:numinf} we investigate numerically the stability of the limiting spectral problems
in both the cases $\alpha=-2$ and $\alpha=0$.  This numerical study indicates that, in both these cases, all periodic solutions of the appropriate limiting
profile equations are \emph{spectrally unstable} and hence, by Corollary \ref{t:limstab}, that  spectrally stable periodic traveling wave solutions
of the viscous St. Venant system \eqref{swl} do not exist for sufficiently large Froude numbers; see Figure \ref{comparison}.  
%

\medskip

{\bf Numerical Observation 1.} {\it For both $\alpha>-2$ and $\alpha=-2$, 
the limiting eigenvalue equations have strictly unstable spectra, hence
converging profiles are spectrally unstable for $F$ sufficiently large.}

\medskip
\begin{figure}[htbp]
 \begin{center}
$
\begin{array}{lccr}
(a) \includegraphics[scale=0.22]{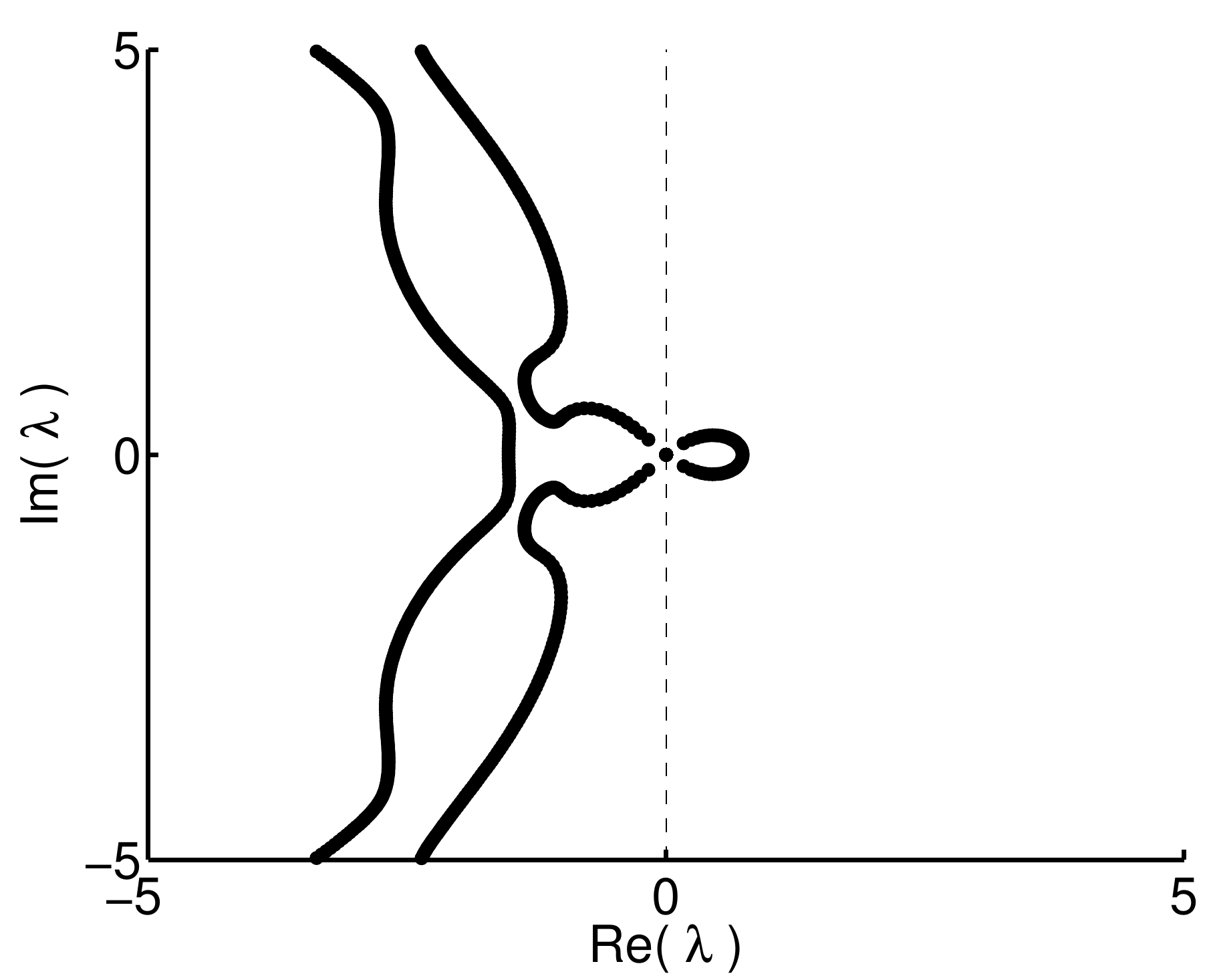}&\quad&&(b) \includegraphics[scale=0.22]{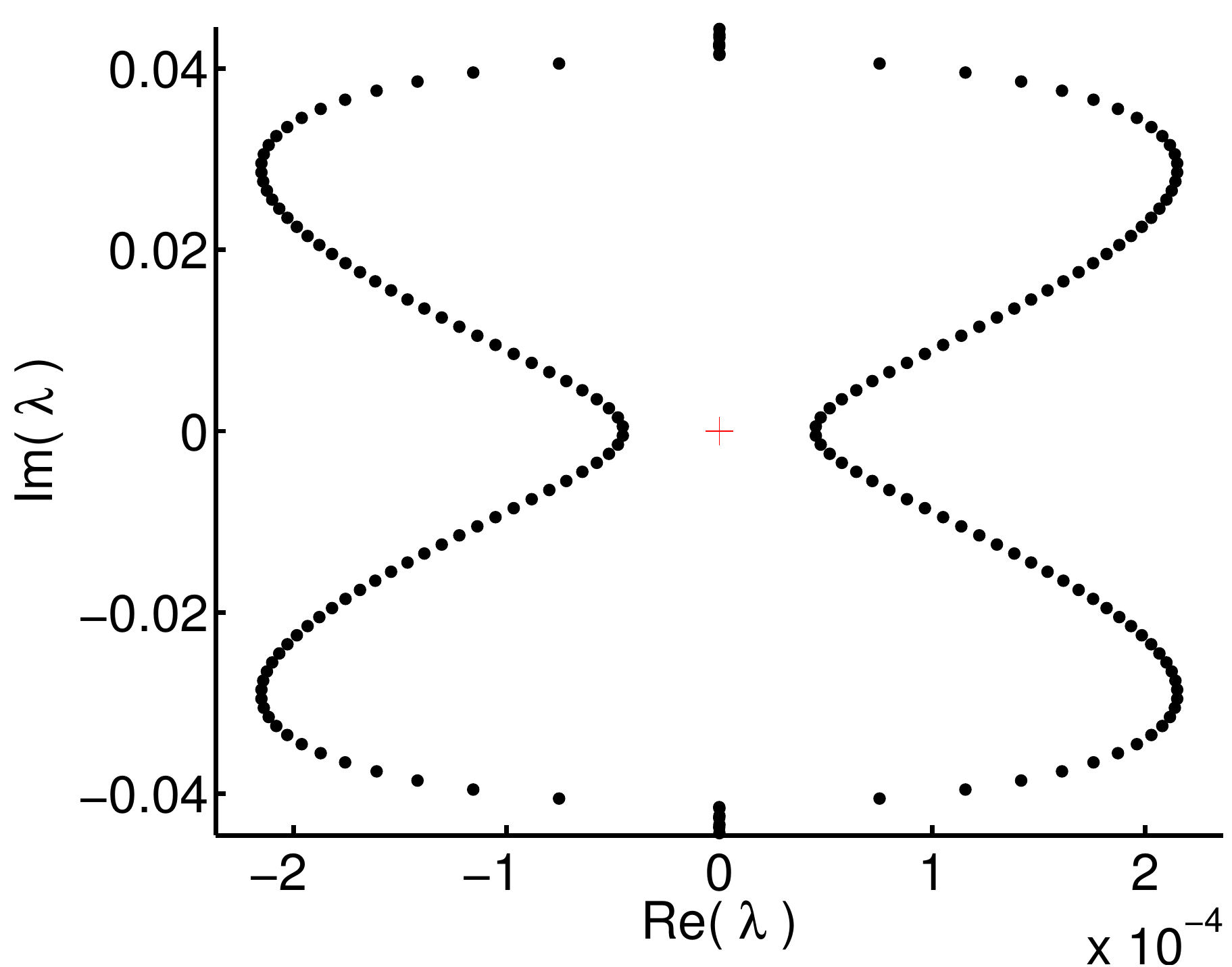}
\end{array}
$
\end{center}
\caption{In (a) and (b) we plot a numerical sampling of the (unstable) spectrum corresponding to the $F\to\infty$ limiting spectral problems
for the cases $\alpha=-2$ and $\alpha>-2$, respectively, for a representative periodic stationary solution of the appropriate limiting profile equation.
}
\label{comparison}\end{figure}

%
%
%

\subsubsection{Intermediate $F$}\label{intro:intermediate_F}
%
%

We complete our stability investigation in Section \ref{numerical:intermediate} by carrying out a numerical study for $F$ bounded away from the distinguished values
$2$ and $+\infty$ of the $L^2(\RM)$-spectrum of the linearized operator
obtained from linearizing \eqref{swl} about a given periodic traveling wave solution.
%
For $F$ relatively small ($2<F\le 4$), we find, unsurprisingly, 
a smooth continuation of the picture for $F\to 2^+$, featuring a single band of stable periods between two concave upward curves; see Fig. \ref{fig476} (c).
However, continuing into the large-but-not-infinite regime ($2.5\leq F\leq 100$),
we find considerable additional structure beyond that described in
Numerical Observation 1.


Namely, for $\alpha\in[-2,0]$ we see that the stability region is enclosed in a lens-shaped region between two smooth concave upward curves corresponding to the lower stability boundary and an upper high-frequency instability boundary, pinching off at a special value $F_*(\alpha)$ after which, consistent with Numerical
Observation 1, stable roll waves no longer exist;
see Figure \ref{fig476}(a) for the case $\alpha=-2$.  Examining these curves further for different values of $\alpha$, we find that they obey a remarkably
simple power-law description in terms of $F$, $q$, $\alpha$, and $X$; see Figure \ref{fig476}(b) for an example log-log plot in the case $\alpha=-2$.  
The general description of this power-law behavior is provided by the following.

\medskip
{\bf Numerical Observation 2.} {\it Both lower and upper stability boundaries 
appear for $F\gg1$ to be governed by \emph{universal power laws}
$c_1 \log F + c_2 \log q + c_3 \log \frac{X}{\nu}= d$, \emph{independent of 
parameters $-2\leq \alpha\leq 0$}, $\nu>0$, 
where for the lower boundary, $c_1 = 0.692$, $c_2 = -3.46$, $c_3 = 1$, 
and $d=0.3$, and for the upper boundary, $c_1 = 0.791$, $c_2 = -1.73$, 
$c_3 = 1$, and 
$d = 6.22$: see Figure \ref{fig462}.
Values $\alpha>0$ were not computed.
}

\medskip

%
%


Together with the small-$F$ description of Theorem \ref{mains}, these 
observations give an \emph{essentially complete description} of stability of periodic 
roll wave solutions of the St. Venant equations \eqref{swe},
for $-2\leq \alpha\leq 0$.

\begin{figure}[htbp]
 \begin{center}
$
\begin{array}{lcr}
(a) \includegraphics[scale=0.24]{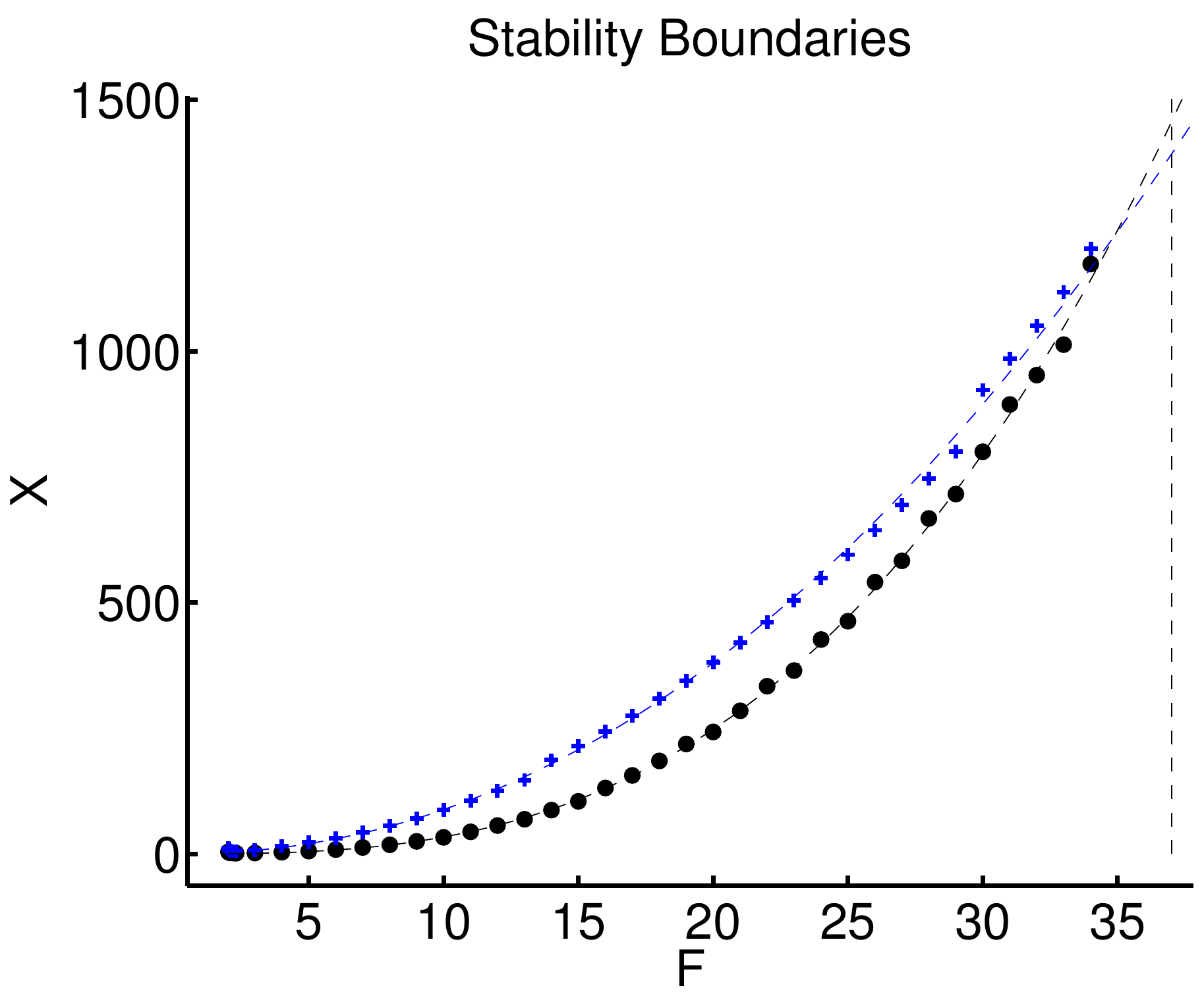}
&(b) \includegraphics[scale=0.24]{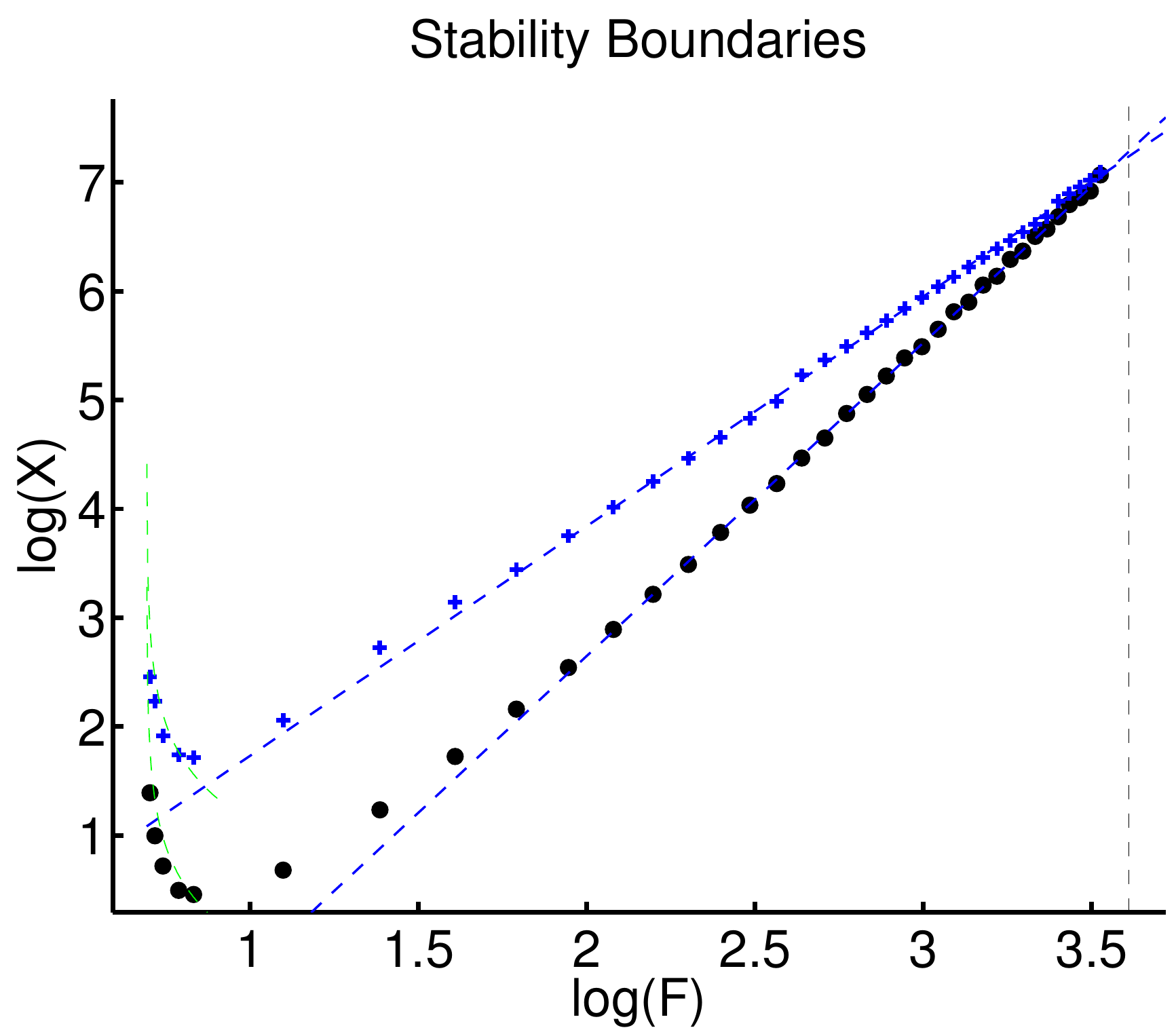}
&(c) \includegraphics[scale=0.24]{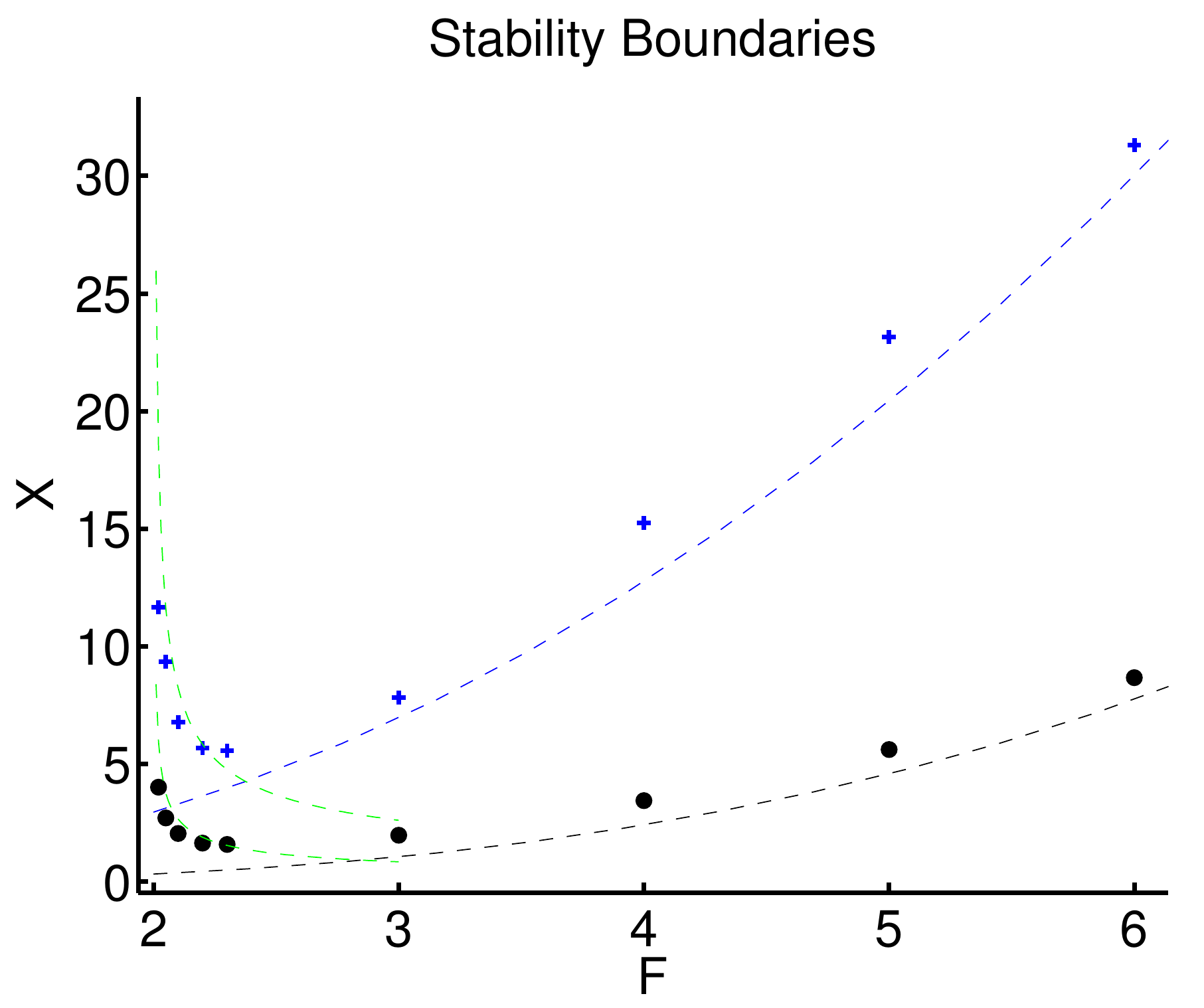}
\end{array}
$
\end{center}
\caption{
Lower and upper stability boundaries
for $\alpha=-2$, $\nu=0.1$, and, motivated by \eqref{scales}, scaling $q=0.4F^{-2}$.
Solid dots show numerically observed boundaries. 
Pale dashes 
indicate approximating curves 
given by (a) (upper) $X =  e^{0.087}F^{1.88}$ and (lower) $X = e^{-2.97}F^{2.83}$, (b) (upper) $\log(X) = 1.88\log(F)+0.087$ and (lower) $\log(X) = 2.83 \log(F)-2.97$.
Pale dotted curves (Green in color plates)
indicate theoretical boundaries as $F\to 2^+$.
(c) Small- to large-$F$ transition.
}
\label{fig476}\end{figure}

\begin{figure}[htbp]
 \begin{center}
$
\begin{array}{lcr}
(a) \includegraphics[scale=0.2]{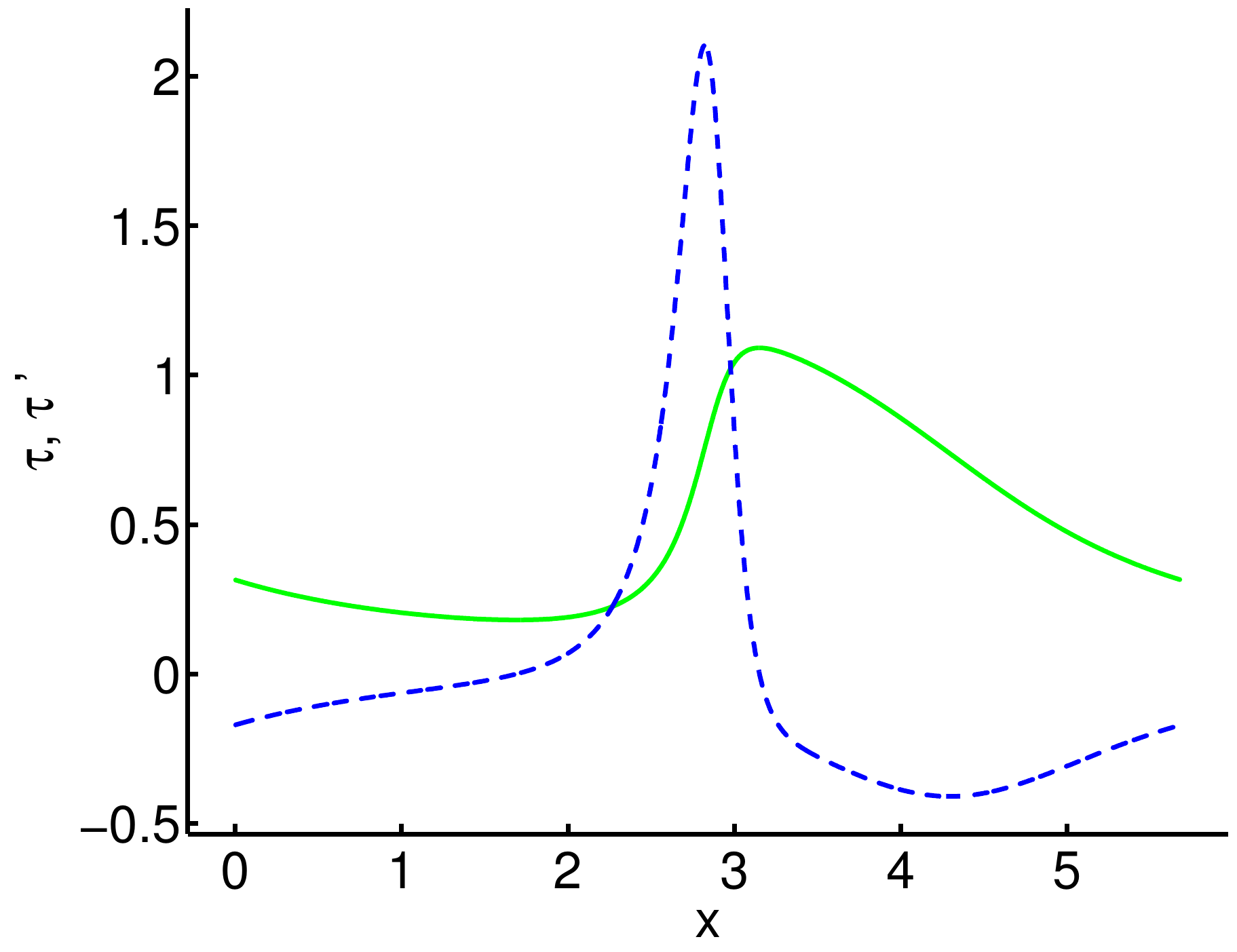} & (b)  \includegraphics[scale=0.2]{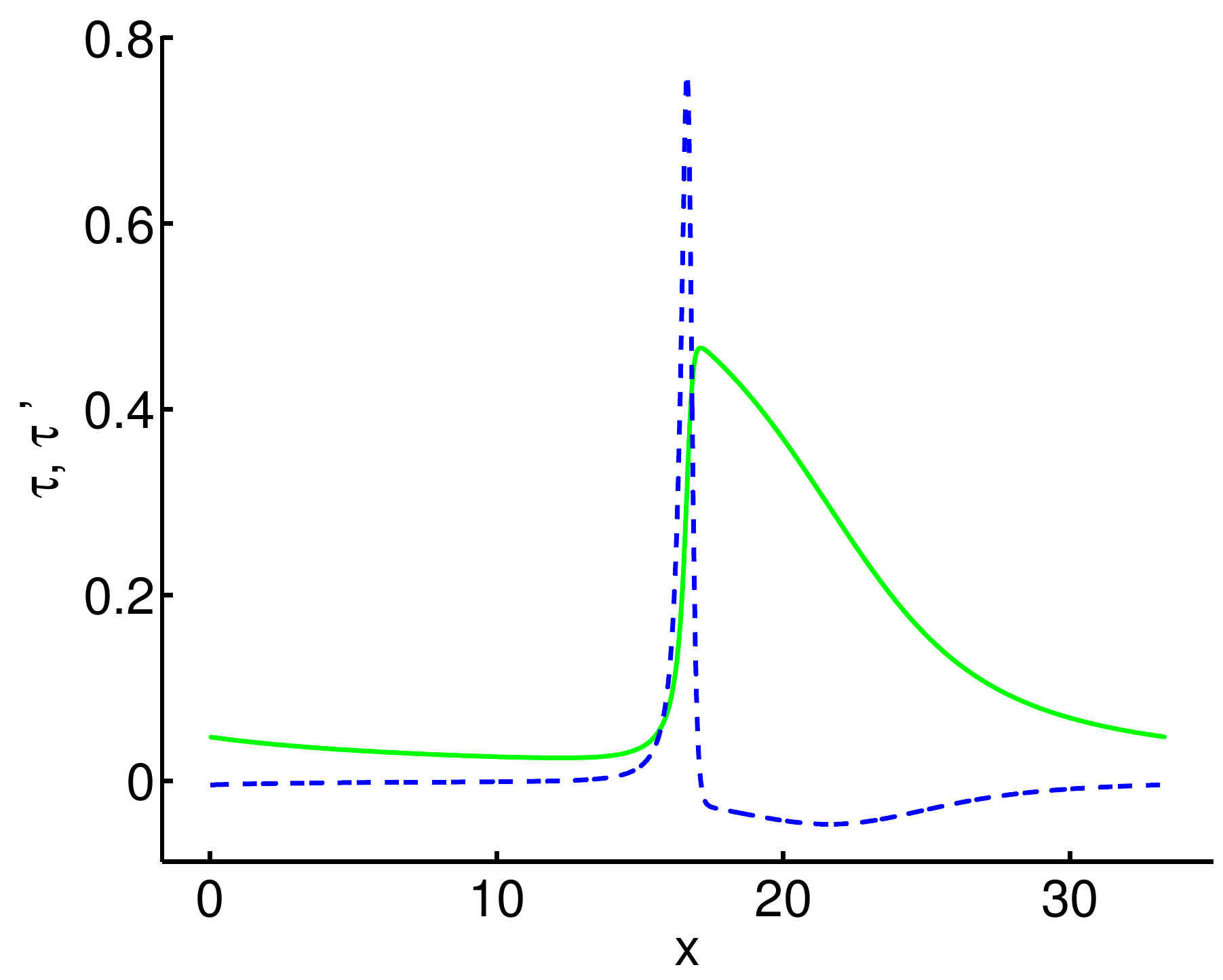} & (c) \includegraphics[scale=0.2]{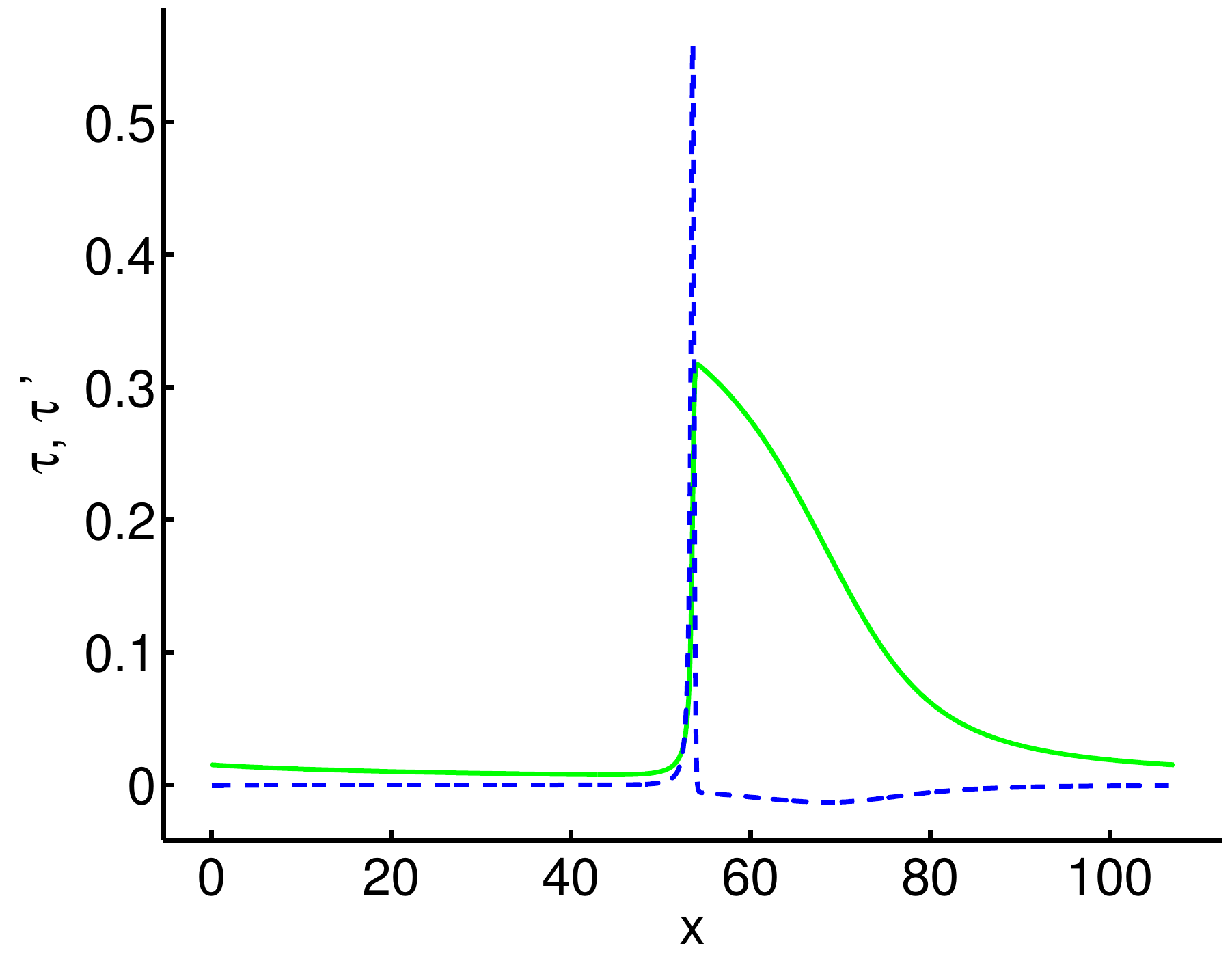}\\
(d) \includegraphics[scale=0.2]{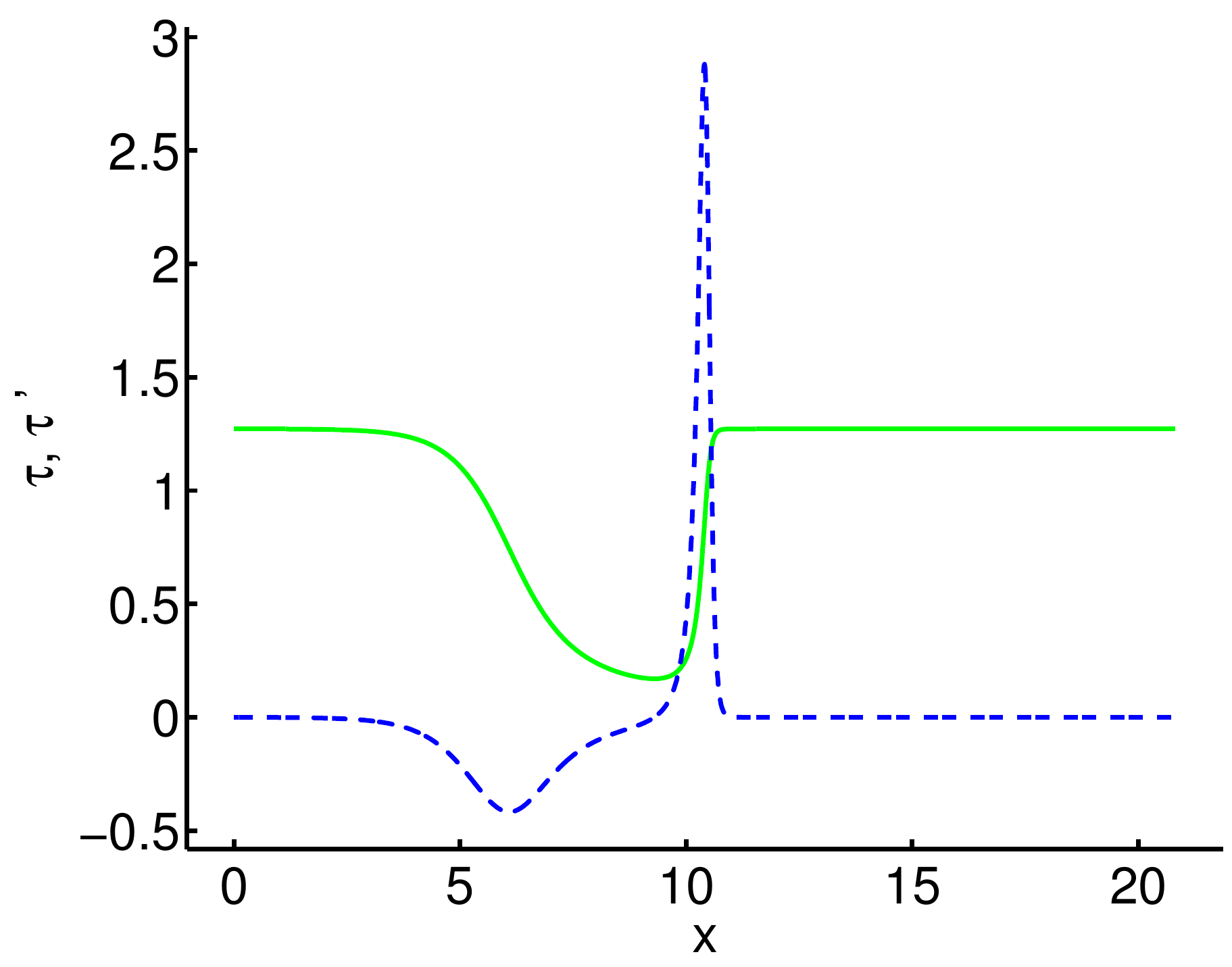} & (e) \includegraphics[scale=0.2]{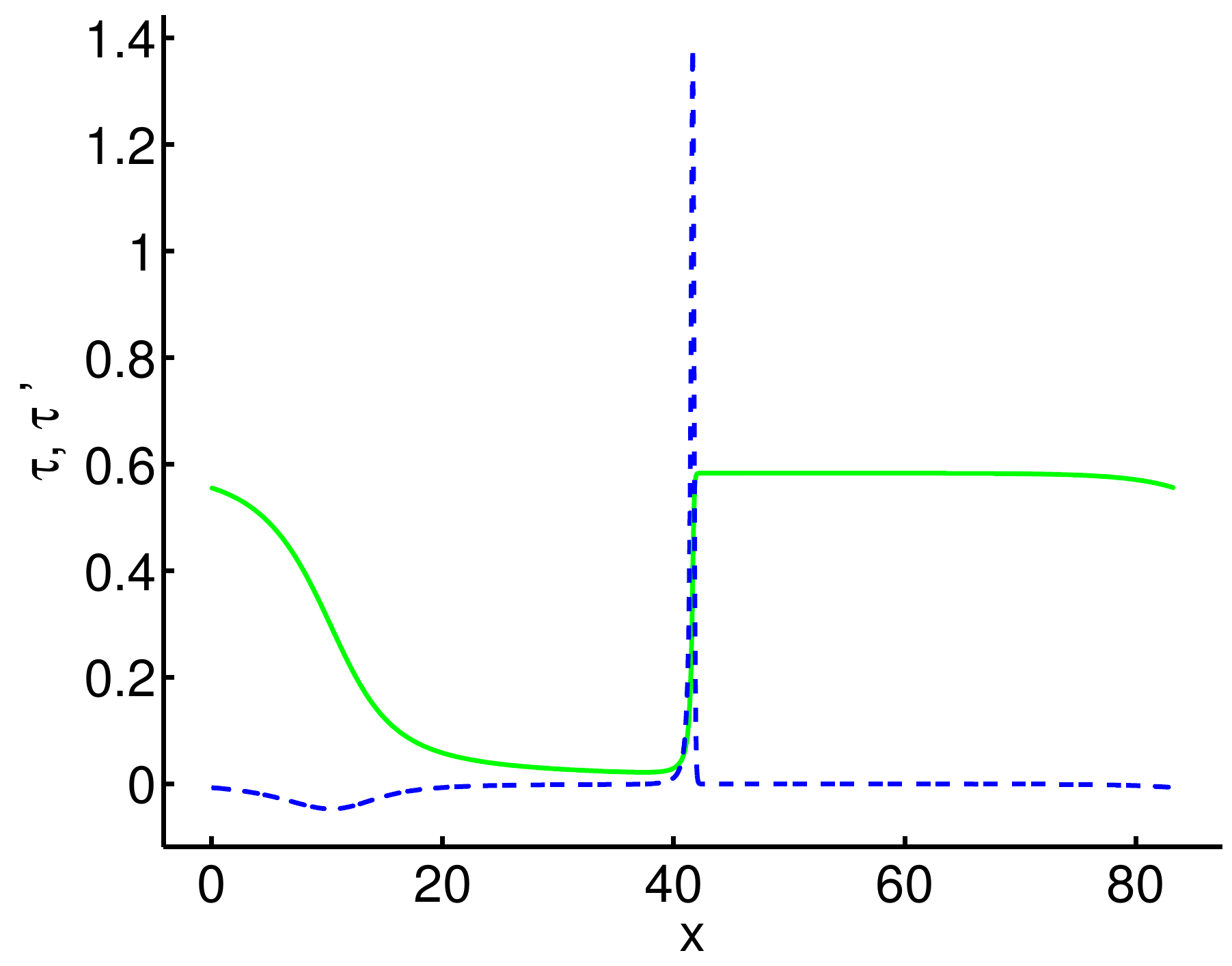} & (f) \includegraphics[scale=0.2]{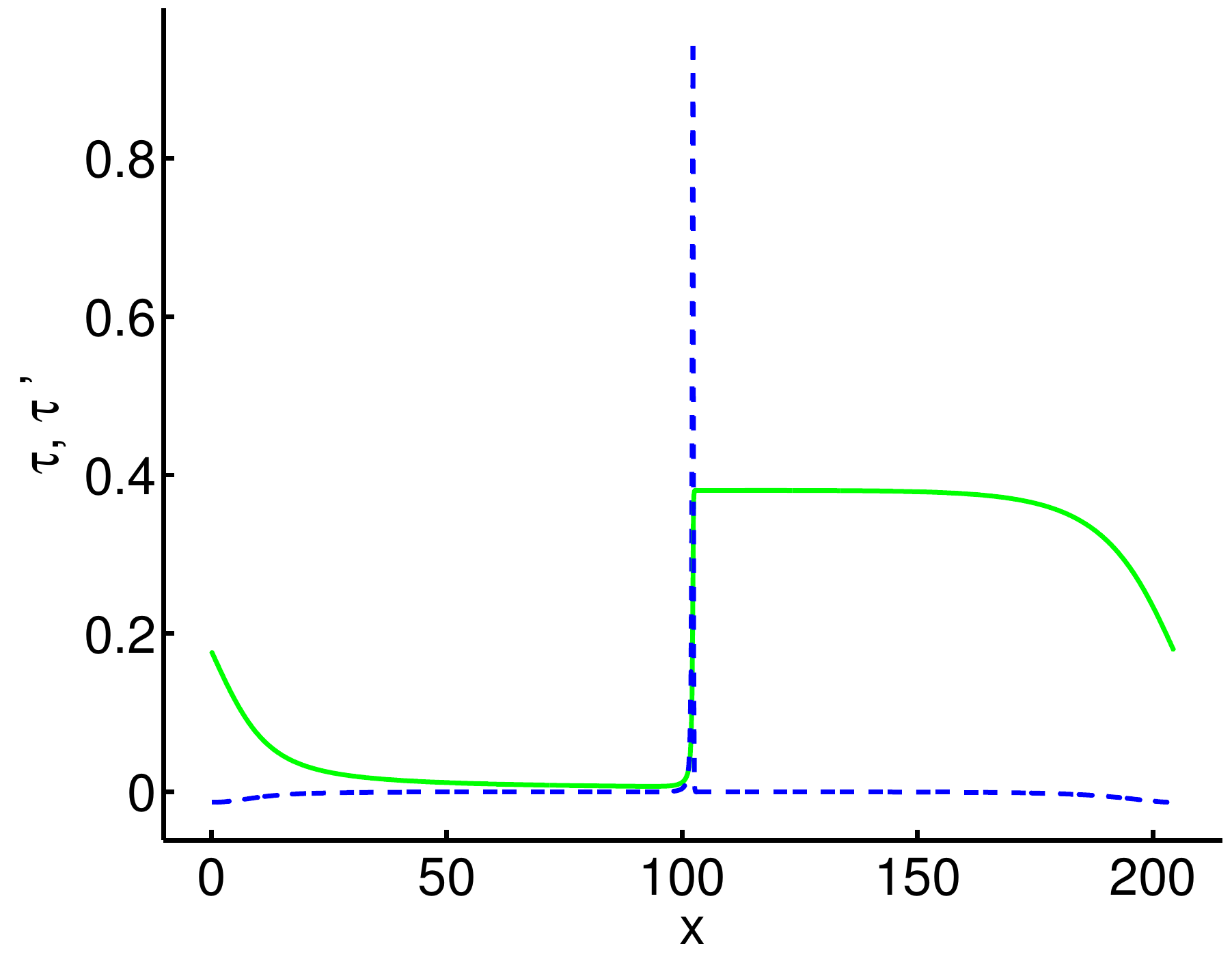}
\end{array}
$
\end{center}
\caption{Convergence to 
Dressler waves: We plot $\bar \tau(x)$ with pale solid curves 
and $\bar \tau'(x)$ with
dark dashed curves (green and blue respectively in color plates).
Here $\alpha = -2$, $\nu = 0.1$, $q=0.4F^{-2}$, 
and (a) $F = 5$, $X \approx 6.25$, (b) $F = 10$, $X \approx $ 33.3, (c) $F = 15$, $X \approx 107$, (d) $F = 5$, $X \approx 20.8 $, (e) $F = 10$, $X \approx $ 83.3, (f) $F = 15$, $X \approx $ 205.}
\label{fig39b}\end{figure}

\subsection{Discussion and open problems}\label{s:discussion}
There have been a number of numerical and analytical studies of viscous
roll waves in certain small-amplitude limits, in particular for the KdV-KS equations
governing formally the weakly unstable limit $F\to 2^+$ \cite{FST,CD,CDK,BN,EMR,PSU,BJNRZ2,JNRZ1,B}.
However, to our knowledge, the present study represents
the first systematic investigation of the
stability of {\it arbitrary amplitude} roll wave solutions of the 
viscous St. Venant equations for inclined thin film flow.

Our main mathematical contribution is the rigorous validation of the 
formal KdV-KS $\to$ KdV limit as a description of behavior in the small
Froude number/weakly unstable limit $F\to 2^+$.
This, together with the works 
%
\cite{EMR,BN,JNRZ1,B} on KdV-KS$\to$KdV, gives a complete classification of existence and stability of viscous St. Venant roll waves in the weakly unstable regime.
We note again that KS-KdV$\to$KdV
 is a canonical weakly unstable limit for the type of long-wave instabilities arising in thin film flow, in the same way that 
the (real and complex) Ginzburg-Landau equations are canonical models for 
finite-wavelength ``Turing-type'' instabilities.
However, its analysis, based on singular perturbations of periodic 
KdV solutions, is essentially different from that of
the finite-wavelength case
based on regular perturbation of constant solutions \cite{M1,M2,S2}.


From a practical point of view, the main point is perhaps 
the numerically-obtained description of behavior in  the passage from small-amplitude to large-amplitude behavior.
In particular, the universal scaling law of Numerical Observation 2 gives an 
unexpected global, simple-to-apply description of stability
that seems potentially of use in biological and engineering applications,
 for which the St. Venant equations appear to be the preferred ones in 
current use. 
(Compare with the very complicated behavior in Fig. \ref{f:islands} as $\delta$ is varied away from the small-$\delta$ limit.)
This adds new insight beyond the qualitative picture 
afforded by the canonical KdV-KS$\to$KdV limit.
In particular, our numerical results indicate a sharp transition at 
$F\approx 2.3$ from the quantitative predictions of the small-amplitude theory to the quite different large-$F$ prediction of Numerical Observation 2.
As hydraulic engineering applications typically involve values $2.5\leq F \leq 20$
\cite{Je,Br1,Br2,A,RG1,RG2,FSMA}, this distinction appears physically quite
relevant.

A very interesting open problem, both from the mathematical and engineering point of view, is to rigorously verify this numerically-observed rule of thumb.
As noted above 
(see Fig. \ref{fig476}), 
the upper and lower stability boundaries described in Numerical Observation 2 obey different scalings from those
prescribed in \eqref{scales}, 
as $F\to \infty$,
with period 
growing faster than
$X\sim F^{-1/2-5\alpha/4}$ by a factor $F^{1/2-c_1 + \alpha(c_2/2+5/4)}$
that is $\gg 1$ for $\alpha\leq \alpha_*\approx 1.54$:
in particular, for the two main physical values of interest $\alpha=-2$ and
$\alpha=0$, corresponding to constant (Eulerian) period and constant inflow,
respectively. 
Indeed, given the large values of $F$ to which the stability region 
extends, this may be deduced by Numerical Observation~1, 
which implies that all such waves of period $O(F^{-1/2-5\alpha/4})$ are 
necessarily unstable for $F\gg 1$.

An important consequence is that,
rescaling viscosity $\nu$ so that the resulting period $ X \nu$
after standard invariant scaling remains constant, we find that $\nu\to 0$. Hence, the limiting behavior of stable waves is described by the joint {\it inviscid, large-Froude number limit} $\nu\to 0$, $F\to +\infty$.
Figure \ref{fig39b}, depicting periodic profiles at the upper and lower stability boundaries for values $F=5,10,15$, clearly indicate convergence as $F$ increases to inviscid Dressler waves \cite{Dr}, alternating smooth portions
and shock discontinuities. This agrees with recent observations of \cite{BL}
that large-amplitude roll waves are experimentally well-predicted by
a simplified, asymptotic version of the inviscid theory.
We conjecture that our lower (low-frequency) stability boundary, corresponding to loss of hyperbolicity of associated Whitham equations, agrees with the inviscid threshold suggested by \cite[Theorem~1.2]{N},\footnote{
For fixed $4<F^2<90$, this states that waves are stable for 
fixed velocity/period and inclination angle $\theta$  sufficiently small: 
equivalently (by rescaling), for
fixed inclination angle and period $X$ sufficiently large.} while the upper (high-frequency) stability boundary, corresponding to appearance of unstable spectra far from the origin, arises through a 
homoclinic, or ``large-$X$,'' 
limit similar to that studied 
in \cite{G2,SS} for reaction diffusion, KdV and related equations.

We note that both 
analysis and numerics are complicated in the large-$X$ limit by the appearance, differently from the case treated in 
\cite{SS}
of essential spectra through the origin of the limiting solitary wave 
profile at $X=+\infty$, along with the usual zero eigenvalue
imposed by translational invariance.
Whereas point spectra of a solitary wave are approximated as $X\to +\infty$ by  
individual loops of Floquet spectra, curves of continuous spectra are 
``tiled'' by arcs of length $\sim X^{-1}$,
leading to the plethora of zero-eigenvalues (marked as pale dots, red in color plates)
visible in Fig. \ref{showspec}(c)--(d).
The large number of roots as $X\to \infty$ leads to numerical difficulty for both Hill's method and numerical Evans function techniques,
making the resolution of the stability region an extremely delicate computation,
requiring 40 days on IU's 370-node Quarry supercomputer cluster to complete (Appendix \ref{s:comp}).
The asymptotic analysis of this region is thus of considerable practical as well as theoretical interest.

Another very interesting open question in the large-$F$ regime
is nonlinear stability of spectrally stable wave, in the absence of slope condition \eqref{e:Eslope}, or $2\nu
\bar u_x < F^{-2}$, where $\bar u$ denotes the background profile.
This condition was used in two ways in the nonlinear stability analysis of 
\cite{JZN}: first, to carry out high-frequency resolvent estimates used to obtain linearized stability estimates, and second, to carry out ``nonlinear damping'' type estimates used to close the nonlinear iteration 
yielding nonlinear modulational stability.
In \cite{JZN}, both were carried out by closely-related energy estimates.
However, in the linear case, applying more delicate linear ODE techniques like those of Section \ref{s:est}, one
finds that the needed high-frequency resolvent bounds require
not that \eqref{e:Lslope} hold pointwise but only on average. As $\bar u$ is periodic, the resulting averaged condition is $0<F^{-2}$, 
hence always satisfied. Thus, {\it condition \eqref{e:Lslope} 
seems 
superfluous for linearized stability.}
This strongly suggests that \eqref{e:Lslope} may be dropped in the nonlinear analysis as well, either by a pseudodifferential analysis paralleling that of the linear case, or by an energy estimate with strategically-chosen (periodic) exponential weight analogously as in \cite{Z2} for viscous shocks.

Finally, recall that, just as the KdV-KS$\to$KdV limit is derived formally  from the viscous St. Venant equations in the weakly unstable limit, the viscous St. Venant equations are derived formally from the more fundamental free-boundary Navier--Stokes equations in the shallow water limit. Alternatively, the KdV-KS$\to$KdV limit may be derived directly from the Navier--Stokes equations in a formal weakly unstable/shallow water limit. Rigorous verification of this formal limit, directly from the free-boundary Navier--Stokes equations, is perhaps the fundamental open problem in the theory.

\subsection{Plan of the paper}
In Section \ref{s:existence}, we recall the formal derivation of KdV-KS 
from the viscous St. Venant system and establish Theorem \ref{maine} 
concerning the existence of small amplitude roll-waves.  
These calculations will serve as a guideline for the subsequent analysis: 
indeed, we will follow the general strategy for  the proof of spectral stability for KdV-KS periodic waves presented in \cite{JNRZ1}. 
In Section~\ref{s:est}, we begin studying the stability of these small amplitude roll-waves by computing a priori estimates on possible unstable eigenvalues for their
associated linearized (Bloch) operators: energy estimates provide natural $O(1)$ bounds (as $\delta=\sqrt{F-2}\to 0$) whereas an approximate diagonalization process is needed to obtain the sharper bound $O(\delta^{3})$. 
At this stage, we recover, after a suitable rescaling and up to some negligible terms, the spectral problem associated to KdV-KS obtained after the Fenichel's transformations. 
In Section \ref{evans_expansion}, we then follow the proof in \cite{JNRZ1} to complete the spectral stability analysis: for any fixed \emph{non-zero} Bloch number, possible unstable eigenvalues $\lambda(\xi)$ for the linearized St. Venant system as $\delta\to 0^+$ are expanded as 
$\lambda(\delta;\xi,\lambda_0)=\delta^3\lambda_0+\delta^4 \lambda_1(\xi,\lambda_0)+O(\delta^5)$, where $\lambda_0\in i\mathbb{R}$ is an explicit eigenvalue associated with the linearized (Bloch) operator for the KdV equation and the corrector $\lambda_1(\xi,\lambda_0)$ is {\it exactly} the corrector found in the analogous study of the stability of KdV-KS wavetrains in the singular limit $\delta\to 0^+$: see \cite{B,JNRZ1} and Section~\ref{s:limit} above. In particular, there it was proven (through numerical evaluation of integrals of
certain elliptic functions) that $\textrm{Ind}(k)<0$, as defined in \eqref{index} for all $k\in\mathcal{P}$ corresponding to periods
$X=X(k)$ in an open interval $(X_m,X_M)$ with $X_m\approx 8.44$ and $X_M\approx 26.1$.
On the other hand, in the regime 
$0<|\lambda|/\delta^3+|\xi|\ll 1$ 
%
a further expansion of the Evans function is needed. There, we show that 
modulo a rescaling of $\lambda$ by $\delta^3$
this expansion is exactly the one derived in \cite{JNRZ1} 
for the singular KdV limit of the KdV-KS equation.  From the results of \cite{JNRZ1}, this
concludes the proof of our description of spectral stability in the small-Froude number limit $F\to 2^+$. 
Finally, in Section~\ref{s:numinf},
we carry out a numerical analysis similar to the one in \cite{BJNRZ2} where for the KdV-KS equation the full set of model parameters was explored: here we consider the influence of $2<F<\infty$ on the range of stability of periodic waves,
as parametrized by period $X$ and discharge rate $q$.


\medskip

\noindent
{\bf Authors Note:} Since the completion of our analysis, it has been shown that the technical slope condition \eqref{e:Eslope} may indeed be dropped
as an assumption to establish the nonlinear modulational stability of diffusively stable roll-waves; see \cite{RZ}.

\section{Existence and Stability of Roll-Waves in the Limit $F\to 2^+$}\label{s:F-->2}

In this section, we rigorously analyze in the weakly unstable limit $F\to 2^+$ the spectral stability of periodic traveling wave solutions
of the St. Venant equations \eqref{swl} to small localized (i.e. integrable)
perturbations.  
We begin by studying the existence of such solutions and determining their asymptotic expansions. 
In particular, we show that such waves exist and, up to leading order,
are described by solutions of the KdV-KS equation \eqref{kdv-ks} in the singular limit $\delta\to 0$ (Theorem \ref{maine}).

\subsection{Existence of small-amplitude roll-waves: proof of Thm. \ref{maine}}\label{s:existence}

The goal of this section is to establish the result of Theorem \ref{maine}.  To begin, notice that 
traveling wave solutions of the shallow-water equations \eqref{swl} with wave speed $c$ are stationary solutions 
of the system 
\begin{equation}\label{swl:trav}
\partial_t\tau-\partial_x(u+c\tau)=0,\quad\partial_t u+\partial_x\left(\frac{\tau^{-2}}{2F^2}-c u\right)=1-\tau u^2+\nu\partial_x\left(\tau^{-2}\partial_xu\right)
\end{equation}
of PDE's.  In particular, from the first PDE it follows that $u=q-c\tau$ for some constant of integration $q\in\RM$ and hence $\tau$ must satisfy the profile ODE
\begin{equation}\label{profile}
\partial_x\left(\frac{\tau^{-2}}{2F^2}+c^2\tau\right)=1-\tau (q-c\tau)^2-c\nu\partial_x\left(\tau^{-2}\partial_x\tau\right).
\end{equation}
Clearly then, we have that $(\tau,u)=(\tau_0,\tau_0^{-1/2})$ is an equilibrium solution of \eqref{swl:trav} for any $\tau_0>0$.  Furthermore, linearizing
the profile ODE \eqref{profile} about $\tau=\tau_0$ yields, after rearranging, the ODE
\[
c\nu\tau_0^{-2}\tau''+\left(c^2-c_s^2\right)\tau'+\left(\frac{\tau_0^{-3/2}/2-c}{\tau_0^{-1/2}/2}\right)\tau=0,\quad c_s=\frac{\tau_0^{-3/2}}{F}\,.
\]
Considering the eigenvalues of the above linearized equation as being indexed by the parameters $u_0$, $c$, and $q$ it is straightforward to check
that a Hopf bifurcation occurs when
\[
c=c_s=\frac{\tau_0^{-3/2}}{F}\quad\textrm{and}\quad F>2.
\]
This verifies that as the Froude number $F$ crosses through $F=2$ the equilibrium solutions $(\tau,u)=(\tau_0,\tau_0^{-1/2})$, corresponding to a parallel flow, becomes linearly unstable through a Hopf bifurcation, and hence nontrivial periodic traveling wave solutions of \eqref{swl} exist for $F>2$. Moreover, at the bifurcation point the limiting period of such waves is given by $X=\frac{2\pi}{\omega}$ where $\omega=\tau_0^{5/4}\nu^{-1/2}\sqrt{F-2}$.

With the above preparation in mind, we want to examine the small-amplitude periodic profiles generated in the weakly-nonlinear limit $F\to 2^+$.  To this end, we set $\delta=\sqrt{F-2}$ and notice that by rescaling space and time in the KdV-like fashion 
$Y=\delta (x-c_0\,t)/\nu^{1/2}$ and $S=\delta^3 t/\nu^{1/2}$, with $c_0=\tau_0^{-3/2}/2$, 
\eqref{swl} becomes
\begin{equation}\label{swl:trav-rescale}
\delta^2\partial_S\tau-\partial_Y(u+c_0\tau)=0,\quad
   \delta^3\partial_S u+\delta\partial_Y\left(\frac{\tau^{-2}}{2F^2}-c_0 u\right)=\nu^{1/2}\left(1-\tau u^2\right)+\nu^{1/2}\delta^2\partial_Y\left(\tau^{-2}\partial_Y u\right).
\end{equation}
We now search for small-amplitude solutions of this system of the form
$
(\tau,u)=(\tau_0,\tau_0^{-1/2})+\delta^2(\tilde\tau,\tilde u)
$
with wave speed $c_0$ in the limit $F\to 2^+$.  The unknowns $\tilde\tau$ and $\tilde{u}$ satisfy the system
\begin{align*}
&\delta^2\partial_S\tilde\tau-\partial_Y(\tilde{u}+c_0\tilde{\tau})=0\\
&\delta^3\partial_S\tilde{u}+\delta\partial_Y\left(\frac{(\tau_0+\delta^2\tilde\tau)^{-2}}{2\delta^2F^2}-c_0\tilde{u}\right)\\
&\qquad=\nu^{1/2}\delta^{-2}\left(1-(\tau_0+\delta^2\tilde\tau)
   (\tau_0^{-1/2}+\delta^2\tilde{u})^2\right)+\nu^{1/2}\delta^2\partial_Y\left((\tau_0+\delta^2\tilde\tau)^{-2}
\partial_Y\tilde{u}\right).
\end{align*}
Defining
the new unknown $\tilde{w}=\delta^{-2}\left(\tilde u+c_0\tilde\tau\right)$ and inserting $\tilde{u}=-c_0\tilde\tau+\delta^2\tilde{w}$
above yields
\begin{align*}
&\partial_S\tilde\tau - \partial_Y\tilde{w}=0\\
&\delta^3\partial_S\tilde{w}+\delta^{-1}\partial_Y\left(\frac{(\tau_0+\delta^2\tilde\tau)^{-2}-\tau_0^2+2\tau_0^{-3}\delta^2\tilde\tau-3\tau_0^4\delta^4\tilde\tau^2}{2\delta^2F^2}+\left(c_0^2-\frac{\tau_0^{-3}}{F^2}\right)\tilde\tau+\frac{3\tau_0^4}{2F^2}\delta^2\tilde\tau^2-2c_0\delta^2\tilde{w}\right)\\
&\qquad=\nu^{1/2}\left((2\tau_0^{-1/2}c_0-c_0^2\tau_0)\tilde\tau^2-2\tau_0^{1/2}\tilde{w}-
c_0
\tau_0^{-2}
\partial_{YY}\tilde\tau+\delta^2\tilde g(\tilde\tau,\tilde w,\delta)\right)\\
&\qquad\qquad
+\nu^{1/2}\tau_0^{-2}\delta^2\partial_{YY}\tilde w
+\nu^{1/2}\delta^2\partial_Y\left(\tilde r(\tilde\tau,\delta)\partial_Y(\delta^2\tilde w-c_0\tilde\tau)\right)
\end{align*}
for some smooth functions $\tilde g, \tilde r$.  
Expanding $F^{-2} = \frac{1}{4}\left(1-\delta^2\right)+O(\delta^4)$ 
reduces the above system to
\begin{align*}
&\partial_S\tilde\tau-\partial_Y\tilde w=0\\
&\delta^3\partial_S\tilde w + \delta\partial_Y\left(\frac{\tilde\tau}{4\tau_0^3}+\frac{3}{8\tau_0^4}\tilde\tau^2-\tau_0^{-3/2}\tilde w+\delta^2 \tilde f(\tilde\tau,\delta)\right)\\
&\qquad=\nu^{1/2}\left(\frac{3}{4\tau_0^2}\tilde\tau^2-2\tau_0^{1/2}\tilde{w}-\frac{1}{2\tau_0^{7/2}}\partial_{YY}\tilde\tau+\delta^2 \tilde g(\tilde\tau,\tilde w,\delta)\right)\\
&\qquad\qquad
+\nu^{1/2}\tau_0^{-2}\delta^2\partial_{YY}\tilde w
+\nu^{1/2}\delta^2\partial_Y\left(\tilde r(\tilde\tau,\delta)\partial_Y(\delta^2\tilde{w}-c_0\tilde\tau)\right)
\end{align*}
for some smooth function $\tilde{f}$. Rescaling the independent and dependent variables via
$$
(Y,S,\tilde\tau,\tilde{w})\mapsto\left(\tau_0^{5/4}Y,\ \tfrac14\tau_0^{-1/4}S,\ 3\tau_0^{-1}\tilde\tau,\ 12\tau_0^{1/2}\tilde{w}\right),
$$
we arrive at the rescaled system
\begin{equation}\label{e:w_reduced}
\begin{aligned}
&\partial_S\tilde\tau-\partial_Y\tilde w=0\\
&\frac{\delta^3}{8\tau_0^{1/4}\nu^{1/2}}\partial_S\tilde w + \frac{\delta}{2\tau_0^{1/4}\nu^{1/2}}\partial_Y\left(\tilde\tau+\frac{1}{2}\tilde\tau^2-\tilde w+\delta^2 f(\tilde\tau,\delta)\right)\\
&\qquad=\frac{1}{2}\tilde\tau^2-\tilde{w}-\partial_{YY}\tilde\tau+\delta^2 g(\tilde\tau,\tilde w,\delta)
+\frac12\delta^2\partial_{YY}\tilde w
+\delta^2\partial_Y\left(r(\tilde\tau,\delta)\partial_Y(\delta^2\tilde{w}-c_0\tilde\tau)\right)
\end{aligned}
\end{equation}
for some smooth functions $f$, $g$, and $r$.

We now search for periodic traveling waves of the form $(\tilde\tau,\tilde w)(Y-\tilde cS)$ in the rescaled system \eqref{e:w_reduced}.  
Changing to the moving coordinate frame $(Y-\tilde cS,S)$, in which the $S$-derivative becomes zero, and integrating the first equation with respect to the 
new spatial variable $Y-\tilde cS$, we find that
$\tilde w=\tilde q-\tilde c\tilde\tau$ for some constant $\tilde q$. Substituting this identity into the second equation in \eqref{e:w_reduced}, also 
expressed
in the moving coordinate frame $(Y-\tilde cS,S)$, gives
\begin{align*}
&\frac{\delta^3}{8\tau_0^{1/4}\nu^{1/2}}\,\tilde c^2 
\tilde\tau'
+\frac{\delta}{2\tau_0^{1/4}\nu^{1/2}}\left((1+\tilde c)\tilde\tau+\frac{1}{2}\tilde\tau^2+\delta^2f(\tilde\tau,\delta)\right)'\\
&\qquad=-\tilde q+\frac{1}{2}\tilde\tau^2+\tilde c\tilde\tau+\delta^2G(\tilde\tau,\delta)-\left((1+\delta^2B(\tilde\tau,\delta))\tilde\tau'\right)'
\end{align*}
for some smooth functions $G$ and $B$.  Next, introducing the near-identity change of dependent variables 
\begin{equation}\label{cov}
\tilde T=-\left(\tilde\tau+\delta^2\int_0^{\tilde\tau}B(x,\delta)dx\right)
\end{equation}
gives, finally, the reduced, nondimensionalized profile equation
\begin{equation}\label{eq:red:f}
\tilde{T}''+\frac{1}{2}\tilde T^2-\tilde c\tilde T-\tilde q=-\tilde\delta\left((\tilde c+1)\tilde T-\frac{1}{2}\tilde T^2\right)'+\tilde\delta^2m(\tilde T,\tilde T',\tilde\delta),
\quad\tilde\delta:=\frac{\delta}{2\tau_0^{1/4}\nu^{1/2}}
\end{equation}
for some smooth function $m$.  

It is well known
(see \cite{BD} for instance)
that the limiting $\tilde{\delta}=0$ profile equation
\begin{equation}\label{kdvprof}
T_0''+\frac{1}{2}T_0^2-\tilde c\,T_0=\tilde q
\end{equation}
selects for a given $(\tilde c,\tilde q)$, up to translation, a one-parameter subfamily of the cnoidal waves of the KdV equation
$
u_t+uu_x+u_{xxx}=0,
$
which are given given given explicitly as a three-parameter family by 
\[
T_0(x;a_0,k,\kappa)\ =\ a_0+12k^2\kappa^2\cn^2\left(\kappa x,k\right),
\]
where $\cn(\cdot,k)$ is the Jacobi elliptic cosine function with elliptic modulus 
$k\in(0,1)$, 
$\kappa>0$ is a scaling parameter, and $a_0$ is an arbitrary real constant related to the Galilean invariance of the KdV equation,
with the parameters $(a_0,k,\kappa)$ being constrained by $(\tilde c,\tilde q)$ through the relations
\begin{align*}
\tilde c&\ =\ a_0+4\kappa^2\,(2k^2-1)\,,\\
\tilde q&\ =\ 24k^2(1-k^2)\kappa^4-a_0\,(\tfrac12a_0+4\kappa^2\,(2k^2-1))\,.
\end{align*}
Note that these cnoidal profiles are $2K(k)/\kappa$ periodic, where $K(k)$ is the complete elliptic integral of the first kind.  

Now, noting that \eqref{eq:red:f} can be written as
$$
\frac{1}{\tilde\delta}\left(\frac12(\tilde{T}')^2+\frac{1}{6}\tilde T^3-\frac{\tilde c}{2}\tilde T^2-\tilde q\tilde T\right)'=
\tilde T'\,\left(-\left(\tilde c+1)\tilde T-\frac{1}{2}\tilde T^2\right)'+\tilde\delta\,m(\tilde T,\tilde T',\tilde\delta)\right),
$$
standard arguments in the study of \emph{regular perturbations of planar Hamiltonian systems} (see, e.g., \cite[Chapter~4]{GH}) 
imply that, among the above-mentioned one-dimensional family of KdV cnoidal waves $T_0$, only those satisfying
\be\label{e:svselection}
\int_0^{2K(k)/\kappa}T_0'\left((\tilde c+1)T_0-\frac{1}{2}T_0^2\right)'dx=0
\ee
can continue for $0<\tilde\delta\ll 1$ into a family of periodic solutions of \eqref{eq:red:f} and that, further, 
simple zeros of \eqref{e:svselection} do indeed 
continue for small $\tilde\delta$ into a unique, up to translations, three-parameter family of periodic solutions of \eqref{eq:red:f}; that is, for
each fixed $0<\tilde\delta\ll 1$ we find, up to translations, a two-parameter family of periodic solutions of \eqref{eq:red:f} that may be parametrized
by $a_0$ and $k$.
%
The observation that, in the present case, the selection principle \eqref{e:svselection} indeed determines a unique wave
that is a simple zero, follows directly from the proof of 
Proposition~\ref{p:kdvsolnexpand} (see Remark~\ref{r:selection}) since equation \eqref{kdvprof} implies
$$
\int_0^{2K(k)/\kappa}T_0'\left((\tilde c+1)T_0-\frac{1}{2}T_0^2\right)'dx
\ =\ \int_0^{2K(k)/\kappa}T_0'\left(T_0+T_0''\right)'dx\ =\ -\int_0^{2K(k)/\kappa}T_0\left(T_0''+T_0''''\right)dx\,,
$$
in agreement with the KdV-KS case.
This shows that the profile expansion agrees to order $O(\tilde \delta)$ with
the KdV-KS expansion.
Indeed, further computations show that the expansion of the profile coincides 
with the KdV-KS expansion up to 
order $O(\tilde \delta^2)$.
To complete the proof of Theorem~\ref{maine} we need only observe that, instead of fixing the speed as above and letting the period vary, one may 
alternatively fix the period and vary the velocity.

\br\label{btrmk}
Viewed from a standard dynamical systems point of view, the $F\to 2^+$
limit may be recognized as a {\it Bogdanev--Takens}, or
saddle-node bifurcation; see, e.g., the corresponding 
bifurcation analysis carried out for an artificial viscosity version
of Saint Venant in \cite{HC}.  
The unfolding of a Bogdanev--Takens point proceeds, similarly as above,
by rescaling/reduction to a perturbed Hamiltonian system \cite[Section~7.3]{GH}. 
\er

\subsection{Estimate on possible unstable eigenvalues}\label{s:est}

Next, we turn to analyzing the spectral stability (to localized perturbations)\footnote{The strongest kind of spectral stability, in the sense that it implies spectral stability to co-periodic perturbations, subharmonic perturbations, side-band perturbations, etc.}
of the asymptotic profiles of the St. Venant equation constructed in Theorem~\ref{maine}.  
To begin, let $(\bar\tau_\delta,\bar u_\delta)(x-\bar c_\delta t)$ denote a periodic traveling wave solution of the viscous St. Venant equation \eqref{swl}, as given by Theorem~\ref{maine} for $\delta=\sqrt{F-2}\in(0,\delta_0)$. More explicitly, in terms of the expressions given 
in
Theorem~\ref{maine}, the periodic profiles are 
\be\label{newscale1}
\bar\tau_\delta(\theta)\ =\ \tau_0\,+\,\delta^2\tfrac{\tau_0}{3}\tilde\tau_\delta\left(\frac{\tau_0^{5/4}\delta}{\nu^{1/2}}\,\theta\right)\,,\qquad
\bar u_\delta(\theta)\ =\ u_0\,+\,\delta^2\tfrac{u_0}{6}\tilde u_\delta\left(\frac{\tau_0^{5/4}\delta}{\nu^{1/2}}\,\theta\right)
\ee
and the period $X_\delta$, wave speed $\bar c_\delta$, and constant of integration $q_\delta\equiv\bar u_\delta+\bar c_\delta\bar\tau_\delta$ 
are expressible via
\be\label{newscale2}
X_\delta\ =\ \frac{\nu^{1/2}}{\tau_0^{5/4}\delta}X\,,\qquad 
\bar c_\delta\ =\ c_0+\frac{\delta^2}{4\tau_0^{3/2}}\tilde c_\delta\,,\qquad 
q_\delta\ =\ u_0+\bar c_\delta \tau_0+\frac{\delta^2}{12\tau_0^{1/2}}\tilde q\,,
\ee
with $\tau_0, u_0$ constant, $c_0=\tau_0^{-3/2}/F$, and $\theta=x-\bar c_\delta t$.
Linearizing \eqref{swl} about $(\bar\tau_\delta,\bar u_\delta)$ in the co-moving coordinate
frame\footnote{Henceforth, we suppress the dependence of $\bar\tau$, $\bar u$ $q$, and $\bar c$ on $\delta$.} 
$(x-\bar c t,t)$ leads to the linear evolution system
\begin{equation}\label{swl:lin}
\displaystyle
\partial_t\tau-\partial_x(u+\bar c\tau)=0,\quad \partial_t u-\partial_x\left(\bar c\,u+\left(\frac{\bar\tau^{-3}}{F^2}-2\bar\tau^{-3}\bar u'\right)\tau\right)=-\bar u^2\tau-2\bar u\bar\tau+\nu\partial_x(\bar\tau^{-2}\partial_x u)
\end{equation}
governing the perturbation $(\tau,u)$ of $(\bar{\tau},\bar{u})$. 
Seeking time-exponentially dependent modes leads to the spectral problem
\begin{equation}\label{spec1}
\begin{array}{ll}
\displaystyle
(u+\bar c\tau)'=\lambda\,\tau,\\
\displaystyle
\nu(\bar\tau^{-2}u')'=(\lambda+2\bar u\bar\tau) u-\left(\left(\frac{\bar\tau^{-3}}{F^2}-2\bar\tau^{-3}\bar u'\right)\tau'+\bar c u'\right)+\left(\bar u^2-\left(\frac{\bar\tau^{-3}}{F^2}-2\bar\tau^{-3}\bar u'\right)'\right)\tau,
\end{array}
\end{equation}
where primes denote differentiation with respect to $x$.  
In particular, notice that \eqref{spec1} is an ODE spectral problem with $X_\delta$-periodic coefficients.  As described in Section \ref{s:previous} above, Floquet theory implies that the $L^2(\RM)$ spectrum associated with 
\eqref{swl:lin}
is comprised entirely of essential spectrum and can be smoothly parametrized by the discrete eigenvalues of the spectral problem \eqref{spec1} considered with the quasi-periodic boundary conditions $(\tau,u)(x+X_\delta)=e^{i\xi}(\tau,u)(x)$ for some value of the Bloch parameter
$\xi\in[-\pi/X_\delta,\pi/X_\delta)$. The underlying periodic solution $(\bar\tau,\bar u)$ is said to be 
(diffusively) spectrally stable provided conditions (D1)-(D3) introduced in the introduction hold. 
Reciprocally, 
the solution will be spectrally unstable if there exists a $\xi\in[-\pi/X_\delta,\pi/X_\delta)$ such that the associated Bloch
operator has an eigenvalue in the open right half plane.

In this section, we provide a priori estimates on the possible unstable Bloch eigenvalues of the above eigenvalue problem \eqref{spec1}.  As a first step, we
carefully examine the hyperbolic-parabolic structure of the eigenvalue problem and demonstrate that, as $F\to 2^+$ or, equivalently, as $\delta\to 0^+$, the unstable Bloch eigenvalues of this system are $O(1)$.  Next, we prove a simple consistent splitting result that establishes all unstable Bloch eigenvalues of \eqref{spec1} converge to zero as $\delta\to 0^+$.  We can then bootstrap these estimates to perform a more refined analysis of the eigenvalue problem 
demonstrating
that such unstable eigenvalues are necessarily $O(\delta^3)$ as $\delta\to 0^+$.

\subsubsection{Unstable eigenvalues converge to zero as $\delta\to 0^+$}\label{s:cvgzero}

We begin by ruling out the existence of sufficiently large unstable eigenvalues for \eqref{spec1}.  
Setting $Z:=(\tau,u,\bar{\tau}^{-2}u')^T$, and recalling that $\bar{u}=q-\bar c\bar\tau$ for some constant $q\in\RM$, we first write \eqref{spec1} as a first order system 
\begin{equation}\label{evans1}
Z'(x)=A(x,\lambda)Z(x),
\end{equation} 
where
\begin{equation}\label{Amatrix}
A(x,\lambda):=
\left(
\begin{array}{ccc}
\lambda/\bar c & 0 & -\bar\tau^2/\bar c\\
0 & 0 & \bar\tau^2\\
\frac{(q-\bar c\bar\tau)^2-\bar\alpha_x-\bar\alpha\lambda/\bar c}{\nu} & 
       \frac{\lambda+2\bar\tau(q-\bar c\bar\tau)}{\nu} & \frac{-\bar c\bar\tau^2+\bar\alpha \tau^2/\bar c}{\nu}
\end{array}
\right),
\qquad
\bar \alpha:=\bar\tau^{-3}(F^{-2}+2\bar c\nu\bar\tau').  
\end{equation}
Setting 
$
B(x,\lambda):=
\left(
\begin{array}{ccc}
\lambda/\bar c & 0 & 0\\
0 & 0 & \bar\tau^2\\
-\frac{\bar\alpha\lambda/\bar c}{\nu} & 
       \frac{\lambda}{\nu} & 0
\end{array}
\right)
$
and noting that $A-B$ is $O(1)$ as $|\lambda|\to\infty$, we expect that the spectral problem \eqref{evans1} is governed by the principal part $B(x,\lambda)$
for $|\lambda|$ sufficiently large.  
A
direct inspection shows that the eigenvalues of $B$ are given by $\frac{\lambda}{\bar c}$
and $\pm\bar\tau\sqrt{\frac{\lambda}{\nu}}$, so that the eigenvalues of $B$ have two principal growth rates as $|\lambda|\to\infty$.  In the following
we keep track of both of these spectral scales by a series of carefully chosen coordinate transformations preserving periodicity; for 
details of these transformations, see Section 4.1 of \cite{BJRZ}. 

With the above preliminaries, we begin by verifying that the unstable spectra for the system \eqref{spec1} 
are $O(1)$ for $\delta$ sufficiently small. 
Throughout, we use the notation $\|u\|^2=\int_{0}^{X_\delta} |u(x)|^2dx$.  
Note that although we focus on uniformity in $\delta$ in the forthcoming estimates of the unstable spectra, 
the norms $\|\cdot\|$ do depend on $\delta$ through the period $X_\delta$.

\begin{lemma}\label{lemma1}
Let $(\bar\tau_\delta,\bar u_\delta)$ be a family of periodic traveling wave solution of \eqref{swl}
defined as in Theorem~\ref{maine} for all $\delta=\sqrt{F-2}\in(0,\delta_0)$ for some $\delta_0>0$ sufficiently small.
Then, there exist 
constants $R_0,\eta>0$ and $0<\delta_1<\delta_0$ 
such that, for all $\delta\in(0,\delta_1)$, the spectral
problem \eqref{spec1} has no $L^\infty(\RM)$ eigenvalues with $\Re(\lambda)\geq-\eta$ and $|\lambda|\geq R_0$.
\end{lemma}

\begin{proof}
Suppose that $\lambda$ is an $L^\infty(\RM)$ eigenvalue for the spectral problem \eqref{spec1} and let 
$(\tau,u)$ be a corresponding eigenfunction satisfying $(u,\tau,\tau')(X_\delta)=e^{i\xi}(u,\tau,\tau')(0)$ for some $\xi\in[-\pi/X_\delta,\pi/X_\delta)$.
As described above, setting $Z:=(\tau,u,\bar{\tau}^{-2}u')^T$ allows us to write  \eqref{spec1} as 
the first order system \eqref{evans1}, where the coefficient matrix $A(x,\lambda)$ is given explicitly in \eqref{Amatrix}.
By performing a series of $X_\delta$-periodic change of variables, carried out in detail in \cite[Section 4.1]{BJRZ}, we find there exists a $X_\delta$-periodic change of variables $W(\cdot)=P(\cdot;\lambda,\delta)Z(\cdot)$ 
that transforms the above spectral problem into 
\be\label{Dequation}
W'(x)=(D_{\lambda}+N)(x,\lambda)W(x),
\ee
supplemented with the boundary condition $W(X_\delta)=e^{i\xi}\,W(0)$, 
where the matrices $D_\lambda, N$ are defined as
\be\label{Dmatrix}
\displaystyle
D_\lambda(\cdot,\lambda)={\rm diag}\left(\frac{\lambda}{\bar c}+\theta_0+\frac{\theta_1}{\lambda},\quad \bar \tau\sqrt{\frac{\lambda}{\nu}},\quad - \bar \tau\sqrt{\frac{\lambda}{\nu}}\right),
\qquad
\theta_0  = \frac{\bar\alpha\bar\tau^2}{\bar c\,\nu},
\ee
$\bar \alpha$ as in \eqref{Amatrix}, and
	$
N:=\left(
\begin{array}{cc}
0&N_{H,D}\\
N_{D,H}&N_{D,D}\\
\end{array}
\right)
$
with $N_{D,D}$ a $2\times2$ matrix. 
Here,
$N(\cdot,\lambda)$ is an $X_\delta$-periodic matrix and, moreover, the individual blocks of the matrix $N(\cdot,\lambda)$ expand as
\begin{equation}\label{Nmatrix}
\begin{array}{rcl}
N_{D,D}(\cdot,\lambda)&=&N_{D,D}^0+\lambda^{-\frac12}N_{D,D}^1+\lambda^{-1}N_{D,D}^2\\
N_{D,H}(\cdot,\lambda)&=&N_{D,H}^0+\lambda^{-\frac12}N_{D,H}^1\\
N_{H,D}(\cdot,\lambda)&=&N_{H,D}^0+\lambda^{-\frac12}N_{H,D}^1+\lambda^{-1}N_{H,D}^2+\lambda^{-\frac32}N_{H,D}^3\ ,
\end{array}
\end{equation}
with $|N^j_{*,*}|$ bounded uniformly
in $\delta\ll 1$.
Explicit formulae for 
$N^j_{k,l}$ and $\theta_{1}$ are given in Appendix \ref{s:HFB}.

Now, a crucial observation is that, by Theorem \ref{maine}, 
together with the scalings 
\eqref{newscale1}--\eqref{newscale2},
we have
\[
\lim_{\delta\to 0^+}\theta_0(x)\ =\ \frac{\tau_0^{1/2}}{2\nu}>0.
\]
It follows that $\theta_0$ is strictly positive and uniformly bounded away from zero, for all $\delta>0$ sufficiently small.
Thus, there exists $\eta,\delta_1>0$ sufficiently small and $R_0>0$ sufficiently large such that if $|\lambda|\geq R_0$ and 
$\Re(\lambda)\geq-\eta$, then the quantity $\Re(\frac{\lambda}{\bar c}+\theta_0+\frac{\theta_1}{\lambda})$ 
is strictly positive and bounded away from zero, uniformly in $\delta$ for $0<\delta<\delta_1$.
Likewise, 
by choosing $\delta_1$ smaller if necessary, 
the quantity $|\lambda|^{1/2}\Re(\bar \tau\sqrt{\frac{\lambda}{\nu}})$ may be taken to be strictly positive and bounded away from $0$, uniformly in $\delta$ in the same set of parameters.

Finally, under the same conditions, taking $R_0$ possibly larger 
and decomposing 
$$
W:=(W_H,W_{D,+},W_{D,-})^T,
$$
observing that $|W_H|$, $|W_{D,+}|$ and $|W_{D,-}|$ are $X_\delta$-periodic functions, 
it follows by standard energy estimates,
taking the real part of the complex $L^2[0,X_\delta]$-inner product of each $W_j$ against the 
$W_j$-coordinate of \eqref{Dequation}--\eqref{Nmatrix}, 
using the above-demonstrated 
coercivity (nonvanishing real part) 
of the entries of the leading-order diagonal term $D_\lambda$, 
and rearranging,
that 
there exists a constant $C>0$ independent of $R_0$ such that, for all $0<\delta<\delta_1$, $\Re(\lambda)\geq-\eta$ and $|\lambda|\geq R_0$, we have
\[
\|W_H\|^2 \leq C\left(\|W_{D,+}\|+\|W_{D,-}\|\right)\|W_H\|
\textrm{ and }
\|W_{D,+}\|^2+\|W_{D,-}\|^2 \leq
CR_0^{-1/2}\left(\|W_{D,+}\|+\|W_{D,-}\|\right) \|W_H\|\,.
\]
Thus,
$\|W_H\| \leq C\left(\|W_{D,+}\|+\|W_{D,-}\|\right)\leq C^2R_0^{-1/2} \|W_H\|$,
	yielding a contradiction for $R_0$ sufficiently large.
\end{proof}

Next, we rescale the spatial variable $x$ as $y=\delta\,x$, noting that, since $\delta X_\delta\ =\ \nu^{1/2}\tau_0^{-5/4}X$, the period  is then independent of $\delta$. Then the first-order system \eqref{evans1} can be rewritten as 
\begin{equation}\label{zeqn}
\delta\,Z'(y) =\ \left(A_0(\lambda)\,+\,\delta^2\,A_1(y;\lambda,\delta)\right)\ Z(y)
\end{equation}
coupled with the boundary condition $Z(\delta\,X_\delta)=e^{i\xi}\,Z(0)$ for some $\xi\in[-\pi/\delta X_\delta,\pi/\delta X_\delta)$, where
\[
A_0(\lambda)
= 
\left(\begin{array}{ccc}
\displaystyle 2\lambda\tau_0^{3/2} & 
\displaystyle -2\lambda \tau_0^{3/2} & 0\\
0 & 0 &\displaystyle  \tau_0^2\,\nu^{-1}\\
\displaystyle \lambda-2\sqrt{\tau_0} &
\displaystyle  2\lambda & 0
\end{array}\right)
\]
is constant
and $A_1(\cdot;\lambda,\delta)$ is uniformly bounded (for $\lambda$ and $\delta$ in any compact set).  More precisely, for any $\delta$
in a compact subset of $[0,\delta_0)$, we have
$$
A_1(\cdot;\lambda,\delta)= \begin{pmatrix}\mathcal O(\lambda)&\mathcal O(\lambda)&0\\ 0&0&\mathcal O(1)\\ \mathcal O(1)&\mathcal O(1)&\mathcal O(1)\end{pmatrix}.
$$
By analyzing  the eigenvalues of $A_0(\lambda)$, we now show that the possible unstable eigenvalues for \eqref{spec1} converge to the origin
as $\delta\to 0^+$.

\begin{lemma}\label{lemma2}
Let $(\bar\tau_\delta,\bar u_\delta)$ be a family of periodic traveling wave solution of \eqref{swl}
defined as in Theorem~\ref{maine} for all $\delta=\sqrt{F-2}\in(0,\delta_0)$ for some $\delta_0>0$ sufficiently small.
%
%
Then, for
every $\eps>0$, there exists a $\delta_1\in(0,\delta_0)$ such that for all $\delta\in(0,\delta_1)$,  
the spectral problem \eqref{spec1} has no $L^\infty(\RM)$ eigenvalues with $\Re(\lambda)\geq 0$ and $|\lambda|\geq\eps$.
\end{lemma}

\begin{proof}
By Lemma \ref{lemma1}, it is sufficient to consider $\lambda$ on a compact set
$\eps \leq |\lambda|\leq R_0$, $\Re \lambda \geq 0$, 
whence \eqref{zeqn}, $\delta\to 0^+$ represents
a uniform family of semiclassical limit problems, with $A_0$, $A_1$ varying
in compact sets.
By standard WKB-type estimates (see, e.g., the ``Tracking Lemma'' of \cite{GZ,ZH,PZ}), these have no bounded solutions for $0<\delta\leq \delta_0$ sufficiently small, so long as $A_0$ satisfies
{\it consistent splitting}, meaning that its eigenvalues have nowhere-vanishing
real parts: equivalently, $A_0(\lambda)$ has no purely imaginary eigenvalue for $\Re \lambda>0$ and $\eps\leq |\lambda|\leq R_0$.
Indeed, it is easy to see using convergence as $\delta\to 0$ of the associated periodic Evans function of Gardner \cite{G} (following from
continuous dependence on parameters of solutions of ODE), 
the correspondence between bounded solutions and
zeros of the Evans function,
and analyticity of the Evans function together with
properties of limits of analytic functions,
that $A_0(\lambda)$ has a pure imaginary eigenvalue, i.e., the $\delta=0$ version of \eqref{zeqn} has a bounded solution,
if and only if there are bounded solutions of \eqref{zeqn} for a sequence $\lambda_\delta\to \lambda$ as $\delta\to 0$.

To prove the lemma therefore, we establish consistent splitting
of $A_0$ for $\lambda \in \Lambda:=
\{\lambda: \eps\leq |\lambda|, \Re \lambda>0\}$.
The eigenvalues $\gamma(\lambda)$ of the matrix $A_0(\lambda)$ are the solutions of the equation
\begin{equation}\label{charA0}
\gamma^3-2\lambda\tau_0^{3/2}\gamma^2-2\lambda\tau_0^2\nu^{-1}\gamma+2\lambda^2\tau_0^{7/2}\nu^{-1}+4\lambda\tau_0^4\nu^{-1}=0.
\end{equation}
Suppose that $\gamma=i\Omega\in\RM i$ is an eigenvalue of $A_0(\lambda)$ for 
some $\lambda\in\Lambda$.  From \eqref{charA0} it follows
that $\lambda$ must be a root of the quadratic equation
\begin{equation}\label{charA02}
2\lambda^2\tau_0^{7/2}\nu^{-1}+2\lambda\left(2\tau_0^4\nu^{-1}+\Omega^2\tau_0^{3/2}-2i\Omega\tau_0^2\nu^{-1}\right)-i\Omega^3=0.
\end{equation}
By Lemma \ref{lemma1} above, if $\Omega$ is sufficiently large then 
the roots of \eqref{charA02} satisfy $\Re(\lambda)\leq 0$, else, by the discussion surrounding the Evans function, above, there would be bounded solutions of
\eqref{zeqn} for $\Re \lambda>0$, $|\lambda|$ large, and $\delta$ arbitrarily small, a contradiction.
(Alternatively, one may repeat the steps of the proof of Lemma \ref{lemma1} 
for \eqref{zeqn} with $\delta=0$.)
Increasing $\Omega$ then from the supposed value corresponding to an eigenvalue of $A_\lambda$, and tracking the corresponding root $\lambda$ of
\eqref{charA02}, we see that eventually this root must cross the imaginary axis in moving from $\Re \lambda \geq 0$ to $\Re \lambda \leq 0$.
Thus, it is sufficient to search for 
roots of \eqref{charA02} of the form $\lambda=i\Theta$ for some $\Theta\in\RM$.  Substituting this ansatz into \eqref{charA02}
and grouping real and imaginary parts implies that $\Omega$ and $\Theta$ satisfy the system of equations
$\Theta\,(\tau_0^{3/2}\Theta-\Omega)=0$ and 
$2\tau_0^{3/2} \Theta=\Omega\times\Omega^2/[\Omega^2+2\nu^{-1}\tau_0^{3/2}]$, 
from which it easily follows that $\Omega=\Theta=0$.  It follows that for all $\Omega\neq0$, the real parts of the roots $\lambda_j(\Omega)$ of \eqref{charA02}
have constant signs, so that, for each $\eps>0$, $A_0(\lambda)$ indeed has consistent splitting in the region $\Lambda$, and
the lemma immediately follows.

\end{proof}

\subsubsection{Unstable eigenvalues are $O(\delta^3)$}
Next, we bootstrap the estimates of Lemma \ref{lemma1} and Lemma \ref{lemma2} to provide a second energy estimate on the reduced ``slow", or KdV, block of the spectral
problem \eqref{spec1} in the limit $\delta\to 0^+$.  
Notice that {\it this result relies heavily on the fact that the corresponding spectral problem for the linearized KdV equation about a cnoidal wave $T_0(\cdot;a_0,k,\mathcal{G}(k))$ described in Theorem~\ref{kdvks-evexpand} has been explicitly solved in \cite{BD,Sp} using the associated completely integrable structure,} and in particular has been found to be spectrally stable\footnote{In the 
Hamiltonian sense, meaning the linearization about $T_0$ of the KdV equation 
has purely imaginary spectrum.}
for all $k\in(0,1)$.

\begin{proposition}\label{Prop1}
Let $(\bar\tau_\delta,\bar u_\delta)$ be a family of periodic traveling wave solution of \eqref{swl}
defined as in Theorem~\ref{maine} for all $\delta=\sqrt{F-2}\in(0,\delta_0)$ for some $\delta_0>0$ sufficiently small.\\
Then there exist positive constants $C_1$, $C_2$ and $\delta_1\in(0,\delta_0)$ such that for all $\delta\in(0,\delta_1)$ the spectral
problem \eqref{spec1} has no $L^\infty(\RM)$ eigenvalues with 
$\Re(\lambda)\geq C_1\delta^4$ or ($\Re(\lambda)\geq 0$ and $|\lambda|\geq C_2\delta^3$).
\end{proposition}

\begin{proof}
By Lemma \ref{lemma1} and Lemma \ref{lemma2}, the possible unstable eigenvalues for the spectral problem \eqref{spec1} converge to the origin as $\delta\to 0^+$.  To analyze the behavior of these possible unstable eigenvalues further, 
we rescale the unknown $Z$ in \eqref{zeqn} to $W=(Z_1,\lambda^{2/3}Z_2,\lambda^{1/3}Z_3)^T$, $\lambda^{1/3}$ denoting the principle third root, yielding the system 
$$
\delta\,W'(y)= \left(B_0(\lambda)\,+\,\delta^2\,\lambda^{-1/3}\,B_1(y;\lambda,\delta)\right)\ W(y)
$$
with boundary conditions $W(\delta~X_\delta)=e^{i\xi}W(0)$ for some $\xi\in[-\pi/\delta X_\delta,\pi/\delta X_\delta)$, where 
$$
B_0(\lambda) =
\lambda^{1/3} 
\begin{pmatrix}
2\lambda^{2/3}\tau_0^{3/2} & 
-2 \tau_0^{3/2} & 0\\
0 & 0 &\tau_0^2\,\nu^{-1}\\
\lambda-2\sqrt{\tau_0} &
2\lambda^{1/3} & 0
\end{pmatrix}
$$
and $B_1(\cdot;\lambda,\delta)$ is uniformly bounded for $(\lambda,\delta)$ in any compact subset of $\CM\times[0,\delta_0)$.  More precisely, for any $\delta$ in a compact subset of $[0,\delta_0)$, we have
\[
B_1(\cdot;\lambda,\delta) = \begin{pmatrix}O(\lambda^{2/3})&O(\lambda^{2/3})&0\\ 0&0&O(\lambda^{2/3})\\ O(\lambda^{2/3})&O(1)&O(\lambda^{1/3})\end{pmatrix}\,.
\]
To track the most dangerous terms we write the above system as
\begin{equation}\label{evans2}
\delta\,W'(y) =\left(\lambda^{1/3}\,M_0\,+\,(2\lambda^{2/3}\,+\,\delta^2\,\lambda^{-1/3}\,\beta(y))\,M_1
\,+\,\lambda\,R_0(\lambda)\,+\,\delta^2\,R_1(y;\lambda,\delta)\right)\ W(y),
\end{equation}
where 
$$
M_0 = \begin{pmatrix}
0& -2 \tau_0^{3/2} & 0\\
0 & 0 &\tau_0^2\,\nu^{-1}\\
-2\sqrt{\tau_0} &0 & 0
\end{pmatrix},
\,\qquad 
M_1 = 
\begin{pmatrix}
0 & 0 & 0\\
0 & 0 &0\\
0 & 1 & 0
\end{pmatrix},
$$
and the function $\beta(\cdot)$ is some explicit periodic function expressed in terms of $\tau_0$, $\nu$ and asymptotic KdV profiles, 
while the $R_j$ matrices are uniformly bounded functions of $\lambda$ and $\delta$ on compact subsets of $\CM\times[0,\delta_0)$.  The goal is to now reduce the first-order problem \eqref{evans2} to a constant coefficient problem at a sufficiently high order in $\lambda$ and $\delta$.

To begin, we diagonalize $M_0$ by defining the matrices
$$
P_0 = \begin{pmatrix}
1 & 0 & 0\\
0 & \frac{K_0}{2 \tau_0^{3/2}} &0\\
0 & 0 & -\frac{\nu\,K_0^2}{2 \tau_0^{7/2}}
\end{pmatrix}~~
\begin{pmatrix}
1 & 1 & 1\\
1 & \omega &\omega^2\\
1 & \omega^2 & \omega
\end{pmatrix},
\quad
P_0^{-1}=\frac13\begin{pmatrix}
1 & 1 & 1\\
1 & \omega^2 &\omega\\
1 & \omega & \omega^2
\end{pmatrix}~~
\begin{pmatrix}
1 & 0 & 0\\
0 & \frac{2 \tau_0^{3/2}}{K_0} &0\\
0 & 0 & -\frac{2 \tau_0^{7/2}}{\nu\,K_0^2}
\end{pmatrix},
$$
where $K_0=4^{1/3}\tau_0^{4/3}\nu^{-1/3}$ and $\omega=e^{2i\pi/3}$.  
Setting $Y(\cdot)=P_0^{-1}W(\cdot)$, the system \eqref{evans2} can be written in equivalent form 
\begin{equation}\label{evans3}
\delta\,Y'(y) = \left(\lambda^{1/3}\,D_0\,-\frac{\tau_0^2}{3\nu K_0}\,(2\lambda^{2/3}\,+\,\delta^2\,\lambda^{-1/3}\,\beta(y))\,Q_1
\,+\,\lambda\,\tilde R_0(\lambda)\,+\,\delta^2\,\tilde R_1(y;\lambda,\delta)\right)\ Y(y)
\end{equation}
where
$$
D_0\ =\ \begin{pmatrix}
-K_0& 0 & 0\\
0 & -K_0\omega &\\
0 &0 & -K_0\omega^2
\end{pmatrix}\,
\qquad 
Q_1\ =\ 
\begin{pmatrix}
1 & \omega & \omega^2\\
\omega & \omega^2 &1\\
\omega^2 & 1 & \omega
\end{pmatrix},
$$
and $\tilde R_j$ are uniformly bounded in $(\lambda,\delta)$ on compact subsets of $\CM\times[0,\delta_0)$.  This effectively
diagonalizes system \eqref{evans2} to leading order.

Aiming at reducing \eqref{evans3} to a constant-coefficient problem at a 
higher order, 
we choose $q(\cdot)$ satisfying 
\[
q'(\cdot)=\tfrac{\tau_0^2}{3\nu K_0}(\beta(\cdot)-\langle\beta(\cdot)\rangle),
\]
where $\langle\cdot\rangle$ denotes average over one period.  Notice that the periodicity of $\beta(\cdot)$ implies that $q$ is $(\delta~X_\delta)$-periodic. 
Now, since the matrix $(\I_3+\delta^2\lambda^{-1/3}q(\cdot)Q_1)$ is invertible for $\delta^2\lambda^{-1/3}$ sufficiently small,
for such parameters we can make the change of variables 
$$
U(\cdot) =(\I_3+\delta^2\lambda^{-1/3}\,q(\cdot)\,Q_1)\ Y(\cdot),
$$ 
with $U(\cdot)$ satisfying the first-order system
\[
\begin{array}{rcl}
\displaystyle
\delta\,U'(y)&=&\displaystyle
\big(\lambda^{1/3}\,D_0\,-\frac{\tau_0^2}{3\nu K_0}\,(2\lambda^{2/3}\,+\,\delta^2\,\lambda^{-1/3}\,\langle\beta(\cdot)\rangle)\,Q_1
\\
&&\displaystyle
+\,\lambda\,\bar R_0(y;\lambda,\delta)\,+\,\delta^2\,\bar R_1(y;\lambda,\delta)\,+\,\delta^4\,\lambda^{-2/3}\,\bar R_2(y;\lambda,\delta)\big)\ U(y),
\end{array}
\]
with $\bar R_j$ uniformly bounded in $(\lambda,\delta)$ on compact subsets of $\CM\times[0,\delta_0)$.

Next, we diagonalize the system at a higher order. Since $D_0$ has distinct eigenvalues, we may choose a constant matrix $P_1$ such that
the commutator  $[P_1,D_0]$ equals the off-diagonal part of $Q_1$.  Then provided that $\lambda$ and $\delta^2\lambda^{-2/3}$ 
are small enough we may change the unknown to 
$$
S(\cdot) =\left(\I_3+\,\frac{\tau_0^2}{3\nu K_0}\,(2\lambda^{1/3}\,+\,\delta^2\,\lambda^{-2/3}\,\langle\beta(\cdot)\rangle)\,P_1\right)\ U(\cdot)
$$ 
with $S(\cdot)$ satisfying the first-order system
\begin{equation}\label{evans4}
\delta\,S'(y) = \left(\,D_1(\lambda,\delta)\,
+\,\lambda\,\hat R_0(y;\lambda,\delta)\,+\,\delta^2\,\hat R_1(y;\lambda,\delta)\,+\,\delta^4\,\lambda^{-1}\,\hat R_2(y;\lambda,\delta)\right)\ S(y)
\end{equation}
where $\hat R_j$ are uniformly bounded in $(\lambda,\delta)$ on compact subsets of $\CM\times[0,\delta_0)$ and 
$D_1(\lambda,\delta)$ is a constant-coefficient diagonal matrix whose diagonal entries are
\begin{align*}
\mu_0(\lambda,\delta)&=-K_0\,\lambda^{1/3}-\frac{\tau_0^2}{3\nu K_0}\,(2\lambda^{2/3}\,+\,\delta^2\,\lambda^{-1/3}\,\langle\beta(\cdot)\rangle)\\
\mu_+(\lambda,\delta)&=-K_0\,\omega\,\lambda^{1/3}-\frac{\tau_0^2\omega^2}{3\nu K_0}\,(2\lambda^{2/3}\,+\,\delta^2\,\lambda^{-1/3}\,\langle\beta(\cdot)\rangle)\\
\mu_-(\lambda,\delta)&=-K_0\,\omega^2\,\lambda^{1/3}-\frac{\tau_0^2\omega}{3\nu K_0}\,(2\lambda^{2/3}\,+\,\delta^2\,\lambda^{-1/3}\,\langle\beta(\cdot)\rangle).
\end{align*}
In particular, we find
\begin{align*}
\Re(\mu_0(\lambda,\delta))&=\displaystyle-\left(K_0+\frac{\tau_0^2}{3\nu K_0}\,\delta^2\,|\lambda|^{-2/3}\,\langle\beta(\cdot)\rangle\right)\,\Re(\lambda^{1/3})-\frac{\tau_0^2}{3\nu K_0}\,\frac{2\Re(\bar\lambda\,\lambda^{1/3})}{|\lambda|^{2/3}},\\
\Re(\mu_+(\lambda,\delta))&=\displaystyle-\left(K_0+\frac{\tau_0^2}{3\nu K_0}\,\delta^2\,|\lambda|^{-2/3}\,\langle\beta(\cdot)\rangle\right)\,\Re(\omega\lambda^{1/3})-\frac{\tau_0^2}{3\nu K_0}\,\frac{2\Re(\bar\lambda\,\omega\lambda^{1/3})}{|\lambda|^{2/3}},\\
\Re(\mu_-(\lambda,\delta))&=\displaystyle-\left(K_0+\frac{\tau_0^2}{3\nu K_0}\,\delta^2\,|\lambda|^{-2/3}\,\langle\beta(\cdot)\rangle\right)\,\Re(\omega^2\lambda^{1/3})-\frac{\tau_0^2}{3\nu K_0}\,\frac{2\Re(\bar\lambda\,\omega^2\lambda^{1/3})}{|\lambda|^{2/3}},
\end{align*}
so that, when $\Re(\lambda)\geq0$ and $\lambda$ and $\delta^2\,|\lambda|^{-2/3}$ are sufficiently small,
\begin{align*}
\Re(\mu_0(\lambda,\delta))&\leq-C|\lambda|^{1/3},\\
\Re(\mu_+(\lambda,\delta))&\geq
\begin{cases}C|\lambda|^{1/3},&\textrm{if}\quad\Im(\lambda)\geq0\\
C\left(\frac{\Re(\lambda)}{|\lambda|^{2/3}}+\frac{|\Im(\lambda)|}{|\lambda|^{1/3}}\right),&\textrm{if}\quad\Im(\lambda)\leq0
\end{cases},\\
\Re(\mu_-(\lambda,\delta))&\geq
\begin{cases}C|\lambda|^{1/3},&\textrm{if}\quad\Im(\lambda)\leq0\\
C\left(\frac{\Re(\lambda)}{|\lambda|^{2/3}}+\frac{|\Im(\lambda)|}{|\lambda|^{1/3}}\right),&\textrm{if}\quad\Im(\lambda)\geq0
\end{cases}.
\end{align*}
Using an energy estimate as in the proof of Lemma \ref{lemma1} above, it 
follows 
for $\epsilon,\delta_1>0$ sufficiently small and $R>0$ sufficiently large that
for $0<\delta<\delta_1$ 
system \eqref{evans4} 
has no bounded solutions provided that
\begin{equation}\label{evbd1}
\Re(\lambda)\geq0\,,\quad |\lambda|\leq\epsilon\,,\quad \Re(\lambda)\,|\lambda|^{1/3}+|\Im(\lambda)|\,|\lambda|^{2/3}\,\geq R\,\delta^4\,.
\end{equation}
In particular, this shows that as $\delta\to 0^+$ the unstable eigenvalues satisfy $|\lambda|=O(\delta^{12/5})$. 

Next, we refine the above bound by using spectral stability of the limiting cnoidal wave. To do so, we scale unknowns of the system \eqref{spec1} according to
\be\label{e:spectral-scaling}
\tau(\theta)\ =\ \delta^2\tfrac{\tau_0}{3}a\left(\frac{\tau_0^{5/4}\delta}{\nu^{1/2}}\,\theta\right)\,,\qquad
u(\theta)\ =\ \delta^2\tfrac{u_0}{6}b\left(\frac{\tau_0^{5/4}\delta}{\nu^{1/2}}\,\theta\right)\,,\qquad
\lambda\ =\ \frac{\delta^3}{4\tau_0^{1/4}\nu^{1/2}}\,\Lambda\,.
\ee
Our goal is to prove that the rescaled system obtained from \eqref{spec1} has no unstable eigenvalues $\Lambda$ with $|\Lambda|$ sufficiently large.
To this end, notice that by \eqref{evbd1} the possible unstable eigenvalues $\Lambda$ must satisfy the estimate
\begin{equation}\label{evbd2}
\Re(\Lambda)\geq0\,,\quad \Re(\Lambda)\,|\Lambda|^{1/3}\,+\,\delta\,|\Im(\Lambda)|\,|\Lambda|^{2/3}\,=\,O(1),
\end{equation}
so that, in particular, we already know that $\Lambda=O(\delta^{-3/5})$. Rewriting the eigenvalue system \eqref{spec1} in terms  of the unknown $V(Y)=(\delta^{-2}(b(Y)+\bar c\, a(Y)),a(Y),a'(Y))$ results in the system
\begin{equation}\label{evans5}
V'(Y)\ =\ \left(A_{KdV}(Y;\Lambda)+\Lambda\,\delta\,B+\delta\,C(Y)+R(Y;\Lambda,\delta)\right)\,V(Y)
\end{equation}
with\footnote{We don't need here the exact form of $C$ but we specify it for latter use.}
$$
A_{KdV}(\cdot;\Lambda)\ =\ \begin{pmatrix}
0& \Lambda & 0\\
0 & 0 &1\\
-1 &\sigma_0-T_0(\cdot)& 0
\end{pmatrix}\,,
$$
$$
B\ =\ \frac{1}{2\tau_0^{1/4}\nu^{1/2}}
\begin{pmatrix}
0 & 0 & 0\\
0 & 0 &0\\
0 & 1 & 0
\end{pmatrix}\,\qquad 
C(\cdot)\ =\ 
\frac{1}{2\tau_0^{1/4}\nu^{1/2}}
\begin{pmatrix}
0 & 0 & 0\\
0 & 0 &0\\
0 &T_0'(\cdot)-T_1(\cdot)&-1+T_0(\cdot)-\sigma_0
\end{pmatrix}\,,
$$
and
$$
R(\cdot;\Lambda,\delta)\ =\ 
\begin{pmatrix}
0 & 0 & 0\\
0 & 0 &0\\
O(\Lambda\,\delta^3)+O(\delta^2) & O(\Lambda\,\delta^3)+O(\delta^2) & O(\Lambda\,\delta^2)+O(\delta^2)
\end{pmatrix}\,.
$$

To make use of known results about KdV we rewrite the above spectral problem in a more standard way by introducing 
the unknown $W(\cdot)=(V_2(\cdot),V_3(\cdot),-V_1(\cdot)+(\sigma_0-T_0(\cdot))\,V_2(\cdot))^T$. This leads to
\begin{equation}\label{evans6}
W'(Y)\ =\ \left(H_0(Y;\Lambda)+\Lambda\,\delta\,H_1+\delta\,H_2(Y)+R(Y;\Lambda,\delta)\right)\,W(Y)
\end{equation}
with
\be\label{e:evans-H0}
H_0(\cdot;\Lambda)\ =\ \begin{pmatrix}
0& 1 & 0\\
0 & 0 &1\\
-\Lambda-T_0'(\cdot) &\sigma_0-T_0(\cdot)& 0
\end{pmatrix}\,,
\ee
\be\label{e:evans-H12}
H_1\ =\ \frac{1}{2\tau_0^{1/4}\nu^{1/2}}
\begin{pmatrix}
0 & 0 & 0\\
1 & 0 &0\\
0 &0 & 0
\end{pmatrix}\,,\qquad
H_2(\cdot)\ =\ \frac{1}{2\tau_0^{1/4}\nu^{1/2}}
\begin{pmatrix}
0 & 0 & 0\\
T_0'(\cdot)-T_1(\cdot) & -1+T_0(\cdot)-\sigma_0 &0\\
0 &0 & 0
\end{pmatrix}\,,
\ee
and
$$
R(\cdot;\Lambda,\delta)\ =\ 
\begin{pmatrix}
0 & 0 & 0\\
O(\Lambda\,\delta^3)+O(\delta^2)&O(\Lambda\,\delta^2)+O(\delta^2)&O(\Lambda\,\delta^3)+O(\delta^2)\\
0 & 0 &0
\end{pmatrix}\,.
$$
The leading order ``KdV" part may now be changed to a diagonal constant-coefficient matrix through 
an explicit periodic Floquet change of variable $P(\cdot;\Lambda)$ such that
$$
P(\cdot;\Lambda)\ =\ \begin{pmatrix}
1 & 0 & 0\\
0 & \Lambda^{1/3} &0\\
0 & 0 & \Lambda^{2/3}
\end{pmatrix}\quad\left(
\begin{pmatrix}
1 & 1 & 1\\
1 & \omega &\omega^2\\
1 & \omega^2 & \omega
\end{pmatrix}+O(\Lambda^{-2/3})\right),
$$
with inverse 
$$
\left(\frac13\begin{pmatrix}
1 & 1 & 1\\
1 & \omega^2 &\omega\\
1 & \omega & \omega^2
\end{pmatrix}+O(\Lambda^{-2/3})\right)\quad
\begin{pmatrix}
1 & 0 & 0\\
0 & \Lambda^{-1/3} &0\\
0 & 0 & \Lambda^{-2/3}
\end{pmatrix}\,.
$$
Indeed, using that $\Lambda=O(\delta^{-3/5})$, replacing $W$ with $Y(\cdot)=P(\cdot;\Lambda)^{-1}W(\cdot)$ leads to the system
\begin{equation}\label{evans7}
Y'\ =\ \left(D(\Lambda)+\Lambda^{2/3}\,\delta\,Q_1+R(\cdot;\Lambda,\delta)\right)\,Y
\end{equation}
with
$$
D(\Lambda)\ =\ \begin{pmatrix}
\mu_0(\Lambda)&0 & 0\\
0 & \mu_+(\Lambda) &0\\
0 &0&\mu_-(\Lambda)
\end{pmatrix}\,\qquad 
Q_1\ =\ \frac{1}{6\tau_0^{1/4}\nu^{1/2}}
\begin{pmatrix}
1 & 1 & 1\\
\omega^2 & \omega^2 &\omega^2\\
\omega & \omega & \omega
\end{pmatrix}\,,
$$
and $R(\cdot;\Lambda,\delta)\ =\ O(\delta)$,
where the $\mu_0$ and $\mu_\pm$ are smooth functions of $\Lambda$.  In particular, 
we know that when $\Re(\Lambda)\geq0$ we have
$$	
\Re(\mu_0(\Lambda))\geq0\,,\qquad\Re(\mu_+(\Lambda))\leq0\,,\qquad\Re(\mu_-(\Lambda))\leq0\,,
$$
and
$$
\begin{array}{rcl}
\mu_0(\Lambda)&=&\displaystyle \Lambda^{1/3}(1+O(\Lambda^{-2/3}))\\[1em]
\mu_+(\Lambda)&=&\displaystyle \Lambda^{1/3}(\omega+O(\Lambda^{-2/3})) \\[1em]
\mu_-(\Lambda)&=&\displaystyle \Lambda^{1/3}(\omega^2+O(\Lambda^{-2/3}))
\end{array}\,.
$$

Next, provided that $\Lambda^{1/3}\delta$ is sufficiently small, which is guaranteed for $\delta$ sufficiently small since $\Lambda=O(\delta^{-3/5})$, we can use a near-identity change of variables to diagonalize the system \eqref{evans7} to order $O(\Lambda^{2/3}\delta)$, resulting in the
system
$$
Z'(Y)\ =\ \left(\,D_1(\Lambda,\delta)\,
+\,O(\Lambda^{5/3}\,\delta^4)\,+\,O(\delta)\right)\ Z(Y),
$$
where $D_1(\Lambda,\delta)$ is a constant-coefficient diagonal matrix whose diagonal entries are
\begin{align*}
\mu_0(\Lambda,\delta)&=\displaystyle\mu_0(\Lambda)\,+\,\frac{1}{6\tau_0^{1/4}\nu^{1/2}} \Lambda^{2/3}\,\delta,\\
\mu_+(\Lambda,\delta)&=\displaystyle\mu_+(\Lambda)\,+\, \frac{\omega^2}{6\tau_0^{1/4}\nu^{1/2}} \Lambda^{2/3}\,\delta,\\
\mu_-(\Lambda,\delta)&=\displaystyle\mu_+(\Lambda)\,+\, \frac{\omega}{6\tau_0^{1/4}\nu^{1/2}} \Lambda^{2/3}\,\delta.
\end{align*}
%
Hence, when $\Re(\Lambda)\geq0$, $\Lambda$ is large, and $\delta\,|\Lambda|^{1/3}$ is small,
$$
\begin{array}{rcl}
\Re(\mu_0(\Lambda,\delta))&\geq&C|\Lambda|^{1/3}\\
\Re(\mu_+(\Lambda,\delta))&\leq&
-C\,\left(\Re(\Lambda)|\Lambda|^{-2/3}\,+\,\delta\,|\Lambda|^{2/3}\right)\\
\Re(\mu_-(\Lambda,\delta))&\leq&
-C\,\left(\Re(\Lambda)|\Lambda|^{-2/3}\,+\,\delta\,|\Lambda|^{2/3}\right),
\end{array}
$$
which, when combined with previous exclusions, is sufficient to prove by an energy argument similar to that in Lemma \ref{lemma1} above
that if $\delta_1\in(0,\delta_0)$ is sufficiently small and $R>0$ is sufficiently large there are no eigenvalues of \eqref{evans7} when $0<\delta<\delta_1$ and
$$
\Re(\Lambda)\geq0\,,\qquad \dfrac{\Re(\Lambda)}{\delta\,|\Lambda|^{2/3}}\,+\,|\Lambda|^{2/3}\,\geq R\,.
$$
In particular, it follows that any unstable eigenvalue $\lambda$ of the original spectral problem \eqref{spec1} must
satisfy the estimates $|\lambda|=O(\delta^3)$ and $\Re(\lambda)=O(\delta^4)$ as $\delta\to 0^+$, by $|\Lambda|\leq R^{3/2}$ and
$\Re(\Lambda)\leq R\delta |\Lambda|^{2/3}\leq R\delta$, together with the scaling
$\lambda = c\delta^3 \Lambda$, with $c>0$ a real constant.
\end{proof}


\subsection{Connection to the KdV-KS index: proof of Thm. \ref{mains}}\label{evans_expansion}
It follows from Proposition \ref{Prop1} that, in order to complete the proof of Theorem \ref{mains}, it remains to study the eigenvalues of \eqref{spec1}, supplemented
with the appropriate Bloch quasi-periodic boundary conditions, of the form $\lambda=\Lambda\delta^3$ with $\Lambda$ confined to a compact subset of $\CM$
and $0<\delta\ll 1$.  More precisely, using the rescaling \eqref{e:spectral-scaling} from the proof of Proposition \ref{Prop1} and
setting $(\alpha,\beta)=(\delta^{-2}(b+\bar c\, a),a)$, we must study the spectral problem
\be\label{e:rescaled-spectral}
\begin{aligned}
&\Lambda \beta-\alpha'\ =\ 0\\
&\tilde\delta\left(\beta+\tilde\tau \beta+\tilde c \beta-\alpha\right)'
\ =\ \tilde\tau \beta+\tilde c\beta-\alpha-\beta''+O(\tilde\delta^2)f(\alpha,\beta,\beta'),
\end{aligned}
\ee
for some smooth function $f$, 
supplemented with the boundary condition $(\alpha,\beta,\beta')(X)=e^{i\xi}(\alpha,\beta,\beta')(0)$ for some $\xi\in[-\pi/X,\pi/X)$. 
Here, $\tilde \delta\ =\ \tfrac12\tau_0^{-1/4}\nu^{-1/2} \delta$ is as in Theorem \ref{maine}, and the bounds in $O(\cdots)$ are uniform
as $\Lambda$ varies on compact subsets of $\CM$.

To study the above one-parameter family of eigenvalue problems, parametrized by the Bloch frequency~$\xi$, we can define a periodic Evans function, a complex analytic function whose zeros, for each fixed $\xi$, agree in location and algebraic multiplicity with the eigenvalues of the boundary value problem \eqref{e:rescaled-spectral}. 
To proceed, we first recall, from the proof of Proposition \ref{Prop1}, 
that the spectral problem
under scaling \eqref{e:spectral-scaling} may be
written as the dynamical system (Eq. \eqref{evans6})
$$
W'(Y)\ =\ \left(H_0(Y;\Lambda)+\Lambda\,\delta\,H_1+\delta\,H_2(Y)+O(\tilde\delta^2)\right)\,W(Y)
$$
for $W=(\beta,\beta',-\alpha+(\sigma_0-T_0)\,\beta)^T$, where $H_0$, $H_1$, $H_2$ are given by \eqref{e:evans-H0} and \eqref{e:evans-H12}. Introducing the new dependent
variables
$Z=(W_1,W_2,W_3+\tilde\delta((\Lambda-T_0'-T_1)W_1+(-1+T_0+\sigma_0)W_2))^T$, 
the differential system 
becomes
\be\label{e:rescaled-evans}
Z'\ =\ 
\left(\begin{pmatrix}
0& 1 & 0\\
0 & 0 &1\\
-\Lambda-T_0'&\sigma_0-T_0& 0
\end{pmatrix}
+\tilde \delta
\begin{pmatrix}
0& 0 & 0\\
0 & 0 &0\\
T_0''-T_1'&2T_0'-T_1+\Lambda&T_0-1+\sigma_0
\end{pmatrix}+O(\tilde\delta^2)\right)Z\,.
\ee
Let $\Phi(\cdot,\Lambda,\tilde\delta)$ denote the associated $3\times 3$ fundamental solution matrix\footnote{In particular, this guarantees that $\Phi(0,\Lambda,\tilde\delta)=$Id.} associated 
with
\eqref{e:rescaled-evans}. It is an easy consequence of 
the regularity with respect to parameters of the flow associated with the differential system \eqref{e:rescaled-evans} 
that $\Phi(\cdot,\Lambda,\tilde\delta)$ is analytic with respect to $\Lambda\in\CM$ and $\tilde\delta$. 
Moreover,
for any fixed $(\xi,\tilde\delta)$, eigenvalues $\Lambda$ of \eqref{e:rescaled-spectral} agree in location and algebraic multiplicity with roots of the Evans function 
\begin{equation}\label{def:ev}
E_{SV}(\Lambda,\xi,\tilde\delta):=\det\left(\Phi(X,\Lambda,\tilde\delta)-e^{i\xi X}{\rm Id}\right);
\end{equation}
see \cite{G} for details.
To complete the proof of Theorem \ref{mains}, we must study the roots of the function $E_{SV}(\cdot,\xi,\tilde\delta)$ on compact subsets of $\CM$ for all $\xi\in[-\pi/X,\pi/X)$ and $0<\tilde\delta\ll 1$.

In order to connect the stability properties of the small amplitude roll-waves 
described 
by Theorem~\ref{maine} to 
those
of the associated leading
order approximating KdV-KS waves given in Proposition \ref{p:kdvsolnexpand}, we also define an 
%
Evans function $E_{KdV-KS}(\Lambda,\xi,\delta)$ for the eigenvalue boundary value problem 
\be\label{e:spectral-KdV-KS}
\Lambda z+((T_\delta-\sigma_\delta)z)'+z'''+\delta\,(z''+z'''')\ =\ 0
\ee
with $(z,z',z'',z''')(X)=e^{i\xi\,X}(z,z',z'',z''')(0)$. 
The next result shows that the Evans function $E_{KdV-KS}$ is faithfully described, to leading order, by the St. Venant Evans function
$E_{SV}$ for $0<\tilde\delta\ll 1$.

\begin{proposition}\label{p:evansexpand}
Uniformly on compact sets of $\Lambda\in\CM$, the Evans function $E_{KdV-KS}$ can be expanded for $0<\delta\ll 1$ as
\begin{equation}\label{e:SV-to-KdV-KS}
E_{KdV-KS}(\Lambda,\xi,\tilde\delta)
\ =\ -e^{i\xi X}(1+O(\tilde\delta))\exp\left(\frac{X}{\tilde\delta}\right)
\left(E_{SV}(\Lambda,\xi,\tilde\delta)
+O(\tilde\delta^2(|\Lambda|^2+|\xi|^2)+O(\tilde\delta^3(|\Lambda|+|\xi|))\right) .
\end{equation}
\end{proposition}

\begin{proof}
A similar expansion has been obtained for $E_{KdV-KS}$ in \cite[Propositions~3.7 \& 4.1]{JNRZ1}. Our proof parallels the arguments there but in a slightly more precise way in order to equate leading order terms with those of $E_{SV}$ .

The starting point is, as in the proof of \cite[Propositions~3.7]{JNRZ1}, that the spectral problem \eqref{e:spectral-KdV-KS} 
may be
equivalently written, through a series of
variables transformations with Jacobian of size $1+\mathcal{O}(\delta)$, as 
$$
\begin{array}{rcl}
Z'&=&\displaystyle 
\left(\begin{pmatrix}
0& 1 & 0\\
0 & 0 &1\\
-\Lambda-T_0'&\sigma_0-T_0& 0
\end{pmatrix}
+\delta
\begin{pmatrix}
0& 0 & 0\\
0 & 0 &0\\
T_0''-T_1'&2T_0'-T_1+\Lambda&T_0-1+\sigma_0
\end{pmatrix}\right)Z\ +\ O(\delta^2)g_1(Z,w)\\[2em]
w'&=&-\frac{1}{\delta}w\ +\ O(\delta)g_2(Z,w)\,
\end{array}
$$
for some smooth functions $g_1,g_2$; see  \cite[Propositions~3.7]{JNRZ1} for details.  
To obtain the expected homogeneity in $(\Lambda,\xi)$ we also use that the above system supports a conservation law
\be\label{B}
\begin{array}{rcl}
\Lambda Z_1&=&\Big[\left(T_0-\sigma_0-1+T_0'+\Lambda+\delta(T_1+T_1'+T_0''-T_0+\sigma_0+1)\right)Z_1\\[1em]
&&+\left(T_0-\sigma_0-1+\delta(T_1+2T_0'+\Lambda))Z_2\right)+\left(-1-\delta(T_0-\sigma_0-1)\right)Z_3\ + O(\delta^2)[Z,w]\Big]'
\end{array}
\ee
and that the derivative $v_\delta'$ of the profile built in Proposition~\ref{p:kdvsolnexpand} provides a special solution of the system when $\Lambda=0$. Elaborating on this, manipulations on lines and columns of the determinant usually performed to validate at the spectral level averaged
(or ``Whitham'') 
equations \cite{Se} then yield that the Evans function $E_{KdV-KS}(\Lambda,\xi,\tilde\delta)$ may be written as
\begin{equation}\label{evans_final}
E_{KdV-KS}(\Lambda,\xi,\tilde\delta)=
-e^{i\xi X}(1+O(\tilde\delta))\exp\left(\frac{X}{\tilde\delta}\right)\det\left(M(\Lambda,\xi,\tilde\delta)+\widetilde{M}(\Lambda,\xi,\tilde\delta)\right),
\end{equation}
where
$$
M(\Lambda,\xi,\tilde\delta)\ =\ 
\begin{pmatrix}
N&\begin{array}{c}0\\0\\0\end{array}\\
\begin{array}{ccc}0&0&0\end{array}&1
\end{pmatrix},~~~
N(\Lambda,\xi,\tilde\delta)\ =\ 
\begin{pmatrix}
O(|\Lambda|+|\xi|)&
O(1)&
O(1)\\[1ex]
O(|\Lambda|+|\xi|)&
O(1)&
O(1)\\[1ex]
O(|\Lambda|^2+|\Lambda||\xi|)&
O(|\Lambda|+|\xi|)&
O(|\Lambda|+|\xi|)
\end{pmatrix}\, ,
$$
and
$$
\widetilde M(\Lambda,\xi,\tilde\delta)\ =\ \begin{pmatrix}
O(\tilde\delta^2(|\Lambda|+|\xi|))&
O(\tilde\delta^2)&
O(\tilde\delta^2)&
O(\tilde\delta^2)\\[1ex]
O(\tilde\delta^2(|\Lambda|+|\xi|))&
O(\tilde\delta^2)&
O(\tilde\delta^2)&
O(\tilde\delta^2)\\[1ex]
O(\tilde\delta^2|\Lambda|(|\Lambda|+|\xi|))&
O(\tilde\delta^2(|\Lambda|+|\xi|))&
O(\tilde\delta^2(|\Lambda|+|\xi|))&
O(\tilde\delta^2)\\[1ex]
O(\tilde\delta|\Lambda|)&
O(\tilde\delta)&
O(\tilde\delta)&
O(\tilde\delta)
\end{pmatrix}.
$$
Again, see \cite[Propositions~3.7]{JNRZ1} for details.  

Now, returning to the St. Venant spectral problem \eqref{e:rescaled-evans}, we notice that
\eqref{e:rescaled-evans} 
also
%
supports a conservation law in form \eqref{B}
%
and, furthermore, the space derivative of the traveling-wave profile provides a special solution when $\Lambda=0$.  
Using calculations completely analogous to those described above for the KdV-KS spectral problem \eqref{e:spectral-KdV-KS},
it follows that we can write
\[
E_{SV}(\Lambda,\xi,\tilde\delta)=\left(1+\mathcal{O}(\delta)\right)\det\left(N(\Lambda,\xi,\tilde\delta)+\tilde{N}(\Lambda,\xi,\tilde\delta)\right),
\]
where $N$ is \emph{exactly} the same matrix as above\footnote{Not just in order of magnitude, but by component-wise identification of the coefficients.} and
%
$$
\tilde N(\Lambda,\xi,\tilde\delta)\ =\ \begin{pmatrix}
O(\tilde\delta^2(|\Lambda|+|\xi|))&
O(\tilde\delta^2)&
O(\tilde\delta^2)\\[1ex]
O(\tilde\delta^2(|\Lambda|+|\xi|))&
O(\tilde\delta^2)&
O(\tilde\delta^2)\\[1ex]
O(\tilde\delta^2|\Lambda|(|\Lambda|+|\xi|))&
O(\tilde\delta^2(|\Lambda|+|\xi|))&
O(\tilde\delta^2(|\Lambda|+|\xi|))
\end{pmatrix}\,.
$$
The expansion \eqref{e:SV-to-KdV-KS} follows conveniently by expanding the determinant $\det(M+\tilde M)$ in \eqref{evans_final}.
\end{proof}

The proof of Theorem \ref{mains} now follows immediately from the proof of Proposition \ref{p:kdvstab} in \cite{JNRZ1} for the stability of KdV-KS waves in the KdV limit $\delta\to 0^+$. Indeed, the proof in \cite{JNRZ1} followed by studying the renormalized KdV-KS Evans function
\[
\overline{E}_{KdV-KS}(\Lambda,\xi,\tilde\delta)
\ =\ e^{-i\xi X}(1+O(\tilde\delta))\exp\left(-\frac{X}{\tilde\delta}\right)E_{KdV-KS}(\Lambda,\xi,\tilde\delta)\,,
\]
and 
using
the asymptotic description of $\overline{E}_{KdV-KS}(\Lambda,\xi,\tilde\delta)$ up to 
$O(\tilde\delta^2(|\Lambda|^2+|\xi|^2))+O(\tilde\delta^3(|\Lambda|+|\xi|))$. Since Proposition \ref{p:evansexpand} implies that 
\[
E_{SV}(\Lambda,\xi,\tilde \delta)=\overline{E}_{KdV-KS}(\Lambda,\xi,\tilde\delta)+O(\tilde\delta^2(|\Lambda|^2+|\xi|^2))+O(\tilde\delta^3(|\Lambda|+|\xi|))
\]
it follows that the same arguments can be applied \emph{without modification} to the Evans function $E_{SV}(\Lambda,\xi,\tilde \delta)$.
For completeness, we briefly sketch the details.

For $k\in(0,1)$ such that condition (A) holds, all the non-zero Bloch eigenvalues of the St. Venant linearized operator admit a smooth expansion in $\tilde\delta$ for $0<\tilde\delta\ll 1$.  In particular, to each pair $(\xi,\Lambda_0)$ such that $\Lambda_0\in\RM i$ is a non-zero eigenvalue of the KdV Bloch operator $\mathcal{L}_\xi[T_0]$ there is a unique root of $E_{SV}(\Lambda,\xi,\tilde\delta)$ for $0<\tilde\delta\ll 1$ that can be expanded in $\tilde\delta$ as
\[
\Lambda(\tilde\delta;\xi,\Lambda_0)=\Lambda_0+\tilde\delta\lambda_1(\xi,\Lambda_0)+\mathcal{O}(\tilde\delta^2)
\]
where $\lambda_1(\cdot,\cdot)$ is the function already involved in definition \eqref{index} ; see \cite[Corollary 3.8]{JNRZ1}. 
As this expansion is uniform in $(\xi,\Lambda_0)$ when $(\xi,\Lambda_0)$ varies in a compact set that does not contain $(0,0)$, 
it follows that for any neighborhood $\mathcal{U}$ of $0$ in the spectral plane the condition $k\in\mathcal{P}$ provides when $\tilde\delta$ is sufficiently small a negative upper bound on the real part of the spectrum in $\C\setminus\{0\}$; see \cite[Corollary 3.10]{JNRZ1} for details.  

On the other hand, when $\xi=0$ the origin is an eigenvalue of the KdV Bloch operator $\mathcal{L}_0[T_0]$ of algebraic multiplicity three and geometric multiplicity two.  To unfold the degeneracy for $0<\tilde\delta\ll 1$ we directly apply the arguments in \cite[Section 4]{JNRZ1}. By studying the three asymptotic regions $0<\tilde\delta\lesssim|\xi|$ \cite[Lemma 4.4]{JNRZ1}, $\tilde\delta\sim|\xi|$ \cite[Lemma 4.6]{JNRZ1}, and 
$0\leq|\xi|\lesssim\tilde\delta$
\cite[Lemma 4.7]{JNRZ1} separately it follows that the condition $k\in\mathcal{P}$ implies that a set of ``subcharacteristic conditions" hold which, in turn, are shown to imply that all the roots of 
$E_{SV}(\Lambda,\xi,\tilde\delta)$, with 
$0<|(\Lambda,\xi)|\ll 1$ and $0<\tilde\delta\ll 1$
all have negative real parts.  More precisely, under these conditions it is shown that the (diffusive) spectral stability conditions (D1)-(D3) and the non-degeneracy hypothesis (H1) are satisfied
when $\delta$ is sufficiently small.
Finally, we note as in the Introduction
that, by the description of profiles in Theorem \ref{maine}, slope condition
\eqref{e:Eslope} (or, equivalently, \eqref{e:Lslope}) is always satisfied.
This completes the proof of Theorem \ref{mains}.

\section{Stability of large amplitude roll-waves}\label{s:infty}

In the previous section, we rigorously justified the KdV-KS equation \eqref{kdv-ks} as a correct description for the weak hydrodynamic instability in inclined thin film flow.  In particular, in the weakly nonlinear regime $F\to 2^+$ we saw that the KdV-KS equation accurately predicts the  stability of the associated small amplitude roll-wave solutions of the shallow-water equations \eqref{swe}.  We now complement this study by continuing our analysis into the large amplitude regime, far from the weakly-unstable limit $F\to 2^+$, performing a systematic stability analysis for roll-waves with Froude number on the entire range of existence $F>2$, including the distinguished limit $F\to\infty$.  We begin by considering the limit $F\to\infty$, identifying a one-parameter family of limiting systems approachable by various scaling choices in the shallow-water equations. We will then numerically study the influence of intermediate Froude numbers $2<F<\infty$ on the range of stability of periodic waves.

\subsection{Scaling as $F\to\infty$}\label{s:per}

In our $F\to\infty$ studies, we investigate the stability of periodic traveling wave solutions of some asymptotic systems obtained from \eqref{swl} via scaling arguments.  To begin, performing the change of variables $(x,t)\mapsto(\tilde x,\tilde t)=(k(x-ct),t)$ in \eqref{swl} and erasing tildes yields the equivalent system
\begin{equation*}
\displaystyle   
\d_t \tau-k\,c\,\d_x\tau-k\d_x u=0,\quad \d_t u-k\,c\,\d_xu+k\,\d_x\left(\frac{\tau^{-2}}{2F^2}\right)=1-\tau\,u^2+\nu k^2\d_x(\tau^{-2}\partial_x u).
\end{equation*}
To prepare for sending $F\to\infty$ above, we first scale the dependent quantities independently via 
$\tau=F^{\atau}\,a$, $u=F^{\au}\,b$, $k=F^{\ak}\,k_0$ and $c=F^{\ac}\,c_0$.  This transforms the above shallow-water system to
\begin{equation}\label{swl:rescaleF->infty}
\begin{array}{rcl}
\displaystyle   
F^{\atau}\d_ta-\,k_0c_0\,F^{\atau+\ak+\ac}\,\d_xa&-&k_0\,F^{\au+\ak}\d_x b\ =\ 0\,,\\\displaystyle 
F^{\au}\d_tb-k_0c_0\,F^{\au+\ak+\ac}\,\d_xb&+&k_0\,F^{\ak-2-2\atau}\,\d_x\left(\frac{a^{-2}}{2}\right)\\\displaystyle
&=&1-F^{\atau+2\au}a\,b^2+\nu k_0^2\,F^{\au+2\ak-2\atau}\,\d_x(a^{-2}\d_x b)\,.
\end{array}
\end{equation}
Seeking stationary $1$-periodic solutions of \eqref{swl:profileF->infty} leads to an ODE system
\ba\nonumber
\displaystyle   
-\,c_0\,F^{\atau+\ac}\,a'-F^{\au}b'\ =\ 0,&&\\\displaystyle 
k_0c_0^2\,F^{\atau+\ak+2\ac}\,a'+k_0\,F^{\ak-2-2\atau}\,\left(\frac{a^{-2}}{2}\right)'&=&1-F^{\atau+2\au}a\,b^2-\nu k_0^2\,c_0\,F^{\ac+2\ak-\atau}\,(a^{-2}a')'\, ,
\ea
governing the traveling wave solutions of \eqref{swl} in the appropriate moving frame.  Note that, 
by its expression as $\int_0^X \bar \tau(x)dx$, the period in the physical Eulerian variables scales as $F^{\atau-\ak}$ under the above transformations. To balance the terms of the \qut{elliptic-constant} right-hand side above, we choose now to impose $\atau+2\au=0$ and $\ac+2\ak-\atau=0$. This effectively  leaves two free parameters, say $\alpha=\atau$ and $\beta=\ak$, with $\au=-\alpha/2$ and $\ac=\alpha-2\beta$. Under this choice, the above profile system now reads
\ba\label{swl:profileF->infty}
\displaystyle   
-\,c_0\,F^{\frac52\alpha-2\beta}\,a'-b'\ =\ 0,&&\\\displaystyle 
k_0c_0^2\,F^{3\alpha-3\beta}\,a'+k_0\,F^{-2-2\alpha+\beta}\,\left(\frac{a^{-2}}{2}\right)'&=&1-a\,b^2-\nu k_0^2\,c_0\,(a^{-2}a')'\,.
\ea

The first equation may be integrated to yield $b=q_0-F^{\frac52\alpha-2\beta}\,c_0\,a$ where $q_0$ is a constant of integration.  Substituting this into the second equation above, the rescaled profile system is reduced to the scalar equation
\be
k_0c_0^2\,F^{3\alpha-3\beta}\,a'+k_0\,F^{-2-2\alpha+\beta}\,\left(\frac{a^{-2}}{2}\right)'
\ =\ 1-a\,(q_0-F^{\frac52\alpha-2\beta}\,c_0\,a)^2-\nu k_0^2\,c_0\,(a^{-2} a')'\,.
\ee
In order to ensure that the ``elliptic part" $(a^{-2}a')'$ is not asymptotically negligible as $F\to\infty$, we restrict ourselves to the parameter regimes where $\tfrac54\alpha\leq\beta\leq2\alpha+2$, $\beta\geq\alpha$.  In particular, this latter choice requires $\alpha\geq -2$. With the above choices, the full rescaled shallow-water equations \eqref{swl:rescaleF->infty} now read as
\begin{equation}\label{swl:rescaleF->infty2}
\begin{array}{rcl}
\displaystyle   
F^{\alpha}\d_ta-\,k_0c_0\,F^{2\alpha-\beta}\,\d_xa&-&k_0\,F^{\beta-\alpha/2}\d_x b\ =\ 0\,,\\\displaystyle 
F^{-\alpha/2}\d_tb-k_0c_0\,F^{\alpha/2-\beta}\,\d_xb&+&k_0\,F^{-2-2\alpha+\beta}\,\d_x\left(\frac{a^{-2}}{2}\right)\\\displaystyle 
&=&1-a\,b^2+\nu k_0^2\,F^{2\beta-5\alpha/2}\,\d_x(a^{-2}\d_x b)\,,
\end{array}
\end{equation}
which is a two-parameter family of systems, parametrized by $\alpha\geq -2$ and $\tfrac54\alpha\leq\beta\leq2\alpha+2$.

Finally, when sending $F\to\infty$ in \eqref{swl:rescaleF->infty2}, we balance the first-order terms of the reduced profile equation \eqref{swl:profileF->infty} by requiring $3\alpha-3\beta=-2-2\alpha+\beta$, i.e. that $\beta=\frac{1}{2}+5\alpha/4$, effectively reducing our scalings to a one parameter family indexed by $\alpha\geq -2$.  In particular, under this choice the full one-parameter family of rescaled shallow-water systems reads as
\begin{equation}\label{swl:rescaleF->inftyFinal}\begin{array}{rcl}
\displaystyle   
F^{\alpha}\d_ta-\,k_0c_0\,F^{3\alpha/4-1/2}\,\d_xa&-&k_0\,F^{1/2+3\alpha/4}\d_x b\ =\ 0\,,\\\displaystyle 
F^{-\alpha/2}\d_tb-k_0c_0\,F^{-3\alpha/4-1/2}\,\d_xb&+&k_0\,F^{-3/2-3\alpha/4}\,\d_x\left(\frac{a^{-2}}{2}\right)\\\displaystyle 
&=&1-a\,b^2+\nu k_0^2\,F\,\d_x(a^{-2}\d_x b),
\end{array}
\end{equation}
and the associated rescaled profile system is equivalent to  $b=q_0-F^{-1}\,c_0\,a$, $q_0$ constant and
\begin{equation}\label{profileF->inftyFinal}
F^{-3/2-3\alpha/4}\,\left(k_0c_0^2\,a'+k_0\,\left(\frac{a^{-2}}{2}\right)'\right)
\ =\ 1-a\,(q_0-F^{-1}\,c_0\,a)^2-\nu k_0^2\,c_0\,(a^{-2} a')'\,.
\end{equation}
By the discussion above Eq. \eqref{swl:profileF->infty},
the Eulerian period of the periodic profiles satisfying \eqref{profileF->inftyFinal} scale as $F^{-\alpha/4-1/2}$. Notice that \eqref{profileF->inftyFinal} is equivalent to the profile equation \eqref{rprof} claimed in the introduction.  

As mentioned in the introduction, taking $F\to\infty$ in \eqref{profileF->inftyFinal} produces different limiting profile equations depending on whether $\alpha=-2$ or if $\alpha>-2$.  Next, we discuss both limiting profile equations and the spectral problems governing the stability of the profiles.

\subsubsection{Case $\alpha=-2$}\label{s:scaling}

Taking $\alpha=-2$ in the above discussion corresponds to rescaling \eqref{swl:profileF->infty} via
$$
\tau=F^{-2}\,a,\quad u=F\,b,\quad k=k_0\,F^{-2},\quad c=c_0\,F^2
$$
Note that in this case the Eulerian period of the profile is held constant as $F\to\infty$.  
With this choice, the profile equation \eqref{profileF->inftyFinal} reads as
\be
\left(k_0c_0^2\,a'+k_0\,\left(\frac{a^{-2}}{2}\right)'\right)
\ =\ 1-a\,(q_0-F^{-1}\,c_0\,a)^2-\nu k_0^2\,c_0\,(a^{-2} a')',
\ee
and $b=q_0-F^{-1}\,c_0\,a$, $q_0$ constant.  Linearizing \eqref{swl:rescaleF->inftyFinal} with $\alpha=-2$ 
about such an $X_0$-periodic profile $(\bar{a},\bar{b})$ yields the associated rescaled spectral problem
\begin{equation}\label{eval2}
\begin{aligned}
F^{-2}\lambda\, a-\,k_0c_0\,F^{-2}\, a'&-k_0\,F^{-1} b '\ =\ 0 \\
F\lambda\, b-k_0c_0\,F\, b'-&3k_0\,\left(\bar a^{-3}\, a\right)'\\
&=- a\,\bar b^2-2\bar a\bar  b+\nu k_0^2\,F\,(-2\bar a \bar b' \,a+\bar a^{-2} b')'\,
\end{aligned}
\end{equation}
to be considered for 
$(a, b)$ satisfying suitable Bloch boundary conditions. 
Sending $F\to\infty$ above, it follows that the profile equation is a regular perturbation of the limiting system 
\be\label{profile_crit}
\left(k_0c_0^2\,a'+k_0\,\left(\frac{a^{-2}}{2}\right)'\right)
\ =\ 1-a\,q_0^2-\nu k_0^2\,c_0\,(a^{-2}a')',
\ee
while the above spectral problem is a regular perturbation of 
\ba\label{eval}
\lambda  a - k_0 c_0  a' -k_0\check b'&=0,\\
\lambda \check b -k_0 c_0 \check b'- k_0 (\bar a_\infty^{-3}  a)'&=\
- a q_0^2 + k_0^2\nu (\bar a_\infty^{-2}\check b'+ 2c_0 \bar a_\infty^{-3} \bar a_\infty' \,a)',
\ea
where $(\bar a_\infty,\bar b_\infty)$, necessarily solutions of \eqref{profile_crit}, denote the limiting profiles of  $(\bar a,\bar b)$ as  $F\to\infty$, and $\check b =F\tilde b$.  The limiting profile equation \eqref{profile_crit} is numerically seen to admit periodic orbits existing in a two parameter family, parametrized by the period $X_0$ and the discharge rate $q_0$. The stability of these profiles may then be investigated by means of the spectral problem
\eqref{eval}: this is discussed in Section \ref{s:numinf} below.

\subsubsection{Case $\alpha>-2$}\label{s:size}

When $\alpha>-2$, it is readily seen that the profile equation \eqref{profileF->inftyFinal} is a regular perturbation
of the ODE
\begin{equation}\label{profile-hamorig}
0\ =\ 1-q_0^2 a-\nu k_0^2\,c_0\,(a^{-2} a')'\,,
\end{equation}
which is Hamiltonian in the unknown $1/a$.  Indeed, denoting $h:=q_0^{-2} a^{-1}$ 
and rescaling space via $x=s/\sqrt{\nu k_0c_0q_0^2}$, the above ODE reads as
\begin{equation}\label{profile-ham}
0=1-h^{-1}+h'',
\end{equation}
where ${}'$ denotes differentiation with respect to $s$.
This is clearly seen to be Hamiltonian and, upon integrating, is equivalent to
\begin{equation}\label{profile-hamquad}
\mu=h-\ln(h)+\frac{1}{2}(h')^2,
\end{equation}
where $\mu$ is a constant of integration.  Elementary phase plane analysis shows that \eqref{profile-ham} admits a one-parameter family of periodic orbits parametrized by the constant $\mu$.  Indeed, for each $\mu>1$ equation \eqref{profile-hamquad} admits a unique (up to spatial translations) periodic solution $h_\mu$, whose period we denote $X_\mu$.  Returning to the original variables then, we see that, choosing $k_0$ to satisfy $\nu k_0^2c_0q_0^2X_\mu^2=1$, the profile equation \eqref{profile-hamorig} admits a three parameter family of periodic orbits with unit period parametrized by $\mu$, $c_0$, and $q_0$.  Clearly, a necessary condition for such a $1$-periodic orbit of \eqref{profile-hamorig} to persist as a solution of \eqref{profileF->inftyFinal} for $F\gg 1$ is that the ($\alpha$-independent) orthogonality condition
$$
0\ =\ \int_0^1\quad\left(\frac1a\right)'\times\left(k_0c_0^2\,a'+k_0\,\left(\frac{a^{-2}}{2}\right)'\right)dx
$$
be satisfied. Note that this yields a selection principle for the wavespeed via
\[
c_0^2\ =\ -\dfrac12\,\dfrac{\displaystyle\int_0^1(a^{-1})'\times(a^{-2})'dx}{\displaystyle\int_0^1(a^{-1})'\times a'dx}
\ =\ \dfrac{\displaystyle\int_0^1a^{-5}\times(a')^2dx}{\displaystyle\int_0^1a^{-2}\times(a')^2dx}
\ =\ \dfrac{\displaystyle\int_0^1a^{-1}\times((a^{-1})')^2dx}{\displaystyle\int_0^1((a^{-1})')^2dx}\,.
\]
Generically zeros of the above selection function are simple and, in this case, as  
in Section~\ref{s:existence} it follows from elementary bifurcation analysis 
that the Hamiltonian profile equation \eqref{profile-hamorig} admits a two-parameter family of periodic orbits, parametrized by $\mu$ and $q_0$ that persist as periodic orbits of \eqref{profileF->inftyFinal} for $F\gg 1$. In particular, note that for $\alpha>-2$, the specific value of $\alpha$ does not enter into either the limiting profile equation nor the selection principle.

Finally, taking $F\to\infty$ in \eqref{swl:rescaleF->inftyFinal}, we see that the spectral stability of the $1$-periodic traveling wave solutions $(\bar{a},q_0)$ of \eqref{swl:rescaleF->inftyFinal} constructed above are determined via the spectral problem
$$\begin{array}{rcl}
\displaystyle   
\Lambda\, a&-&k_0\,\check b '\ =\ 0\,,\\\displaystyle 
0&=&\displaystyle 
- a\,q_0^2+\nu k_0^2\,(2c_0\bar a^{-3}\bar a'\, a+\bar a^{-2}\check b')'\,
\end{array}
$$
where here $(a,b)$ denotes perturbation, $\Lambda=F^{1/2}\lambda$, and $\check b =F b$. The former system may also be written as
$$\begin{array}{rcl}
\displaystyle   
\Lambda\, a&-&k_0\,\check b '\ =\ 0\,,\\\displaystyle 
0&=&\displaystyle 
- a\,q_0^2+\Lambda \nu k_0^2c_0 (\bar a^{-2}\, a)'- \nu k_0^2c_0\,(\bar a^{-2}\, a)''\,.
\end{array}
$$
Therefore, in this case we investigate the spectral problem
\be\label{eig_inf2}
0\ =\  h_\mu^{-2}\,\check a\,-\,\check\Lambda\ \check a'\,+\,\check a''\,.
\ee
for $\check\Lambda=(X_\mu)^{-1}\,\Lambda$ and
$\check a(\cdot)\ =\ q_0^{-2}\left(\bar a^{-2}\  a\right) \left((X_\mu)^{-1}\,\cdot\,\right)\,$
where $h_\mu$ and $X_\mu$ are associated to $\bar a$ as described above.  As above, we point out that for $\alpha>-2$,  the spectral problem governing the spectral stability of the limiting $F=\infty$ profiles is independent of the value of $\alpha$.  As a consequence, spectral instability of the limiting periodic traveling wave solutions constructed above implies spectral instability of the ($\alpha$-dependent) large-$F$ profiles of the system \eqref{swl:rescaleF->inftyFinal} for any $\alpha>-2$.

\subsection{Numerical investigation as $F\to\infty$}\label{s:numinf}
With the above preparations, we report our numerical results concerning the existence and spectral stability of the profiles 
introduced 
in the previous section for the limiting systems in the $F\to\infty$ limit.  

In the case $\alpha>-2$, elementary phase plane analysis indicates that for each $h_-\in(0,1)$ there exists a unique (up to translations) periodic solution
with $h(0)=h_-$.  These profiles were numerically computed for 1000 equally spaced values of $h_-$ in $[0.05,0.95]$ using the Matlab functions
bvp5c and bvp6c, with absolute and relative error tolerances both set to $10^{-8}$ in the bvp solver.  
To this end, we utilized a bisection method to approximate the value $h_+>h_-$ such that $h_+-\log(h_+)=\mu$ and then
approximated the corresponding period by computing $\sqrt{2}\int_{h_--10^{-13}}^{h_++10^{-13}}\left(\mu-x+\log(x)\right)^{-1/2}dx$.

As a first attempt to study the $L^2(\RM)$-spectrum associated to \eqref{eig_inf2}, we utilized a Galerkin truncation method
known as Hill's method.  For each $\xi\in[-\pi/X_\mu,\pi/X_\mu)$, Hill's method proceeds by expanding 
both the unknown $\check{a}$ as well as the background wave $h_\mu$ in the associated
Bloch eigenvalue problem as a Fourier series,  and  then truncating all expansions at some finite order to 
reduce the problem to finding, for each $\xi\in[-\pi/X_\mu,\pi/X_\mu)$ the eigenvalues of a finite-dimensional matrix; see Appendix \ref{s:compmethod} for more details.
For each of the profiles numerically constructed above, unstable spectra were present; see Figure \ref{comparison}(d) for an example.  In addition, we verified the existence
of unstable spectra of \eqref{eig_inf2} by numerically computing winding number for the associated periodic Evans function on a closed contour in the open right half complex
plane verifying that the winding number is indeed greater than zero for some Bloch frequency $\xi$.  Employing Theorem \ref{t:limstab}, his study suggests
that all periodic traveling wave solutions of \eqref{swl:profileF->infty} are \emph{spectrally unstable} for $F\gg 1$.
Note, in our study of spectral problem 
\eqref{eig_inf2} via Hill's method, we used 41 Fourier modes and 1000 Floquet parameters.

\begin{figure}[htbp]
 \begin{center}
$
\begin{array}{lccr}
(a) \includegraphics[scale=0.17]{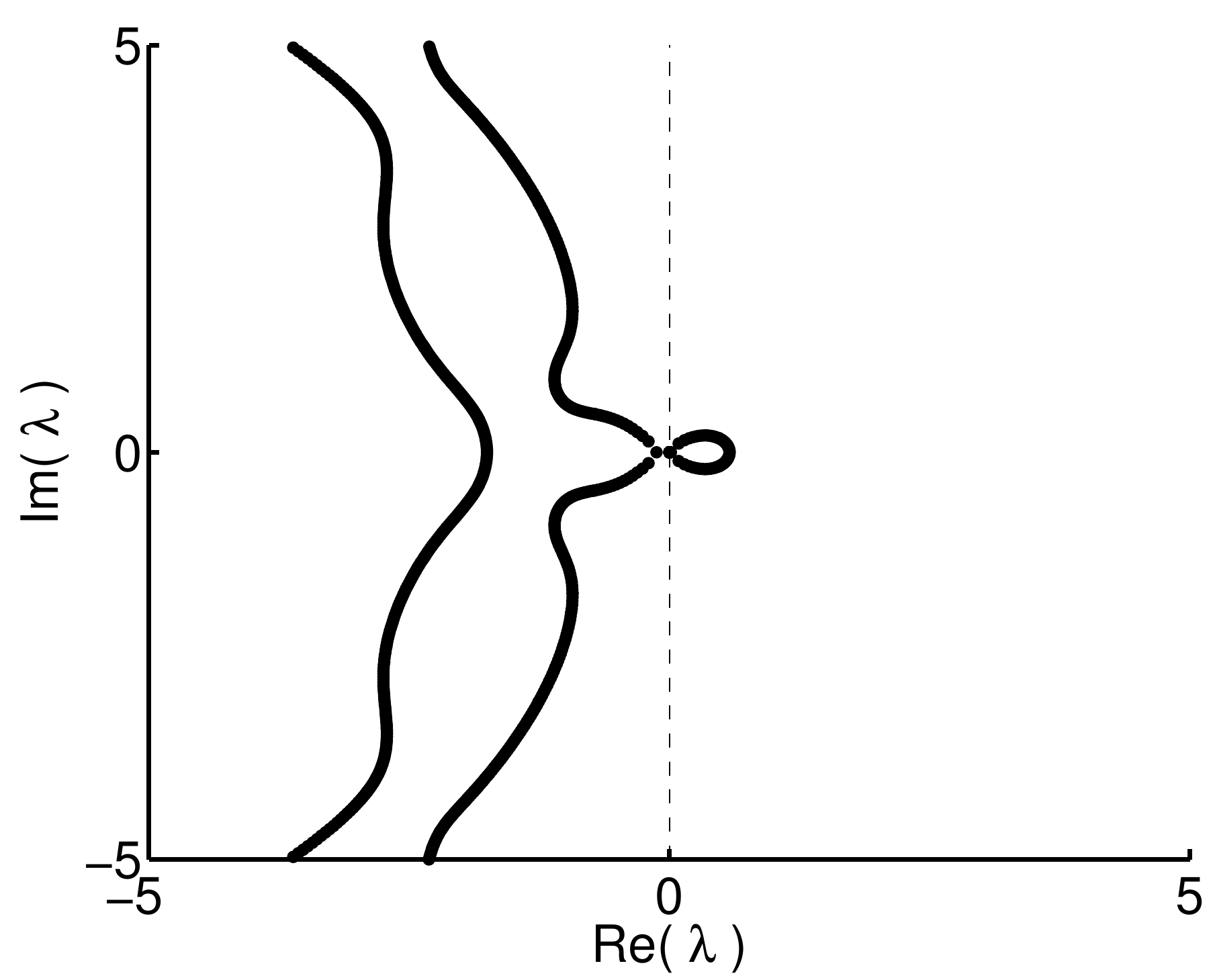}&(b) \includegraphics[scale=0.17]{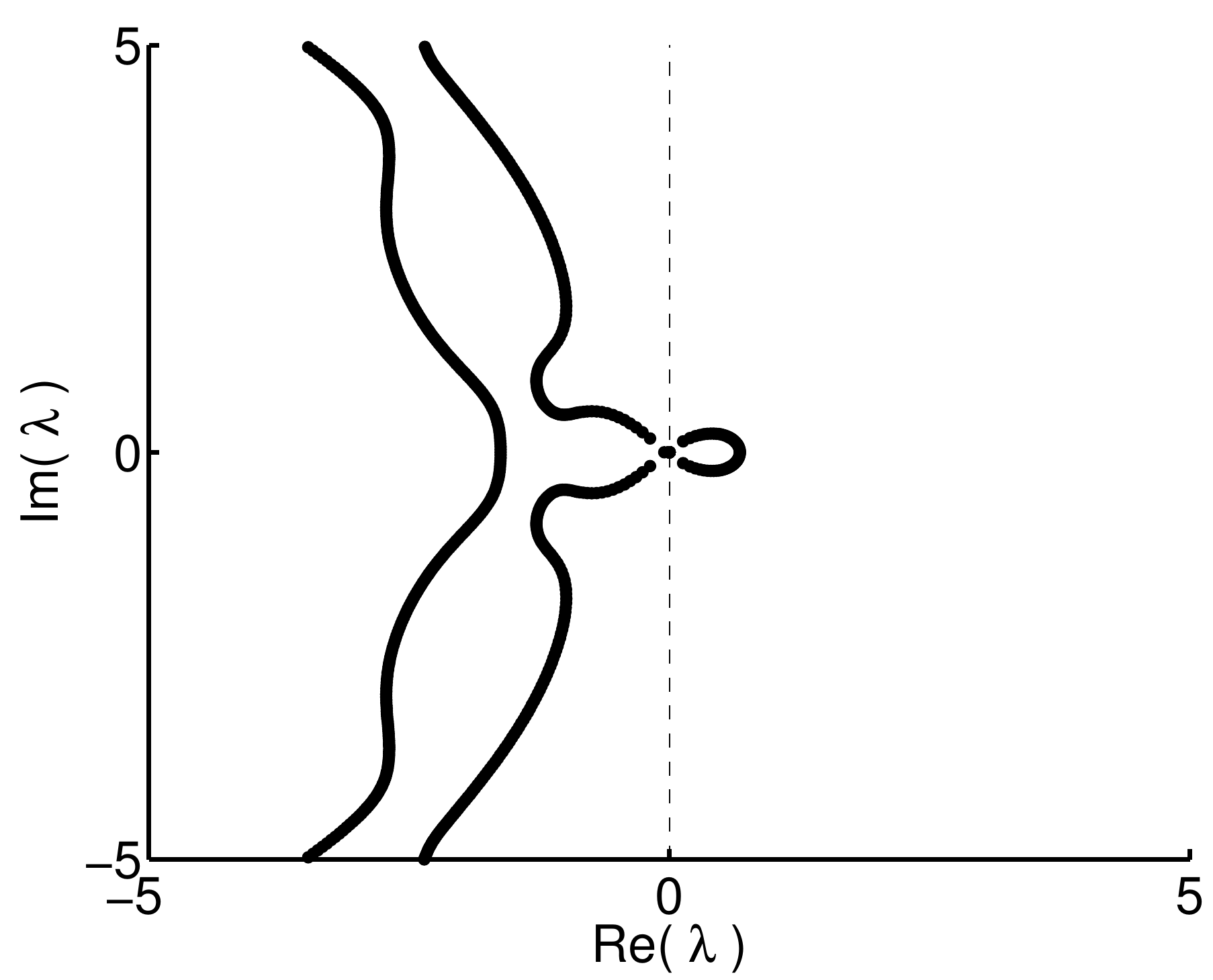} & (c) \includegraphics[scale=0.17]{pix/BJPRZ1fig473}
&(d) \includegraphics[scale=0.17]{pix/BJPRZ1fig303}
\end{array}
$
\end{center}
\caption{In (a) and (b) we plot, respectively for $F = 38$ and $F = 100$, the spectrum 
corresponding to eigenvalue problem
\eqref{eval2} with $\alpha = -2$, 
and in (c) the spectrum corresponding to the limiting spectral
problem
\eqref{eval}. Here 
$\nu  = 0.1$, $q_0 = 0.4$,  
$X = 437.6$, $c = 84.0$, and (a) $X_0 = 0.303$, (b) $X_0 = 0.044$.  In (d), we plot a numerical sampling (via Hill's method) of the unstable spectrum
corresponding to the spectral problem \ref{eig_inf2}, 
corresponding to the case $\alpha>-2$, for a representative periodic stationary solution of \eqref{profile-ham} .
}
\label{comparison}\end{figure}



Concerning the case $\alpha=-2$, the profile equation \eqref{profile_crit} was solved by using the 
Matlab functions bvp5c and bvp6c, where we treated $c_0$ as a free parameter in the bvp solver
and used numerical continuation to solve the profile as $k_0$ was varied.  Our numerical investigation of the associated spectral problem
\eqref{eval} covered $q_0\in[1.6,2.2]$ for $\nu = 2$ and $q_0\in[1.2,2.7]\cup [0.3,0.45]$ for $\nu = 0.1$ with 
step size $0.1$ in $q_0$ and varying, but smaller, steps in $k_0$ in the region where we could solve profiles numerically. 
Again, we found that all these waves have unstable spectra; see Figure \ref{comparison}(c) for an example. In Figure \ref{comparison}, we also plot the spectrum as determined using the \qut{$\alpha = -2$} scaling given in Section \ref{s:scaling} and show that in this scaling the spectrum is unstable for large $F$ as well as in the $F = \infty$ eigenvalue problem \eqref{eval}. Instability in the $F=\infty$ case was thus confirmed by multiple numerical checks, providing strong evidence of instability.

In conclusion, our numerical calculations strongly indicate that for $F$ sufficiently large, spectrally stable periodic traveling wave solutions of the St. Venant system \eqref{swl} do not exist. This justifies Numerical Observation 1 described in \ref{s:intro_infty}.

\subsection{Numerical investigation for intermediate $F$}\label{numerical:intermediate}

In order to better understand the stability of periodic traveling wave solutions of \eqref{swl} away from the distinguished limits $F\to 2^+$ and $F\to\infty$, we also carried out a numerical investigation for some``intermediate" Froude numbers.

To solve the profile equation numerically, we expand \eqref{profile} to obtain the profile ODE equation,
\begin{equation}
\tau'' = \left(\frac{-\tau^2}{c\nu}\right)\left(c^2\tau' - \frac{\tau'}{F^2\tau^3} - 1+\tau(q-c\tau)^2-2c\nu(\tau')^2/\tau^3\right),
\label{profile_expanded}
\end{equation}
and then proceed the same way as described 
in the studies of the limiting $F\to\infty$ systems described above,
this time using a relative error tolerance between $10^{-6}$ and $10^{-10}$ in the bvp solver.
Checking the slope condition \eqref{e:Lslope} across numerically determined profiles, we found that it
is satisfied for profiles with $F\lessapprox 3.5$ but violated for those with
larger $F$.

To study the spectrum, we worked with the original (unrescaled) 
eigenvalue system
\eqref{spec1},
associating parameters with those under the rescaling \eqref{scales}. We found that the results for Hill's method were more accurate using these original coordinates as opposed to the rescaled coordinates \eqref{scales}, even for large $F$. For $q_0 = 0.4$, $k_0$ ranging in the parameter space where profiles could be solved, and for $F$ ranging from 10 to either 20 or 30 by 1, we examined, via Hill's method and the Taylor expansion of the Evans function, the cases $\alpha = -0.7, -1, -1.4, -1.5, -1.6, -1.7, -1.8, -1.9, -2$. We used 201-3000 Fourier modes and 21-31 Floquet parameters in Hill's method and 33-201 Chebyshev nodes for the integral in the Taylor expansion. For each value of $\alpha$, we found that a lower stability boundary curve and an upper high-frequency instability curve meet at some value $F^*(\alpha)$ after which no waves are stable. See Figures \ref{fig476}. 
In addition, we find that the upper and lower stability boundaries appear to have a linear relationship of the form $\log(X/\nu) = b_1\log(F)+b_2\log(q)+ b_3$; see Figure \ref{fig462}. We used a combination of Hill's method and Evans function computations
to determine stability or instability. However, as the period $X$ increases, a small loop of spectra parameterized by the whole interval of Floquet parameters 
$\xi \in [-\pi/X,\pi/X)$ 
shrinks until eventually neither the Evans function nor Hill's method can definitively determine small frequency stability at which point we rely on Hill's method for the overall behavior of the spectrum; 
in this region, 
analytic verification of the stability regions as $F\to \infty$ would be beneficial. However, for the lower intermediate $F$ region where the period is smaller, the stability picture appears clear. In particular, for $\alpha = -2$ our numerical study, though not a numerical proof, strongly suggest the stability region shown in Figure \ref{fig476} (a). As $F$ continues to increase, it becomes easier again to compute the spectrum with Hill's method as the spectrum approaches that of the limiting system given in equation \eqref{profile_crit}; see Figure \ref{comparison} (a)-(c) for an illustration.  For $F = 38$, we examined the full set of periods corresponding to those examined in Figure \ref{fig476}, and found no stability region,
confirming that by this point the upper and lower stability boundaries have met.


\begin{figure}[htbp]
 \begin{center}
$
\begin{array}{lcr}
(a) \includegraphics[scale=0.35]{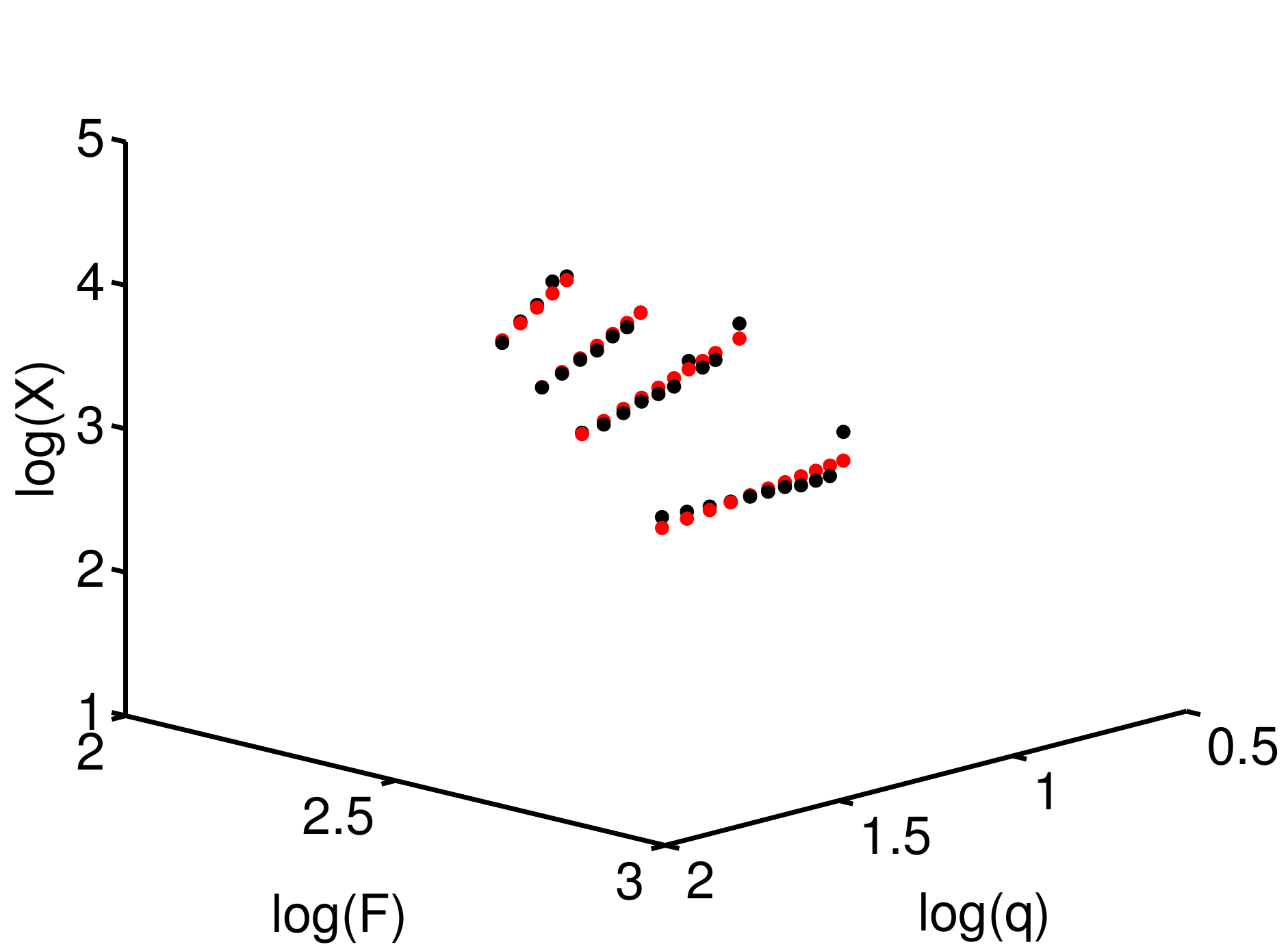}&\quad&(b) \includegraphics[scale=0.35]{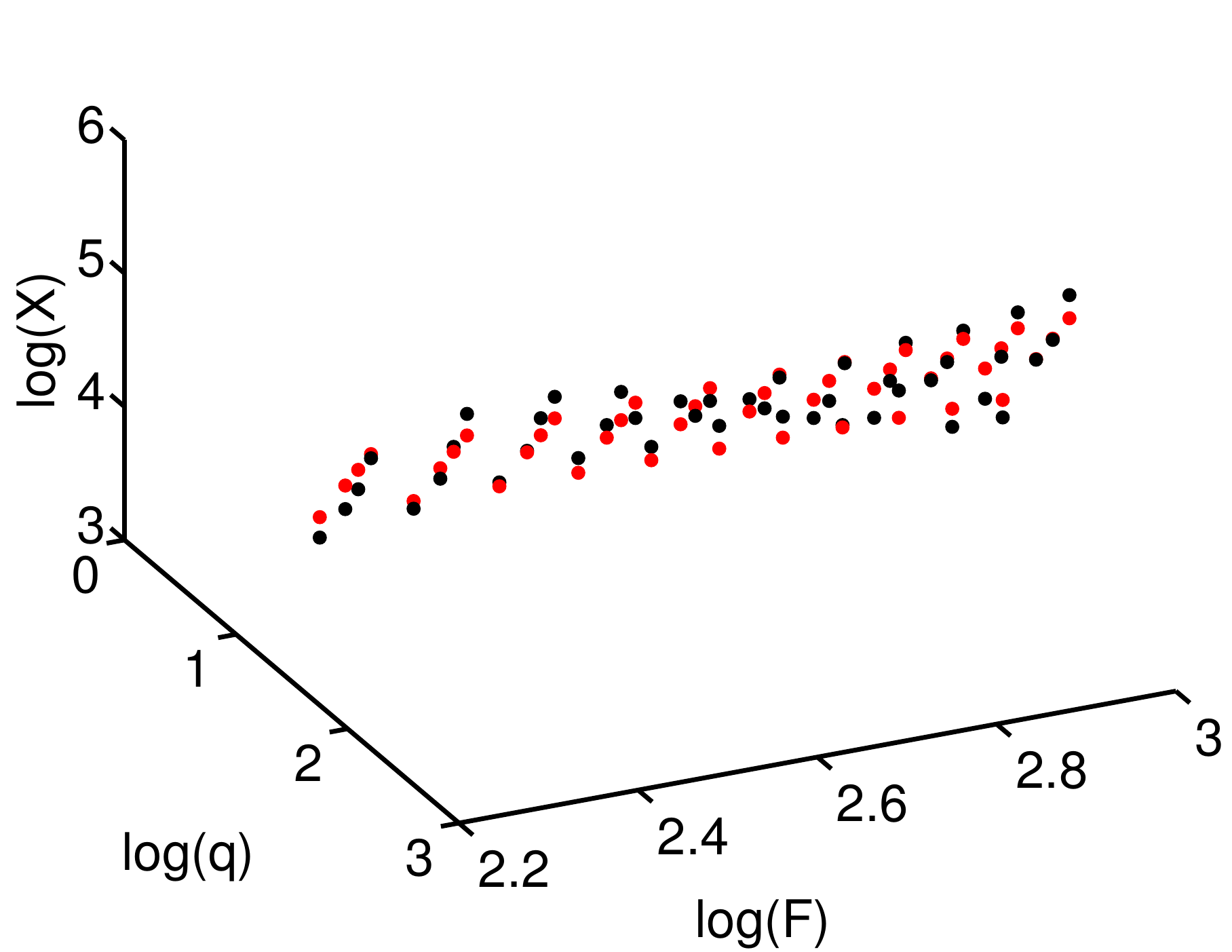}
\end{array}
$
\end{center}
\caption{In (a) and (b), black dots mark the computed boundary and pale dots (red in color plates) mark the best least curve fit. (a) Lower stability boundary. We have $\log(X) \approx = b_1\log(F)+b_2\log(q) + b_3$  where $b_1 = -0.692$, $b_2 = 3.46$, and $b_3 = 0.3$. Here $\alpha = -1.6, -1.8, -1.9, -2$. The maximum error is 0.200 and the maximum relative error is 0.056.  The average relative error is $0.012$ and the average absolute error is $0.041$. (b) High frequency stability boundary. We have $\log(X) \approx b_1\log(F)+b_2\log(q) + b_3$  where $b_1 = -0.791$, $b_2 = 1.73$, and $b_3 = 3.92$. Here $\alpha = -1.6, -1.8, -1.9, -2$.  The maximum error is 0.228 and the maximum relative error is 0.052. The average relative error is $0.024$ and the average absolute error is $0.103$. }
\label{fig462}\end{figure}

\begin{figure}[htbp]
 \begin{center}
$
\begin{array}{ll}
 (a) \includegraphics[scale=0.25]{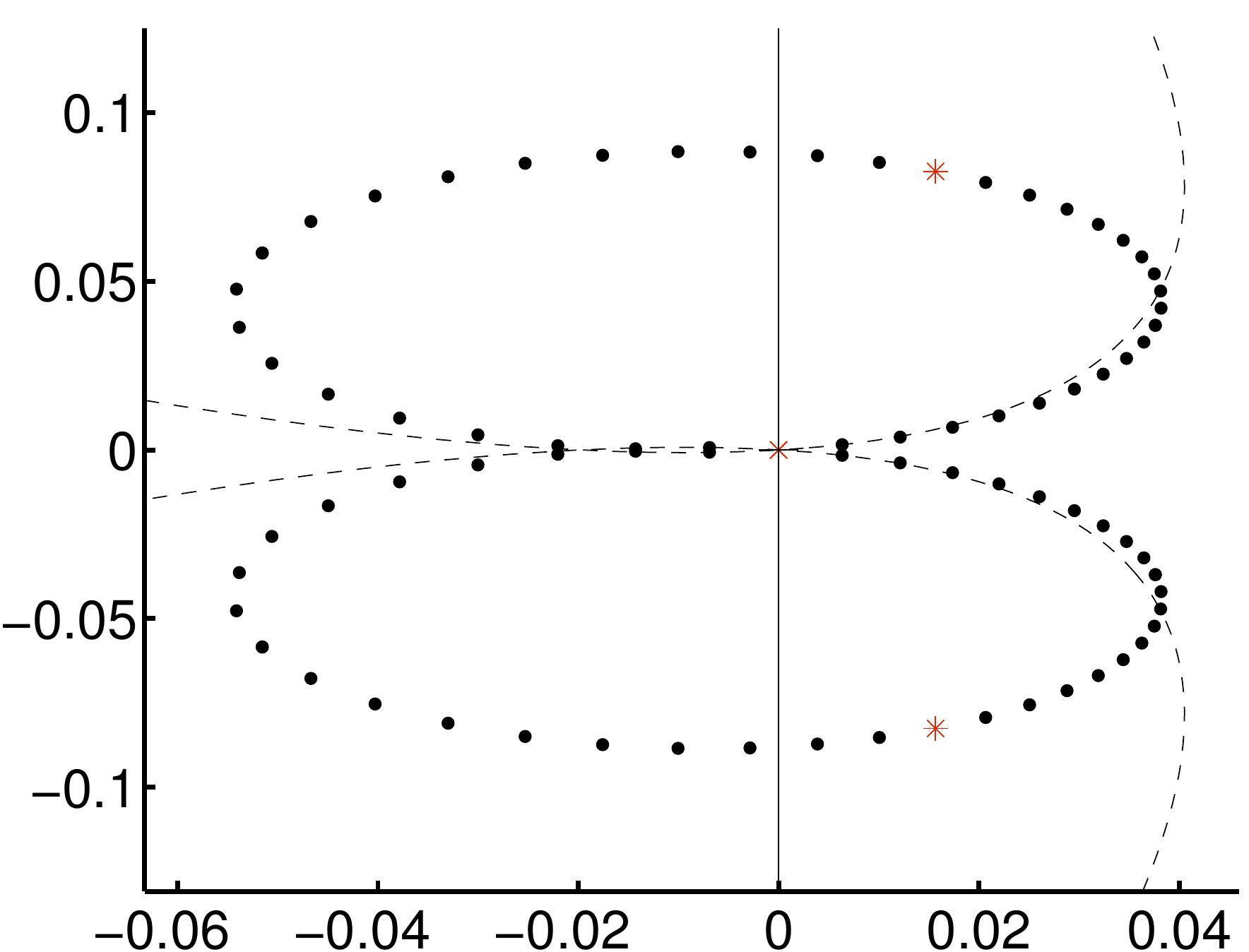} & (b) \includegraphics[scale=0.25]{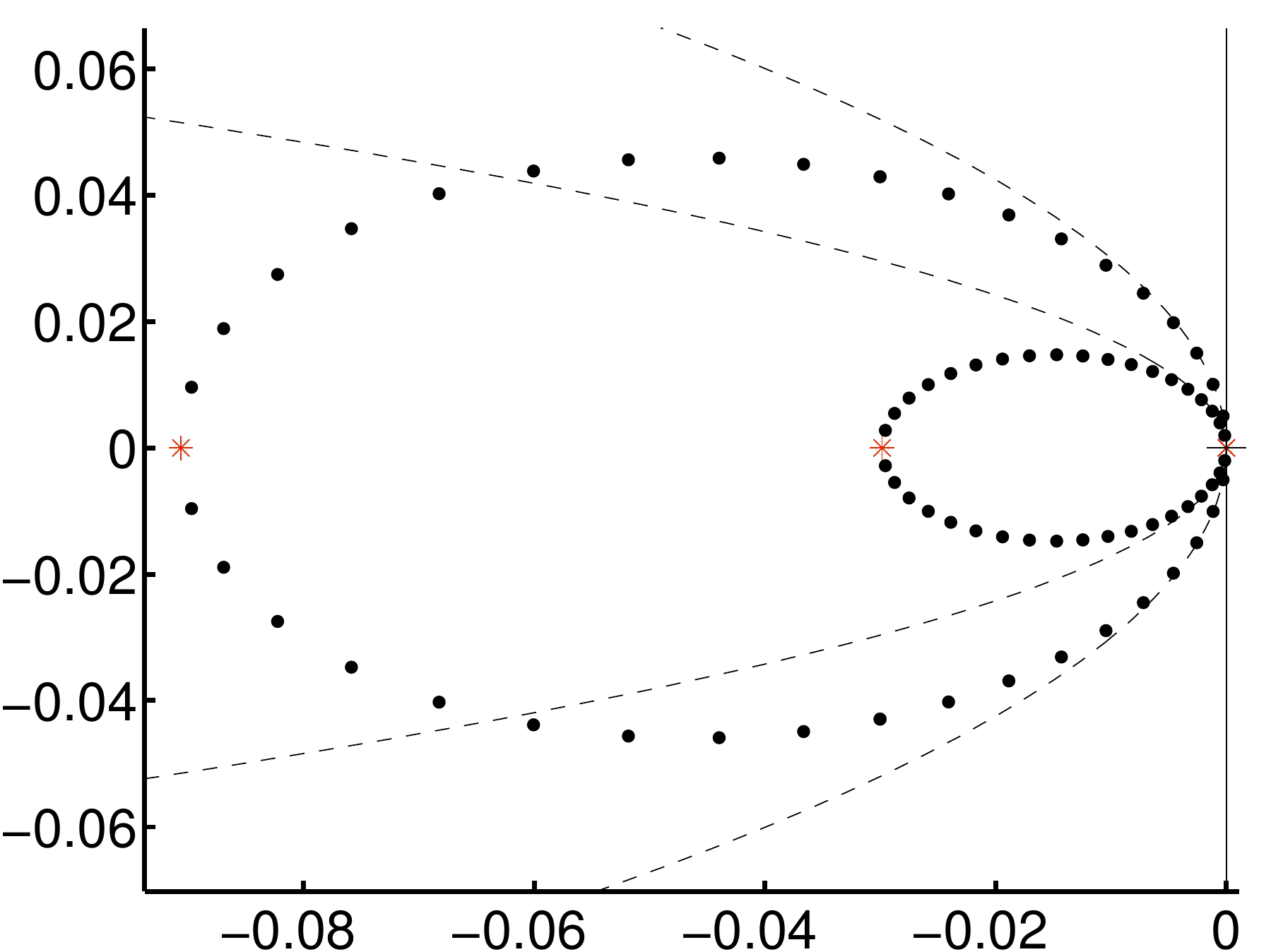} \\
 (c) \includegraphics[scale=0.25]{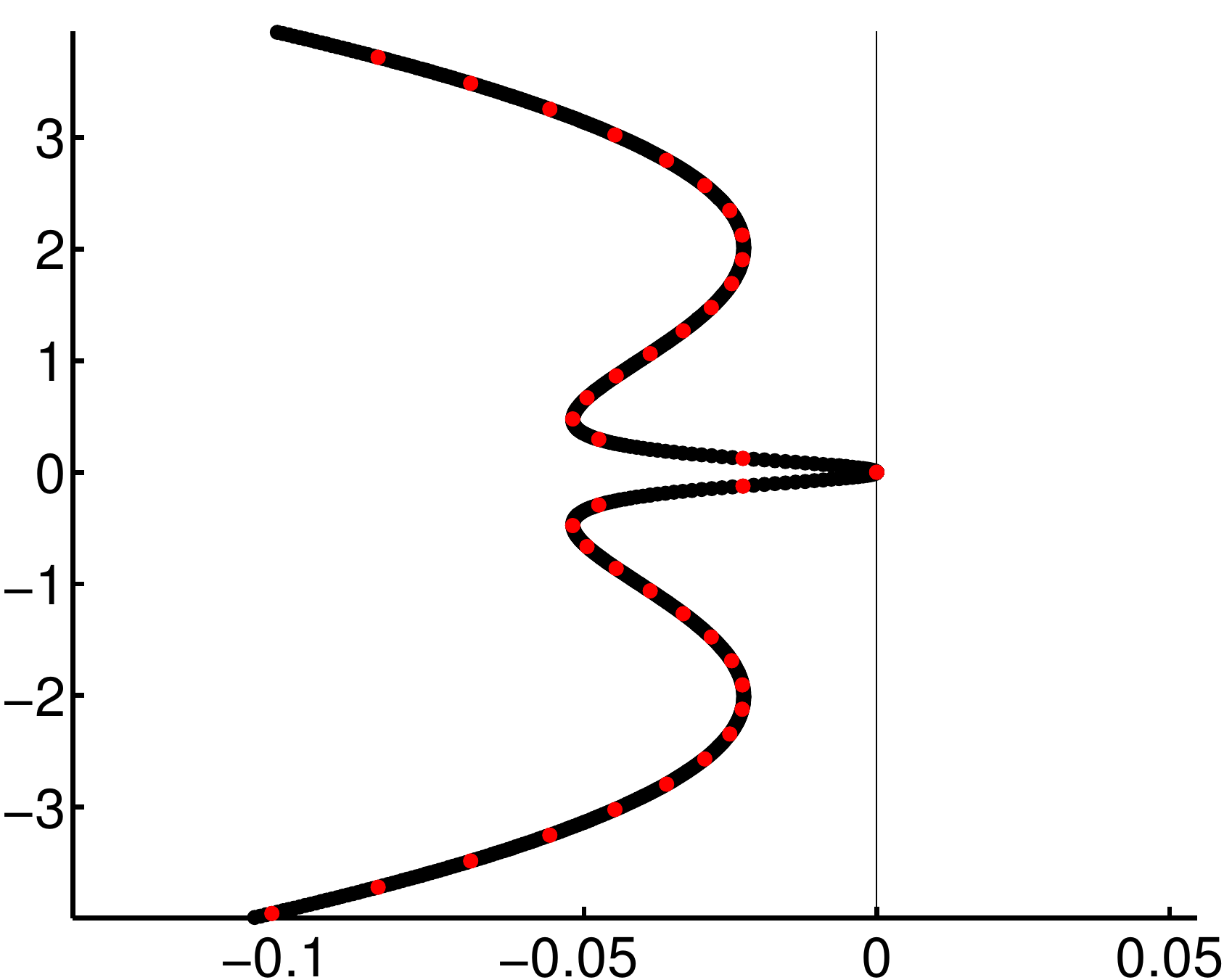} & (d) \includegraphics[scale=0.25]{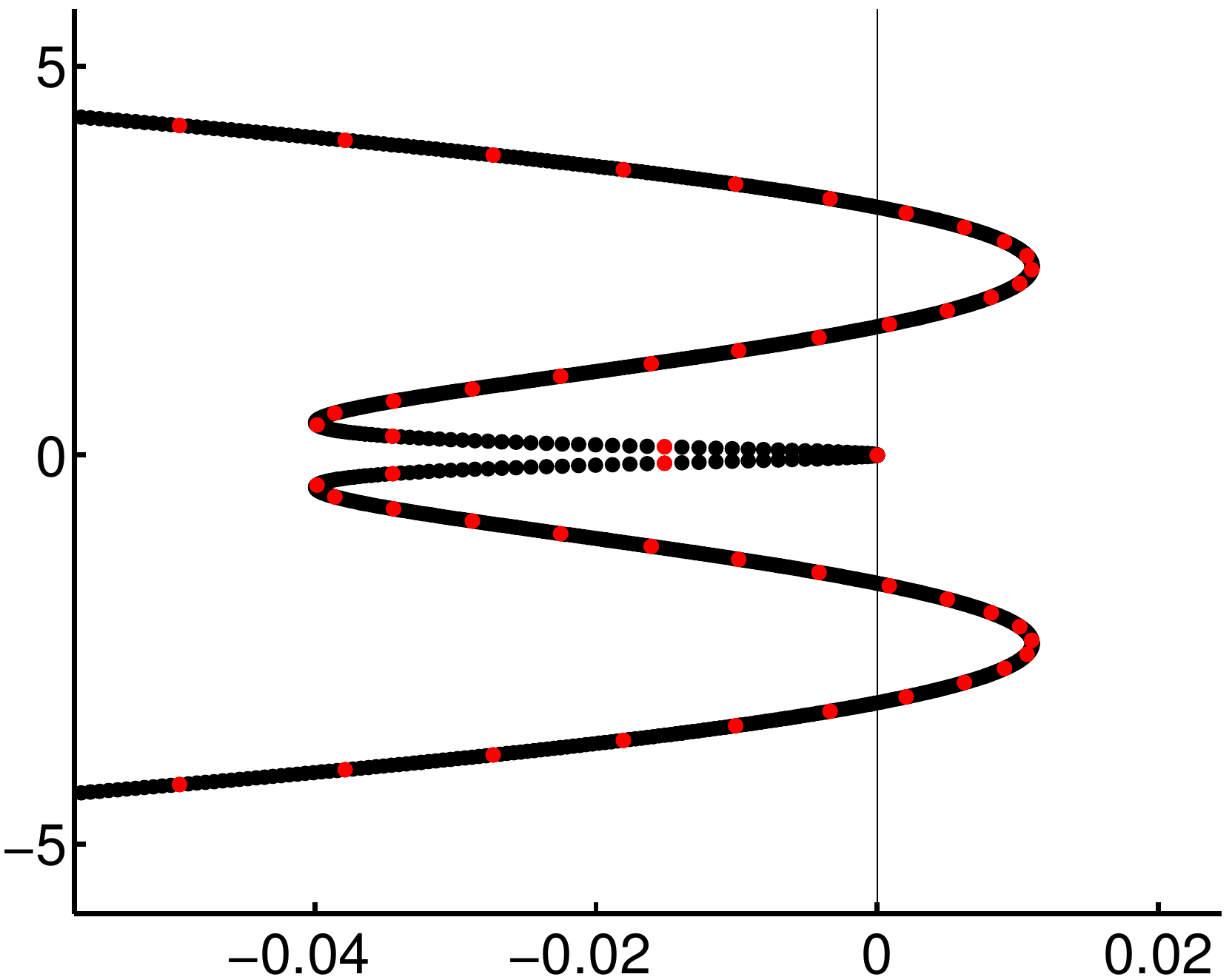}\\

\end{array}
$
\end{center}
\caption{Plot of spectra for the intermediate eigenvalue problem. We plot the 
	approximations returned by Hill's method with black dots and 
	those 
	obtained by Taylor expanding the Evans function with 
	light dashes.
	We plot the spectra corresponding 
	to
	Floquet parameter zero with a pale star (red in color plates). 
	The solid line indicates the imaginary axis. In all cases, $\alpha = -2$, $\nu = 0.1$, and  $q_0 = 0.4$. 
Other parameters 
are 
	(a) $F = 6$, $ X = 7.83$, (b)  $F = 6$, $X = 8.78$, (c) $F = 10$, $X = 76.9$, (d) $F = 10$, $X = 90.9$.
}
\label{showspec}
 \end{figure}

\medskip
{\bf Acknowledgement}:
The numerical stability computations in this paper were carried out
using the STABLAB package developed by Jeffrey Humpherys and
the first and last authors; see \cite{BHZ} for documentation. 
We gratefully acknowledge his contribution. We thank also
Indiana University Information Technology Service for
the use of the Quarry computer with which some of
our computations were carried out. Authors in various combinations also acknowledge the kind hospitality 
of \'ENS Paris, Indiana University, INSA Toulouse and Universit\'e Lyon 1
that hosted them during parts of the preparation of the present paper.
Finally, we thank Sergey Gavrilyuk for a helpful exchange
regarding the experimental literature.

\appendix

\section{Notation for High Frequency Bounds}\label{s:HFB}

The proof of Lemma \ref{lemma1} in Section \ref{s:cvgzero} relied heavily on computations previously carried out in detail in Section 4.1 
of 
\cite{BJRZ}. There, the authors were concerned with the stability analysis of traveling \emph{solitary} wave solutions of the viscous St. Venant equation \eqref{swl}, i.e. those traveling wave solutions that 
decay
exponentially fast to zero as $x-\bar ct\to\pm\infty$.  By a straightforward 
adaptation of this analysis to the periodic case,\footnote{
Namely, using formally identical coordinate changes depending on the profile and
its derivatives.}
it follows that, for each $k\in\mathcal{P}$, there exists an $X_\delta$-periodic change of variables $W(\cdot)=P(\cdot;\lambda,\delta)Z(\cdot)$ that transforms the spectral problem \eqref{evans1} into an equivalent system of the form $W'(x)=\left(D(x,\lambda)+\Theta(x,\lambda)\right)W(x)$, 
supplemented with the boundary conditions $W(X_\delta)=e^{i\xi}W(0)$ for some $\xi\in[-\pi/X_\delta,\pi/X_\delta)$, where the $3\times 3$ matrix-valued $D$ is defined via
	$
D(\cdot,\lambda)=\diag\left(\frac{\lambda}{\bar c}+\theta_0+\frac{\theta_1}{\lambda},~~\bar\tau\sqrt\frac{\lambda}{\nu},~~-\bar\tau\sqrt\frac{\lambda}{\nu}\right),
$
where $\theta_0=\frac{\bar\alpha \bar\tau^2}{\bar c\nu}$ and
$\theta_1=-\frac{\bar\tau^2(q-\bar c\bar\tau)^2}{\nu}-\frac{\bar c\bar\alpha}{\sqrt\nu}\frac{\bar\tau^3}{\nu^{3/2}}$
with $\bar \alpha:=\bar\tau^{-3}(F^{-2}+2\bar c\nu\bar\tau')$.
Moreover, the $3\times 3$ matrix $\Theta(x,\lambda)$ has the block structure
$
\Theta=\left(\begin{array}{cc}
	\Theta_{++} & \Theta_{+-}\\
	\Theta_{-+} & \Theta_{--}
	\end{array}\right),
$
where $\Theta_{++}$ is a $2\times 2$ matrix.  The individual blocks of the matrix $\Theta$ can be expanded as cubic polynomials
in $\lambda^{-1/2}$ with matrix valued coefficients.  More precisely, they expand as
\[
\left\{\begin{aligned}
\Theta_{++}(\cdot,\lambda)&=\Theta_{++}^0 + \lambda^{-1/2}\Theta_{++}^1 + \lambda^{-1}\Theta_{++}^2 +\lambda^{-3/2}\Theta_{++}^3\\
\Theta_{+-}(\cdot,\lambda)&=\Theta_{+-}^0 + \lambda^{-1/2}\Theta_{+-}^1 + \lambda^{-1}\Theta_{+-}^2 +\lambda^{-3/2}\Theta_{+-}^3\\
\Theta_{-+}(\cdot,\lambda)&=\Theta_{-+}^0 + \lambda^{-1/2}\Theta_{-+}^1 + \lambda^{-1}\Theta_{-+}^2 +\lambda^{-3/2}\Theta_{-+}^3\\
\Theta_{--}(\cdot,\lambda)&=\Theta_{--}^0 + \lambda^{-1/2}\Theta_{--}^1 + \lambda^{-1}\Theta_{--}^2 +\lambda^{-3/2}\Theta_{--}^3,
\end{aligned}\right.
\]
with
\begin{align*}
\Theta^0_{++}&=  \bp 0 &-\frac{\bar\tau^2}{\nu}\\
-\frac{\bar\alpha\bar\tau^2}{2\nu}&\frac{\bar\tau'}{\bar\tau}-\frac{\bar c\bar\tau^2}{2\nu}\ep,
\qquad \Theta^1_{+-}= \bp
-\frac{2\bar\tau'}{\sqrt{\nu}}+\frac{\bar\tau^3}{\nu^{3/2}}(\frac{\bar\alpha}{\bar c}+\bar c)\\
-\frac{\bar\tau^{2}\left(q-\bar c\bar\tau\right)}{\sqrt{\nu}}-\frac{\bar\alpha}{2}\frac{\bar\tau^3}{\nu^{3/2}}
\ep,\\
\Theta^1_{++}&= \bp 0&
-\frac{2\bar\tau'}{\sqrt{\nu}}+\frac{\bar\tau^3}{\nu^{3/2}}(\frac{\bar\alpha}{\bar c}+\bar c)\\
\frac{\bar\tau}{2\sqrt{\nu}}\left(\left(q-\bar c\bar\tau\right)^2+c\bar\alpha\frac{\bar\tau^2}{\nu}\right)&
\frac{\bar\tau^{2}\left(q-\bar c\bar\tau\right)}{\sqrt{\nu}}+\frac{\bar\alpha}{2}\frac{\bar\tau^3}{\nu^{3/2}}
\ep,\\
\Theta^2_{++}&=\bp  0 &
-\frac{2\bar\tau^{3}\left(q-\bar c\bar\tau\right)}{\nu}\\
0&\frac{\bar\tau^{2}\left(q-\bar c\bar\tau\right)^2}{2\nu}+\frac{\bar c\bar\alpha}{2\sqrt{\nu}}\frac{\bar\tau^3}{\nu^{3/2}}
\ep,\\
\Theta^3_{++}&=\bp 0 &
-\frac{\bar\tau^{3}\left(q-\bar c\bar\tau\right)^2}{\nu^{3/2}}-\frac{\bar c\bar\alpha}{\sqrt{\nu}}\frac{\bar\tau^4}{\nu^2}\\
0&0\ep,
\qquad \Theta^0_{+-}= \bp \frac{\bar\tau^2}{\nu}\\
\frac{\bar\tau'}{2\bar\tau}-\frac{\bar c}{2}\frac{\bar\tau^2}{\nu}\ep,\\
\Theta^2_{+-}&= \bp
\frac{2\bar\tau^{3}\left(q-\bar c\bar\tau\right)}{\nu}\\
\frac{\bar\tau^{2}\left(q-\bar c\bar\tau\right)^2}{2\nu}+\frac{\bar c\bar\alpha}{2\sqrt{\nu}}\frac{\bar\tau^3}{\nu^{3/2}}\ep,
\qquad \Theta^3_{+-}= \bp
-\frac{\bar\tau^{3}\left(q-\bar c\bar\tau\right)^2}{\nu^{3/2}}-\frac{\bar c\bar\alpha}{\sqrt{\nu}}\frac{\bar\tau^4}{\nu^2}\\
0\ep,\\
\Theta^0_{-+}&=  \bp \frac{\bar\alpha\bar\tau^2}{2\nu}&
\frac{\bar\tau'}{\bar\tau}-\frac{\bar c\bar\tau^2}{2\nu}  \ep,
\qquad \Theta^2_{-+}= \bp
0&\frac{\bar\tau^{2}\left(q-\bar c\bar\tau\right)^2}{2\nu}+\frac{\bar c\bar\alpha}{2\sqrt{\nu}}\frac{\bar\tau^3}{\nu^{3/2}}
\ep,\\
\Theta^1_{-+}&=  \bp
\frac{\bar\tau}{2\sqrt{\nu}}\left(\left(q-\bar c\bar\tau\right)^2+\bar c\bar\alpha\frac{\bar\tau^2}{\nu}\right)&
\frac{\bar\tau^{2}\left(q-\bar c\bar\tau\right)}{\sqrt{\nu}}+\frac{\bar\alpha}{2}\frac{\bar\tau^3}{\nu^{3/2}}\ep,
\qquad \Theta^0_{--}=\bp\frac{\bar\tau'}{\bar\tau}-\frac{\bar c\bar\tau^2}{2\nu}
\ep,\\
\Theta^1_{--}&= \bp-\frac{\bar\tau^{2}\left(q-\bar c\bar\tau\right)^2}{2\nu}+\frac{\bar c\bar\alpha}{2\sqrt{\nu}}\frac{\bar\tau^3}{\nu^{3/2}}\ep,
\qquad \Theta^2_{--}=\bp\frac{\bar\tau^{2}\left(q-\bar c\bar\tau\right)^2}{2\nu}+\frac{\bar c\bar\alpha}{2\sqrt{\nu}}\frac{\bar\tau^3}{\nu^{3/2}}\ep.
\end{align*}
%
From this, the matrices $N_{D,D},N_{H,D},N_{D,H}$ in \eqref{Nmatrix} are obtained through the identification
$$
\left(
\begin{array}{cc}
\Theta_{++}&\Theta_{+-}\\
\Theta_{-+}&\Theta_{--}\\
\end{array}
\right)
\ =\ \left(
\begin{array}{cc}
 0 &N_{H,D}\\
N_{D,H}&N_{D,D}\\
\end{array}
\right),
$$
where $N_{DD}$ is a $2\times 2$ matrix.

\section{Computational Methods}\label{s:compmethod}

For completeness, we very briefly describe the computational methods utilized in our investigations reported in Section \ref{s:numinf}  and \ref{numerical:intermediate} above.
For more details, the interested reader is referred to \cite{BJNRZ2} where an analogous numerical study is performed on the generalized Kuramoto-Sivashinsky equation.

%
%

\subsection{Hill's method} To determine the global picture of spectrum of a linear 
$X$-periodic operator $L$, 
we use Hill's method. The linear operator $L$ takes the form $L_{j,k}= \sum_{q=1}^{m_{jk}}f_{j,k,q}(x)\frac{\partial^q}{\partial x^q}$ where the $f_{j,k,q}(x)$ are $X$ periodic. Following \cite{DKCK}, we represent the coefficient functions $f_{j,k,q}(x)$ as a Fourier series $f_{j,k,q}(x)=\sum_{j=-\infty}^{\infty} \hat \phi_{j,k,q}e^{i2\pi jx/X}$. We use Matlab's fast Fourier Transform to determine the coefficients $\hat \phi_{j,k,q}$. The generalized eigenfunctions are represented 
as\footnote{Mark that in the standard implementation of Hill's method a periodic wave is treated as a periodic function of twice its fundamental period. As recalled in \cite[Section~3.1, p.67]{R}, this is originally motivated by the fact that in applications to self-adjoint second-order scalar operators, the Floquet-zero spectrum will then provide edges of spectral bands.} 
$v(x) = e^{i\xi x}\sum_{j=-\infty}^{\infty} \hat v_j e^{i\pi jx/X}$, where $\xi \in [-\pi/2X,\pi/2X)$ is the Floquet exponent. Upon substituting the Fourier expansions into the eigenvalue problem, fixing $\xi$, and equating the coefficients of the resulting Fourier series, we arrive at the eigenvalue problem $\hat L^{\xi} \hat v = \lambda \hat v$ where $\hat L^{\xi}$ is an infinite dimensional matrix. The spectrum of the operator $L$ is given by 
$\sigma(L) = \bigcup_{\xi \in[-\pi/2X,\pi/2X)} \sigma(L_{\xi})$. 
Truncating the Fourier series at $N$ terms leads to a finite dimensional eigenvalue problem $L_N^{\xi}\hat v = \lambda \hat v$. The matrix $L_N^{\xi}$ is of the form $M_2^{-1}M_1$ where $M_1\hat v = \lambda M_2\hat v$ is the original eigenvalue problem. 
Typically $M_2$ is the identity, but in \eqref{eig_inf2}, $M_2$ is diagonal with $j$th diagonal entry $i(j+\xi)$, hence we avoid $\xi = 0$ in that case so that $M_2$ is invertible. We compute $\sigma(L_N^{\xi})$ on a mesh to approximate the spectral curves of $L$. 
For our numerical studies, we used the implementation of Hill's method built into STABLAB \cite{BHZ}. 
For discussion of Hill's method and its convergence, see \cite{CDe, DK,JZ4}.

\subsection{Evans function} Our results for Hill's method are augmented by use of the Evans function.
To this end, note all the spectral problems we study, such as theone given in \eqref{eig_inf2}, may be written as a first order
system of the form $W'(x)=\mathbb{A}(x;\lambda)W(x)$ and that, further, $\lambda\in\CM$ belongs to the essential
spectrum of the associated linearized operator $L$ if and only if this 1st order system admits a non-trivial solution
satisfying
\[
Y(x+X;\lambda)=e^{i\xi X}Y(x;\lambda),\quad\forall x\in\RM
\]
where here $X$ denotes the period of the coefficients of $L$.
Following Gardner \cite{G}, the Evans function is defined as 
$D(\lambda,\xi):= \det\left(\Psi(X,\lambda)-e^{i\xi X} \right)$ 
where the matrix $\Psi(x,\lambda)$ satisfies $\d_x\Psi(x,\lambda) = \mathbb{A}\Psi(x,\lambda)$ and 
$\Psi(0,\lambda) = \Id$. 
%
By construction then, the roots of $D(\cdot;\xi)$
%
agree in location and algebraic multiplicity with the eigenvalues of the associated $\xi$-dependent spectral problem. 
Unfortunately, the Evans function as described here is poorly conditioned for numerical computation.
To remedy this, as in \cite{BJNRZ2}, we use the observation of Gardner \cite{G} that
$$
D(\lambda,\xi):=\det(\Psi(X)-e^{i\xi X}\Id)=\det \bp \Psi(X)& e^{i\xi X}\Id\\ \Id & \Id \ep,
$$
to express the Evans function as an exterior product of solutions of
$$
\bp Y\\\alpha\ep'= \bp \mathbb{A}(\cdot,\lambda)Y \\0\ep,
$$
with data $( \Id ,\Id)^T$ at $x=0$ and $( e^{i\xi X}\Id,\Id)^T$
at $x=X$; for details see \cite{BJNRZ2}. We then use the {\it polar coordinate method} of \cite{HuZ} to evolve the solutions. This algorithm is numerically well-conditioned \cite{Z1}. All computations were carried out using STABLAB \cite{BHZ}. 

As mentioned above, Hill's method is ideal for obtaining a global picture of the spectrum and the Evans function can be evaluated on contours and the winding number evaluated to determine the presence of zeros. However, neither method determines definitively whether unstable spectra of arbitrarily small size exist, due to numerical error. In particular, such methods can not be used to determine the spectrum of the associated linearized operators in a sufficiently small neighborhood of the origin, i.e. they can not resolve the modulational instability problem.    A strategy for rigorously computing \emph{stability}, which we utilize in our intermediate $F$ numerical studies
reported in \ref{numerical:intermediate} above, involves a Taylor expansion of the Evans function as we now briefly describe; see \cite{BJNRZ2} for details. 

Due to the presence of a conservation law in the governing system (see \cite{Se,JZN,JNRZ2,R}) the Evans function has a double root at the origin when $\xi = 0$.  As such, the Taylor expansion of the Evans function about the origin $(\lambda,\xi)=(0,0)$  takes the form
$D(\lambda,\xi)=c_{2,0}\lambda^2+c_{1,1}\lambda\xi+c_{0,2}\xi^2+c_{3,0}\lambda^3+c_{2,1}\lambda^2\xi+c_{1,2}\lambda\xi^2+c_{0,3}\xi^3+O(|\lambda|^4+|\xi|^4)$
where the coefficients $c_{k,j}$ may be determined via Cauchy's integral formula,
\be\label{coefint}
c_{k,j}=-\frac{1}{4\pi^2}\oint_{\partial B(0,r)}\oint_{\partial B(0,r)}D(\lambda,\xi)\lambda^{-k-1}\xi^{-j-1}d\lambda~d\xi
\ee
with $r>0$ sufficiently small. Setting
$\alpha_j=\frac{-c_{1,1}+(-1)^{j+1}\sqrt{c_{1,1}^2-4c_{2,0}c_{0,2}}}{2c_{2,0}}$,
$\beta_j=-\frac{c_{3,0}\alpha_j^3+c_{2,1}\alpha_j^2 +c_{1,2}\alpha_j+c_{0,3}}
{2c_{2,0}\alpha_j+c_{1,1}}$,
one readily checks that the roots of the Evans function near $(\lambda,\xi)=(0,0)$ may be continued for $|\xi|\ll 1$ as
\[
\lambda_j(\xi)= \alpha_j \xi +\beta_j\xi^2+\frac{\xi^3}{2}\int_0^1(1-s)^2\lambda_j'''(s\xi)ds.
\]
The spectral curves at the origin are thus approximated by $\alpha_j \xi + \beta_j \xi^2$ with spectral stability corresponding to the case $\alpha_j\in\RM i$ and $\Re(\beta_j)<0$; see \cite{BJNRZ2} for details.

In practice, to compute the Taylor expansion coefficients, rather than compute the Evans function on the contour integral in the variable $\lambda$ for fixed $\xi$ given in \eqref{coefint}, we interpolate with $\sum_{k=0}^K e^{i k\xi x}$ ($K =3$ is the largest power of $e^{i\xi x}$ that appears) and then use the Taylor expansion $e^{ik\xi x} = 1+ (ik\xi x)+(ik\xi x)^2/2 + ...$ yielding $D(\lambda,\xi) = \sum_{k=0}^{\infty} c_k \xi^k$, from which the contour integral can be determined simply by reading off the corresponding coefficient. Calling the quantity just determined $\tilde D$, we see
\begin{equation}
\frac{1}{2\pi i} \oint_{|\lambda|= R_1} \frac{\tilde D(\lambda)}{\lambda^{r+1}} d\lambda= \frac{1}{2\pi} \int_{-\pi}^{\pi} \frac{\tilde D(Re^{i\theta})}{R^re^{ir\theta}}d\theta
= \frac{1}{2R^r}\int_{-1}^{1} \tilde D(Re^{i\pi \theta}) e^{-ir\pi \theta} d\theta,
\notag
\end{equation}
which we compute by approximating the integrand with Chebyshev interpolation and integrating. 


\section{Computational effort}\label{s:comp}
\subsection{Computational environment}

In carrying out our numerical investigations, we used a MacBook laptop with 2GB memory and a duo core Intel processor with 2GHz processing speed, a 2009 Mac Pro with 16GB memory and two quad-core intel processors with 2.26 GHz processing speed, and Quarry, a supercomputer at Indiana University consisting of 140 IBM HS21 Blade servers with two 2GH quad-core Intel Zeon 5335 processors per node and delivering 8.96 teraflops processing speed.
All computations were done using Matlab and the Matlab based stability package STABLAB.

\subsection{Computational time}

We begin by providing computational statistics for the representative parameter set $\alpha = -2$, $q_0 = 0.4$, $F = 10$, and $X = 50$. We compute the Evans function, $D(\lambda,\xi)$, on a semicircular contour, $\Omega$, of radius $R = 0.2$ with 42 evenly spaced Floquet parameters $\xi\in[-\pi/X,-\pi/(10X)]\cup [\pi/(10X),\pi/X]$. We require the relative error between contour points of the image contour, $Y_{\xi}$, $D(\cdot,\xi): \Omega\to Y_{\xi}$, not exceed 0.2 so that Rouch\'e's theorem implies the winding number of $Y_{\xi}$ corresponds to the number of roots of $D(\cdot,\xi)$ in $\Omega$.  We use 277 points in $\Omega$, chosen adaptively, to achieve 0.2 relative error in each $Y_{\xi}$ at a computational cost of 56.0 seconds using 8 Matlab workers on Quarry to determine the winding number is zero. Computing the Taylor expansion of the Evans function at the origin requires 61.7 seconds on Quarry, while computing the spectrum via Hill's method using 603 Fourier modes and 21 Floquet parameters comes at a computational cost of 143 seconds. 

As the period $X$ increases or as $F$ increases, the number of Fourier modes needed in Hill's method increases. Using 600 Fourier modes typically takes around 3 minutes on the Mac Pro, while using 1600 Fourier modes takes about 77 minutes, and using 3000 Fourier modes requires approximately 8 hours. 

In creating Figure \ref{fig476} (a) it took 4.36 days of computation time to evaluate the Taylor coefficients and 34.5 days to compute the spectrum using Hill's method, while the Evans function required an estimated 58 hours. 
A typical profile requires only a few seconds to compute, but we must use continuation whereby we use a nearby profile as an initial guess in the boundary value solver, so that computing the profiles also required a great computational effort. Overall, taking into account the use of parallel computing and all values of $\alpha$ investigated, we estimate that total computational time for the project exceeds a year.

\bibliographystyle{alpha}
\bibliography{Ref_fto2} 

\end{document}